\documentclass[10pt,oneside]{article}
\usepackage[T1]{fontenc}
\usepackage[english]{babel}

\usepackage{graphicx,accents}

\newcommand{\cev}[1]{\reflectbox{\ensuremath{\vec{\reflectbox{\ensuremath{#1}}}}}}

\usepackage[cal=cm]{mathalfa}

\usepackage[T1]{fontenc}
\usepackage{titlesec}

\usepackage{hhline}

\titleformat
{\chapter} 
[display] 
{\bfseries\Large\itshape} 
{Story No. \ \thechapter} 
{0.5ex} 
{
    \rule{\textwidth}{1pt}
    \vspace{1ex}
    \centering
} 
[
\vspace{-0.5ex}%
\rule{\textwidth}{0.3pt}
] 

\titleformat{\section}[wrap]
{\normalfont\bfseries}
{\thesection.}{0.5em}{}

\titlespacing{\section}{12pc}{1.5ex plus .1ex minus .2ex}{1pc}

\usepackage{tocloft} 

\usepackage{titlesec}
\titleformat{\subsection}[runin]
       {\normalfont\bfseries}
       {\thesubsection}
       {0.5em}
       {}
       [.]

\makeatletter
\@ifundefined{dddot}{}{}
\usepackage{mathdots}  
\makeatother

\makeatletter
\@ifundefined{ddddot}{}{}
\usepackage{yhmath}
\makeatother

\usepackage{geometry} 
\usepackage{xcolor} 
\usepackage{tikz} 
\usepackage[draft=false]{hyperref} 
\usepackage{mathtools}
\usepackage{amsmath,amsthm,amssymb}
\usepackage[all,cmtip]{xy}

\usepackage{amsfonts,extarrows}
\usepackage{enumerate,amscd,amsxtra,MnSymbol}
\usepackage{mathrsfs}
\usepackage{color}
\usepackage[cmtip,all]{xy}
\usepackage{eucal}
\usepackage{bbold}
\usepackage{upgreek}
\usepackage{paralist}
\usepackage{tikz-cd}
\usepackage{dsfont}
\usepackage{bbm}
\usepackage[bbgreekl]{mathbbol}

\usepackage{fullpage}

\usepackage[capitalise]{cleveref}  
\newcommand{\clevertheorem}[3]{
	\newtheorem{#1}[thm]{#2}
	\crefname{#1}{#2}{#3}
}

\numberwithin{equation}{section} 
\numberwithin{figure}{section} 

\theoremstyle{plain} 
\newtheorem{thm}{Theorem}[subsection]
\crefname{thm}{Theorem}{Theorems}
\newtheorem*{thm*}{Theorem}
\clevertheorem{proposition}{Proposition}{Propositions}
\clevertheorem{prop}{Proposition}{Propositions}
\newtheorem*{prop*}{Proposition}
\clevertheorem{theorem}{Theorem}{Theorems}
\clevertheorem{lemma}{Lemma}{Lemmas}
\clevertheorem{corollary}{Corollary}{Corollaries}
\clevertheorem{cor}{Corollary}{Corollaries}
\clevertheorem{conj}{Conjecture}{Conjectures}
\clevertheorem{questn}{Question}{Questions}

\theoremstyle{definition} 
\clevertheorem{definition}{Definition}{Definitions}
\clevertheorem{assumption}{Assumption}{Assumptions}
\clevertheorem{defn}{Definition}{Definitions}
\clevertheorem{notn}{Notation}{Notations}
\clevertheorem{conv}{Convention}{Conventions}

\clevertheorem{remark}{Remark}{Remarks}
\clevertheorem{rmk}{Remark}{Remarks}
\clevertheorem{warn}{Warning}{Warnings}
\clevertheorem{example}{Example}{Examples}
\clevertheorem{summ}{Summary}{Summaries}

\DeclareMathSymbol\bbDelta \mathord{bbold}{"01}
\DeclareMathSymbol\bDelta \mathord{bbold}{"01}

\newtheorem{remark*}{Remark}

\newtheorem{construction}[thm]{Construction}

\newtheorem{notation}[thm]{Notation}

\newcommand{\bD}{{\mathbb D}}

\renewcommand{\P}{{\mathbb P}}

\newcommand{\mA}{{\mathcal A}}
\newcommand{\mB}{{\mathcal B}}
\newcommand{\mC}{{\mathcal C}}
\newcommand{\mD}{{\mathcal D}}
\newcommand{\mE}{{\mathcal E}}

\newcommand{\mH}{{\mathcal H}}

\newcommand{\mM}{{\mathcal M}}
\newcommand{\mN}{{\mathcal N}}
\newcommand{\mO}{{\mathcal O}}
\newcommand{\mP}{{\mathcal P}}

\newcommand{\mS}{{\mathcal S}}

\newcommand{\mV}{{\mathcal V}}
\newcommand{\mW}{{\mathcal W}}

\newcommand{\mY}{{\mathcal Y}}

\newcommand{\A}{A}
\newcommand{\B}{\mathrm{B}}
\newcommand{\C}{C}

\newcommand{\E}{{E}}
\newcommand{\F}{{F}}
\newcommand{\G}{{G}}

\newcommand{\K}{{K}}
\renewcommand{\L}{{\mathrm L}}

\newcommand{\N}{{\mathrm N}}
\renewcommand{\P}{{P}}
\newcommand{\Q}{{Q}}
\newcommand{\R}{{\mathrm R}}
\newcommand{\rS}{{S}}

\newcommand{\W}{W}
\newcommand{\X}{X}
\newcommand{\Y}{Y}
\newcommand{\Z}{Z}

\newcommand{\bj}{{j}}
\newcommand{\bi}{{i}}
\newcommand{\m}{{m}}
\newcommand{\bk}{{k}}

\newcommand{\n}{{n}}

\newcommand{\op}{\mathrm{op}}

\newcommand{\surj}{\mathrm{surj}}

\newcommand{\Min}{\mathrm{Min}}
\newcommand{\Max}{\mathrm{Max}}

\newcommand{\colim}{\mathrm{colim}}

\newcommand{\LMod}{{\mathrm{LMod}}}

\newcommand{\rev}{{\mathrm{rev}}}

\newcommand{\Env}{{\mathrm{Env}}}  
\newcommand{\ot}{\otimes}

\newcommand{\co}{\mathrm{co}}

\newcommand{\univ}{\mathrm{univ}}

\newcommand{\id}{\mathrm{id}}
\newcommand{\Cat}{\mathrm{Cat}}

\newcommand{\Set}{\mathrm{Set}}

\newcommand{\Alg}{\mathrm{Alg}}

\newcommand{\Rep}{\mathrm{Rep}}

\newcommand{\Fun}{\mathrm{Fun}}

\newcommand{\cocart}{{\mathrm{cocart}}}

\newcommand{\cart}{{\mathrm{cart}}}
\newcommand{\Cart}{{\mathrm{Cart}}}
\newcommand{\bicart}{{\mathrm{bicart}}}

\newcommand{\lax}{{\mathrm{lax}}}
\newcommand{\oplax}{{\mathrm{oplax}}}

\newcommand{\tu}{{\mathbb 1}}
\newcommand{\coop}{\mathrm{coop}}

\newcommand{\Map}{{\mathrm{Map}}}

\newcommand{\Mor}{{\mathrm{Mor}}}

\newcommand{\BM}{\mathrm{BM}}

\newcommand{\coCart}{{\mathrm{coCart}}}

\newcommand{\cop}{\mathrm{cop}}

\newcommand{\map}{\mathrm{Map}}

\newcommand{\RMor}{\mathrm{RMor}}

\newcommand{\cube}{{\,\vline\negmedspace\square}}

\newcommand{\scat}{\mathcal{C}\mathit{at}}

\newcommand{\fcat}{\mathfrak{Cat}}

\newcommand{\LFib}{\mathrm{LFib}}

\begin{document}

\title{\textsc{Colimits
in Oriented Category Theory}}

\author{David Gepner and Hadrian Heine}

\maketitle

\begin{abstract}
In higher category theory, lax colimits are often understood to be a more useful and powerful generalization of usual (homotopy) colimits, which can be recovered from the lax colimit by a suitable localization.
However, lax colimits do not provide the correct notion of gluing from the geometric perspective.
Indeed, they are incompatible with the notion of categorical dimension, the Gray tensor product, and other basic geometric operations.
In this paper, we develop the theory of oriented colimits, which correct the defects of lax colimits, and agree with lax colimits in dimension less than or equal to one.

In order to study oriented colimits, we introduce a version of the Grothendieck construction which is compatible with enrichment in the Gray tensor product.
We prove that the Grothendieck construction induces an equivalence between cartesian fibrations and presheaves of $(\infty,\infty)$-categories, which is enriched in the Gray tensor product
of $(\infty,\infty)$-categories.
Oriented colimits simultaneously generalize the concept of lax colimits and the Gray tensor product,
and differ from lax colimits in much the same way in which the Gray tensor product differs from the cartesian product.
We demonstrate the necessity of oriented colimits by showing that various fundamental constructions in higher category theory fail to be lax colimits but are instances of oriented colimits.

As applications, we classify higher-categorical principal bundles, represent higher dimensional adjunctions by bicartesian fibrations of $(\infty,\infty)$-categories, and obtain higher categorical versions of Quillen's Theorems A and B, which admit very natural formulations in
our framework.

\end{abstract}

\tableofcontents

\vspace{5mm}

\section{Introduction}

\subsection{Colimits in geometry and topology}
Colimits are ubiquitous in geometry, where global objects such as manifolds are typically glued together from local pieces.
For instance, any topological space $X$ is the colimit of an open cover $\{U_i\}$, since $X$ is isomorphic to the colimit of its \v{C}ech nerve
\[
  \cdots \coprod U_{i}\times_X U_j\times_X U_k \Rrightarrow \coprod U_i\times_X U_j \rightrightarrows \coprod U_i.
\]
If $X$ is a manifold and all of the iterated intersections of the $U_i$ are contractible, then $X$ is homotopy equivalent to the geometric realization of its \v{C}ech nerve, a Kan complex which models $X$ as an $\infty$-groupoid.

The homotopical perspective is a fruitful one because various families of structures over $X$, such as principal $G$-bundles, are classified by homotopy classes of maps to an appropriate classifying space.
Specifically, they arise by pulling back a universal family along a classifying map, which is uniquely determined up to homotopy.
This recovers the standard cohomological classification of $G$-bundles, the Mayer--Vietoris sequence in cohomology, and other instances of descent spectral sequences.

Keeping track of higher categorical structure often improves the situation, by identifying the category of such structures with the category of morphisms to the classifying object.
Already in $\infty$-groupoid theory, all maps are classifiable: any map of homotopy types $Y\to X$ is pulled back from the universal family, the projection $\iota(\mS_*)\to\iota(\mS)$ from the core of the category\footnote{In this paper we will write $n$-category in place of $(\infty,n)$-category, for all $-1\leq n\leq\infty$, and write $(m,n)$-category, or strict $(m,n)$-category, when we need to refer to these other notions.} of pointed spaces to the core of the category of spaces (also known as $\infty$-groupoids, homotopy types, or anima).
If as above one wishes to restrict to maps with fixed fibers, it is simply a matter of mapping into the corresponding components of $\iota(\mS)\simeq\coprod_{Z\in\mS} \mathrm{BAut}_{\mS}(Z)$.

Of course, it unnecessary and artificial to pass to the groupoid core, 
and one might guess that the categorical version of the universal family is then the projection $\Cat_*\to\Cat$, the functor which forgets the basepoint.
But this is not sufficiently lax; rather, the universal cocartesian fibration is the projection $\Cat_{\ast//^\oplax}\to\Cat$ from the oplax slice under the point, meaning that morphisms only need to preserve the basepoint up to a specified not-necessarily-invertible morphism.
But even this is not lax enough: $\Cat$ naturally underlies the $2$-category $\scat$, and we ought to be able to consider cocartesian fibrations over a base $2$-category.
It is only in the limiting case where $n=\infty$ that the oplax slice projection $n\scat_{*//^\oplax}\to n\scat$\footnote{We write $n\scat$ for the $\infty$-category of $n$-categories, for all $0\leq n\leq\infty$.} doesn't discard interesting information.
Here one obtains a clean statement: any cocartesian fibration of $\infty$-categories is the pullback of the universal cocartesian family $\infty\scat_{\ast//^\oplax}\to\infty\scat$.

\subsection{Colimits in higher category theory}
Recall that an $n$-category is by definition a category enriched in the category of $(n-1)$-categories, and that
\[
\infty\Cat\simeq\lim\{\cdots\to (n+1)\Cat\to n\Cat\to (n-1)\Cat\to\cdots\}
\]
is the limit, taken along the core functors.
It is easy the see that $n\Cat$ is presentable, hence complete and cocomplete, for all $0\leq n\leq\infty$.
It is even cartesian closed, and is therefore tensored and cotensored over itself in the usual way.
However, these statements are of limited utility.
While it is extremely useful to know that $\infty\Cat$ admits colimits and functor categories, these cartesian-induced operations are not sufficiently lax or compatible with the fundamental notion of categorical dimension.

Even in 1-dimensional category theory, where the Gray tensor and the cartesian product still happen to coincide, there is a more refined notion of colimit, the lax colimit.
The lax colimit of a functor $F:X\to\Cat$ is computed as the total space of the Grothendieck construction applied to $F$, denoted $\int_X F$.
Dually, the lax limit is computed as the category of sections of the projection $\int_X F\to X$.
The ordinary (co)limit is recovered by inverting cocartesian morphisms or restricting to cocartesian sections.

While the lax (co)limit is a powerful and extremely useful construction, it is still not the correct notion, as it also fails to be compatible with categorical dimension.
This is because the lax colimit of the constant functor $X\to\infty\scat$ valued at at $\infty$-category $Y$ is the cartesian product $X\times Y$, and the lax limit is the functor category $\Fun(X,Y)$.
What we need instead is a version of the colimit which returns the Gray tensor product $X\boxtimes Y$ in this basic case.
This is the oriented colimit, and the natural notion of colimit in categories enriched in the Gray monoidal structure on $\infty\Cat$.

\subsection{Oriented category theory}
As in geometry and homotopy theory, the correct product of an $m$-dimensional and an $n$-dimensional category should be an $(m+n)$-dimensional category.
However, their cartesian product is only of dimension the maximum of $\{m,n\}$.
Hence the cartesian product is not graded by dimension.
This leads to a whole host of problems, including the facts that the categorical cyliner, cone, and suspension constructions are only functors of the underling category of $\infty$-categories, and are not compatible with the cartesian enrichment of $\infty\Cat$ over itself via the usual functor categories.

In order to remedy this situation, one is led inevitably to considering the Gray tensor product as the correct refinement of the cartesian product, and its adjoints, the lax and oplax functor categories, as the correct enrichment.
This is the perspective of oriented category theory, as explained below.
As mentioned above and explained in \cite{GepnerHeine2026}, \cite{gepner2026homotopy}, and \cite{gepner2026fibrations}, the cartesian product is not compatible with the natural notion of dimension in higher category theory, and must be replaced by the Gray tensor product in order to remedy this defect.
This difficulty propagates through the entire theory and necessitates the development of a theory of higher categories in which morphism $\infty$-categories compose via the Gray tensor product operation.
Since the Gray tensor comes in two antisymmetric variants, the lax and oplax versions, which are related by various (anti)involutions, it is perhaps useful to think of the resulting structure as a choice of orientation which dictates the precise nature of the composition of the morphism $\infty$-categories.

We write $(\infty\Cat,\boxtimes)$ for the category of $\infty$-categories equipped with the oplax Gray tensor monoidal structure, the category of oriented categories
\[
\Cat\boxtimes =\Cat_{(\infty\Cat,\boxtimes)}
\]
is by definition the category of categories enriched in the monoidal category $(\infty\Cat,\boxtimes)$.
There is also a negatively, or anti-, oriented version
\[
\boxtimes\Cat = {_{(\infty\Cat,\boxtimes)}\Cat},
\]
as well as the bioriented version
\[
\boxtimes\Cat\boxtimes = {_{(\infty\Cat,\boxtimes)}\Cat_{(\infty\Cat,\boxtimes)}}.
\]
With these enrichments, we obtain oriented, antioriented, and bioriented versions of $\infty\Cat$, denoted $\infty\fcat$.

We refer to the resulting theory as oriented category theory, since it is both descriptive as well as compatible with certain pre-existing notions, such as the oriented pullback and oriented pushout familiar from low-dimensional higher category theory.
In fact, these oriented construction are special cases of the theory of oriented (co)limits, which we develop in this paper.

\subsection{Colimits in oriented category theory}

The most basic $n$-category, the $n$-disk $\bD^n=S(\bD^{n-1})=\cdots=S^{\circ n}(\bD^0)$, can be obtained as an iterated unreduced categorical suspension of the point $\bD^0$, the terminal $\infty$-category (which is also the terminal $\infty$-groupoid).
The categorical suspension of an $\infty$-category $X$ is an example of an oriented, but not (op)lax, colimit, namely the (partial) oriented pushout
\[
S(X)=\bD^0\underset{X}{\vec{+}}\,\bD^0
\]
of the span diagram $\bD^0\xleftarrow{} X\xrightarrow{}\bD^0$.
Dually, the morphism categories of an $\infty$-category $X$, appear as the (partial) oriented, but not (op)lax, pullbacks
\[
\Mor_X(s,t)=\bD^0\underset{X}{\vec{\times}}\,\bD^0
\]
of the cospan diagram $\bD^0\xrightarrow{s} X\xleftarrow{t}\bD^0$.

More generally, we have the following table depicting various types of colimits in the classical set-based, homotopical space-based, (op)lax categorical, and oriented higher categorical contexts.
Here $M$ denotes a monoidal $\infty$-category, or category in the classical or homotopical context, and $(-)^\mathrm{gp}$ the group completion.
$$\begin{tabular}{| c | c | c | c| c |}
\hline
& Classical & Homotopical & Categorical & Oriented\\
\hhline{|=|=|=|=|=|}	
Pushout of $Y\leftarrow X\rightarrow Z$ & ${Y\!+\!Z}\!\!\!\underset{X\times \partial I}{+}\!\!\!{X}$ & ${Y\!+\!Z}\!\!\!\underset{X\times\partial I}{+}\!\!\! X\!\times\! I$ & ${Y\!+\!Z}\!\!\!\underset{X\times\partial\bD^1}{+}\!\!\!{X}\!\times\!\bD^1$ & ${Y\!+\!Z}\!\!\!\underset{X\boxtimes\partial\bD^1}{+}\!\!\!{X\!\boxtimes\!\bD^1}$ \\
\hline
\!\!\!\! Constant $Y$-valued functor on $X$\!\!\! &
$X\times Y$ & $X\times Y$ & $X\times Y$ & $ X\boxtimes Y $ \\
\hline
Coend of $[n]\mapsto M^n$ & $*$ & \!\!\!$BM[\mathrm{morphisms}^{-1}]^\mathrm{gp}$\!\!\! & \!\!\!$BM[\mathrm{morphisms}^{-1}]$\!\!\! & $BM$\\
\hline
\end{tabular}$$

As the suspension and morphism objects (a.k.a. oriented paths, or loops when the source and target agree) examples show, it is only the oriented (co)limit which is the suitable for most geometric sorts of applications, just as the homotopy (co)limit is so useful in homotopy theory and its applications in manifold theory or derived categories.
Specifically, it is only the oriented constructions which are compatible with the corresponding operations in homotopy theory and topology.

The category of $\infty$-categories, refines to an oriented category $\infty\fcat$ as well as an $\infty$-category $\infty\scat$, depending on whether or not we enrich it, respectively, with the Gray tensor or cartesian monoidal structure.
The $\infty$-category $\infty\scat$ is capable of expressing (op)lax (co)limits, which are still quite useful even in bicategory theory and are the sorts of higher category (co)limits most frequently studied in the literature.

As alluded to above, however, (op)lax colimits are still somewhat degenerate constructions, as they are not graded by dimension.
The lax colimit of the constant functor $\mD\to\infty\scat$ with value $\mC$ is $\mC\times\mD$, whereas the oriented colimit of the same functor is the Gray tensor product $\mC\boxtimes\mD$.
This strict additivity of dimension is ultimately the reason why it is only the oriented theory which behaves in a geometrically desireable enough way that we are able to import much of the machinery of algebraic topology (or really, obtain them as special cases of their higher-dimensional categorical analogues).

For formal reasons, the category of (anti)oriented categories admits all weighted limits and colimits.
It is therefore possible to consider more general weighted colimits, just as in the (op)lax case, involving $\infty\scat$ equipped with the cartesian monoidal structure.
The (op)lax colimit, however, is a refinement of the ordinary colimit which uses a specific weight obtained from the indexing diagram in question, as explained in \cite{articles}.
The (anti)oriented colimit is a generalization of the (op)lax colimit,
obtained similarly, except that the weights are enriched in the Gray tensor product instead of the cartesian product.

\subsection{The Grothendieck construction}
The Grothendieck construction arose in the study of stacks and fibered categories, where is was used to construct a cartesian fibration $p:X\to S$ from a functor $f:S^{\op}\to\Cat$.
Classically, the fibered category $p:X\to S$ associated to the functor $f:S^{\op}\to\Cat$ (which we assume for the moment factors through $(1,1)$-categories) has objects those pairs $(s,x)$ with $s\in S$ and $x\in f(s)$ and morphisms those pairs $(\varphi,\alpha):(s,x)\to (t,y)$ with $\varphi:s\to t$ in $S$ and $\alpha:x\to \varphi^*(y)$ in $f(s)$, where $\varphi^*:f(t)\to f(s)$ denotes the induced functor.
For certain applications, it is preferable to view presheaves simply as functors $f:S\to\Cat$, by replacing $S$ with $S^{\op}$, in which case the classical Grothendieck construction induces an equivalence between category-valued functors on $S$ and cocartesian fibrations over $S$.

Arguably the most important feature of the Grothendieck construction is that is provides a direct construction of the lax colimit of a functor $f:S\to\Cat$ as the total category $\int_S f$ of the associated cocartesian fibration over $S$.
Analogously, we show that the higher Grothendieck construction computes the oplax colimit of a functor $f:S\to\infty\scat$ for any $\infty$-category $S$.
It is crucial to know the image of the Grothendieck construction
\[
\int_S:\Fun(S,\infty\scat)\to\infty\scat_{/S}.
\]
It automatically factors through $\infty\scat_{/S}\to\infty\scat$, either by its definition of the Grothendieck construction, or by the universal property of the oplax colimit, since the oplax colimit of the terminal object of $\Fun(S,\infty\scat)$ is $S$ itself.

The universal cocartesian fibration can be equivalently described as the oriented fiber of the identity of $\infty\scat$ at the tautological basepoint $\bD^0\to\infty\scat$.
As the oriented fiber of any functor $S\to T$ automatically produces a cocartesian fibration $X\to S$, the functor $\int_S$ necessarily factors through the subcategory $\infty\scat_{/S}^{\cocart}$ spanned by the cocartesian fibrations and (higher) morphisms thereof.
But the oriented fiber refines to an antioriented functor of antioriented categories, and therefore the Grothendieck construction carries this additional structure as well.

To see that the Grothendieck construction is an equivalence onto the (nonfull) subcategory of cartesian fibrations,\footnote{This result is also obtained by Loubaton \cite{loubaton2024categorical} using complicial models.}
one can argue as follows.
For any $\infty$-category $S$ there is a canonical functor
\[
j:S\to \infty\scat^{\cart}_{/S}
\]
which sends $t\in S$ to the (op)lax slice projection $S_{//^\oplax t}\to S$, the image of the Yoneda embedding under the Grothendieck construction.
Given a cartesian fibration $p:X\to S$, we obtain an ``$S$-nerve'' functor
\[
\Fun(S^{\circ},\infty\scat)\to \infty\scat^{\cart}_{/S}\qquad t\mapsto\Fun^\cart_S(S_{//^\oplax t},X).
\]
We show that the (opposite) Grothendieck construction is the left adjoint of this $S$-nerve functor.
Moreover, this observation suggests a strategy to prove that the (opposite) Grothendieck construction is an equivalence onto the subcategory of cartesian fibrations, using the fact that the full subcategory
\[
\{S_{//^\oplax t}:t\in S\}\subset\infty\scat^{\cart}_{/S}
\]
is dense.
This means that the Grothendieck construction of the $S$-nerve functor is an endoequivalence of the category $\infty\scat^{\cart}_{/S}$, and therefore precisely determines the image of the Grothendieck construction.

Consequently, we may unstraighten cartesian fibrations $p:X\to S$ to presheaves $f:S^{\circ}\to\infty\scat$, providing a means of overcoming one of the major technical challenges involved in higher category theory.
This is tremendously useful in practice in homotopy-coherent mathematics, fibrations can usually be specified even when there is only naturally a contractible space of choices for something, whereas a functor requires specific choices for everything which are compatable with the infinity hierarchy of higher compositions.

Already in ordinary category theory, there are several different but useful types of fibrations, and in higher category theory there are many more, although the proliferation of these notions is largely due to the infinite number of automorphisms of the category of $\infty$-categories itself.
For instance, we have the following commutative diagram of various types of fibrations, each of which defines a subcategory of the category of functors of $\infty$-categories.
$$
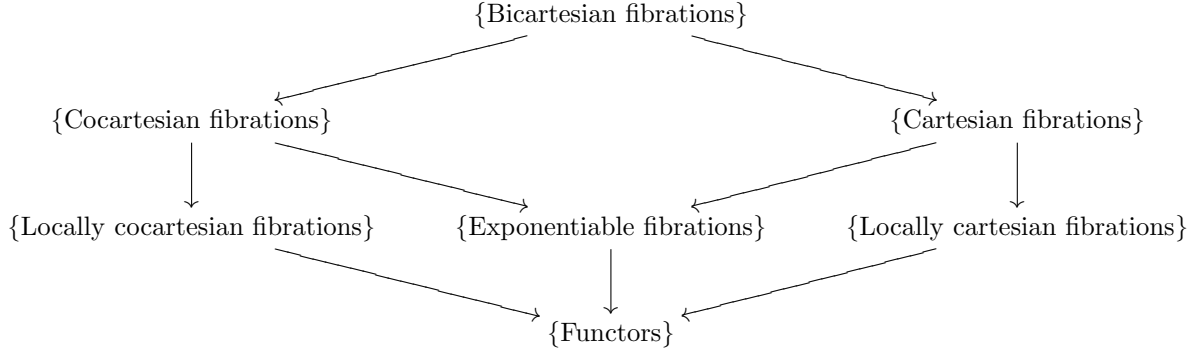
\begin{xy}
\xymatrix{
 &   \{ {\mathrm{Bicartesian} \ \mathrm{fibrations}} \} \ar[ld]\ar[rd] &
\\
\{ \mathrm{Cocartesian} \ \mathrm{fibrations} \}  \ar[rd] \ar[d] & 
& \{ \mathrm{Cartesian} \ \mathrm{fibrations}  \} \ar[ld] \ar[d]\\
\{ {\mathrm{Locally} \ \mathrm{cocartesian} \ \mathrm{fibrations} \} } \ar[rd] & \{ \mathrm{Exponentiable} \ \mathrm{fibrations} \} \ar[d] & {\{\mathrm{Locally} \ \mathrm{cartesian} \ \mathrm{fibrations} \} }  \ar[ld] \\
& \{ \mathrm{Functors}\} &
}
\end{xy} $$
Note that this excludes the bifibrations, as well as the infinitude of dual versions of (co)cartesian fibrations, such as the anti(co)cartesian fibrations.

\subsection{Quillen's theorems A and B}
There are a number of theoretically and practically significant applications of the general theory of
oriented limits and colimits. One particularly important one is it to the classification of cofinal functors.

Specifically,
Quillen's Theorem A, generated by Lurie to the homotopical 1-categorical world, gives a necessary and
sufficient condition for a functor to be cofinal.
We give a necessary and sufficient condition for a functor of $\infty$-categories $\phi: \mC \to \mD $ to be cofinal, in terms of the oriented fibers
$ \mC \,{\vec{\times}}_\mD \{ X \} $ of $\phi$ over the objects $X$ of $\mD$.
Namely, we show that the localization of $\mC \,{\vec{\times}}_\mD \{ X \} $ with respect to the class of morphisms which are cocartesian over $\mC$ must be contractible. 

We also prove an $\infty$-categorical version of Quillen's Theorem B.
Again, let $\phi: \mC \to \mD $ be a functor of $\infty$-categories and let $X$ be an object of $\mC$.
Then the groupoidification oriented functor, also called the classifying space functor, or the left adjoint $\infty\fcat \to \mS$ of the inclusion $\mS\to\infty\fcat$,
preserves the oriented left fiber of $\phi $ over $X$ if, for every morphism $X \to \Y$ in $\mD$, the induced functor
$$ \mD_{//^\oplax X} \times_\mD \mC \to \mD_{//^\oplax Y} \times_\mD \mC $$ is a weak equivalence.

\subsection{Higher dimensional adjunctions}

Duality is one of the fundamental organizing principles of modern mathematics. In category theory, duality is encoded by adjunctions.
Higher category theory replaces adjunctions by higher dimensional adjunctions, which are central in the study of topological field theories \cite{Cob}.
In cobordism $\infty$-categories every cell --- given by a cobordism --- admits adjoints given by the same cobordism of eventually reversed orientation,
and topological quantum field theories are symmetric monoidal functors defined on bordism $\infty$-categories.

In classical category theory, there is a deep connection between fibrations and adjunctions: adjunctions between two categories are elegantly encoded by bicartesian fibrations over the 1-disk.
We extend this relationship to higher dimensional adjunctions and higher dimensional disks.

A {\em bicartesian fibration} is a functor which is both a cocartesian and cartesian fibration.
An adjunction $F:X\rightleftarrows Y:G$ between categories is equivalent to the data of a bicartesian fibration $p:Z\to\bD^1$ such that the fiber over $0$ is $X$ and the fiber over $1$ is $Y$.
When the base is a higher dimensional category, the situation becomes significantly more complicated.

Let's consider the simplest category of dimension greater than one, namely the two-dimensional disk $\bD^2$.
Over $\bD^2$, a bicartesian fibration $p:Z\to\bD^2$ gives rise to a pair of bicartesian fibration $ Z^L, Z^R \to \bD^1$ by pulling back along the two inclusions $\bD^1 \subset \bD^2.$ The latter correspond to adjunctions $F_0,F_1:X\rightleftarrows Y:G_0,G_1$ between $\infty$-categories $X$ and $Y$, the fibers of $p:Z\to\bD^2$ over 0,1.
Choosing objects $x\in X$ and $y\in Y$, then the induced functor $q:\Mor_Z(x,y)\to\bD^1$ is again a bicartesian fibration, and therefore it classifies an adjunction 
$$\Mor_Z(x,y)_0 \simeq \Mor_{Z^L}(x,y) \simeq \Mor_{Y}(F_0(x),y) \simeq \Mor_{X}(x,G_0(y)) \rightleftarrows $$$$ \Mor_Z(x,y)_1 \simeq \Mor_{Z^R}(x,y) \simeq \Mor_{Y}(F_1(x),y) \simeq \Mor_{X}(x,G_1(y)).$$
By naturality in $x$, we obtain an adjunction
\[
G_0(y) \rightleftarrows G_1(y)
\]
which by naturality in $y$ produces an adjunction
$
G_0 \rightleftarrows G_1.
$
Similarly, by naturality in $y$, we obtain an adjunction
\[
F_0(x) \rightleftarrows F_1(x)
\]
which by naturality in $x$ produces an adjunction
$
F_0 \rightleftarrows F_1,
$
which is an induced adjunction on mates.

We can use bicartesian fibrations to model $\infty$-categories with adjoints.
For example, in the 1-categorical case, the Yoneda embedding $X\to\mP(X)\simeq\infty\scat_{/X}^\cart$ sends an object of $x$ to the free cartesian fibration on $x\to X$.
If additionally $X$ has adjoints, this will factor through the (non full) subcategory $\infty\scat_{/X}^{\mathrm{bicart}}\subset\infty\scat_{/X}^{\cart}$ of bicartesian fibrations over $X$.

In order to obtain a notion of $\infty$-category with adjoints, we can reverse-engineered this process.
Namely, we can define an $\infty$-category with adjoints $X$ to be an $\infty$-category $X$ for which the Yoneda embedding $X\to\infty\scat_{/X}^\cart$ factors through the subcategory $\infty\scat^{\bicart}_{/X}$.
The advantage of this approach is that is automatically takes care of coherence issues.
Of course, other approaches are also possible.

\subsection{Main results}

A crucial feature of homotopy theory is that every homotopy type is the homotopy colimit of the constant functor on a point indexed by the space.
The theory of oriented colimits extends this feature to $\infty$-categories.

\begin{theorem}(\cref{Graytenso})
Let $\mC$ and $\mD$ be $\infty$-categories. The Gray tensor product $\mC\boxtimes\mD$ is the oriented colimit of the constant functor $\mD\to\infty\fcat$ with value $\mC$.
\end{theorem}

This result is not formal, despite the fact that oriented categories are categories enriched in the Gray tensor product.
The major steps in the proof of this theorem is the embedding of $\infty$-categories into oriented categories, as constructed in \cite{oriented}, together with the Grothendieck construction, which we also upgrade to the oriented setting.

\begin{theorem}\label{THH1}(\cref{Groeq}, \cref{Groori})
Let $S$ be an $\infty$-category.
Mapping out of the slice categories is a fully faithful oriented functor
$$\infty\fcat_{/S}^\cocart\subset\boxtimes\Fun(S,\infty\fcat)$$
with image those antioriented functors $S\to\infty\fcat$ which factor through $\infty\scat\subset\infty\fcat$.
Moreover, this embedding restricts to give an equivalence of $\infty$-categories
\[
\int_S: \Fun(S,\infty\scat)\simeq\infty\scat_{/S}^\cocart
\]
in which a functor $S\to\infty\scat$ is sent to the oriented left fiber
\[
* \,{\vec{\times}}_{\infty\scat} S \to S.
\]
\end{theorem}

This latter equivalence is usually referred to as the Grothendieck construction, or the unstraightening and straightening inverse equivalences between functors and fibrations, respectively.
Our proof of this theorem is model-independent and uses induction on categorical dimension.

\begin{corollary}
The oriented left fiber $$ \{ *\} \,{\vec{\times}}_{\infty\fcat} \infty\scat \to \infty\scat $$ of the identity of $\infty\scat$ is the universal cocartesian fibration.
\end{corollary}

Considering the effect of the Grothendieck construction on morphism $\infty$-categories gives the following:

\begin{corollary}(\cref{Groorihom})
Let $S$ be an $\infty$-category and $F,G: S \to \infty\fcat$ be oriented functors that factor through the inclusion $\infty\scat\to\infty\fcat$
There is a canonical equivalence
$$ \L\Mor_{{\Fun\boxtimes}(S,\infty\mathfrak{Cat})}(F,G) \simeq \Fun^{\oplax, \cocart}_X(\int_S F, \int_S G)$$
    
\end{corollary}

We establish the following close connection between the Grothendieck construction and lax (co)limits and antioriented limits: 

\begin{theorem}\label{THH4}(\cref{laxgro}, \cref{limsect})
Let $F: S \to \infty\scat$ be a functor of $\infty$-categories and $\mE$ a collection of cells of $S$.
\begin{enumerate}[\normalfont(1)]\setlength{\itemsep}{-2pt}
\item The $\mE$-lax colimit of $F$ is the Grothendieck construction $\int F$ localized with respect to (co)cartesian lifts of cells of $\mE.$

\item The lax limit of $F$ is the $\infty$-category of sections of $\int F \to S $ sending cells of $\mE$ to (co)cartesian cells, 
and (strict) natural transformations over $S$.

\item The antioriented limit of $F$ is the $\infty$-category of sections 
of $\int F \to S $ sending cells of $\mE$ to (co)cartesian cells, 
and (higher) oplax natural transformations over $S$ sending cells of $\mE$ to (co)cartesian cells.

\end{enumerate}

\end{theorem}

\cref{THH1} and \cref{THH4} imply the following:

\begin{corollary}
For any functor $p:Y\to X$ of $\infty$-categories, the following conditions are equivalent:
\begin{enumerate}[\normalfont(1)]\setlength{\itemsep}{-2pt}
\item
$p:Y\to X$ is a cocartesian fibration.
\item[\em{(2)}]
$p:Y \to X$ is the oriented left fiber of a functor $X\to\infty\scat$ at the terminal object inclusion $*\to\infty\scat$.
\item[\em{(3)}]
$p:Y \to X$ is the lax colimit of a functor $X \to\infty\scat$.
\end{enumerate}
\end{corollary}

\cref{THH4} implies the following $\infty$-categorical version of Quillen's Theorem A:

\begin{corollary}(\cref{QuillenA})
Let $Y$ be an $ \infty$-category and $\mE$ a collection of cells of $Y$.
A functor $\phi: Y \to X $ is $\mE$-lax cofinal if
and only if for every $t \in X$
the induced functor $$\{ t \} \,{\vec{\times}}_X Y \to X_{t//^\oplax} $$ induces an equivalence 
$$ (\{ t \} \,{\vec{\times}}_X Y)[{\bar{\mE}^{-1}}] \to (X_{t//^\oplax})
[(\overline{\phi(\mE)})^{-1}].$$ 

\end{corollary}

We proof the following fundamental result on oriented fibers, which implies a higher-categorical version of Quillen's Theorem B:

\begin{theorem}(\cref{QuillenB})
Let $\phi: \mC \to \mD $ be a functor of $\infty$-categories such that 
for every morphism $X \to \Y$ in $\mD$ the induced functor
$$ \mD_{//^\oplax X} \times_\mD \mC \to \mD_{//^\oplax Y} \times_\mD \mC $$ is a weak equivalence, i.e. induces an equivalence on classifying spaces.

For every $Z \in \mC$ the classifying space functor $\infty\fcat \to \mS$
preserves the oriented left fiber of $\phi $ over $Z$.

\end{theorem}

\cref{THH1} implies the following classification of bicartesian fibrations by higher adjunctions:

\begin{theorem}(\cref{Grobica})
For every $\infty$-category $S$ there is a subcategory $\infty\scat^{\mathrm{L}}\subset \infty\scat$ of left adjoint functors and left adjoint higher natural transformations,
and a natural equivalence 
$$ \Map_{\infty\Cat}(S,\infty\scat^{\mathrm{L}}) \simeq \iota_0(\infty\scat^{\bicart}_{/S}),$$
where the right hand side is the space of bicartesian fibrations over $S$.

\end{theorem}

Let $S$ be an $\infty$-category. There is a similar result for right adjoints using the subcategory $\infty\scat^{\mathrm{R}}\subset \infty\scat$ of right adjoint functors and right adjoint higher natural transformations giving a natural equivalence 
$$ \Map_{\infty\Cat}(S,(\infty\scat^{\mathrm{R}})^\coop) \simeq \iota_0(\infty\scat^{\bicart}_{/S}).$$

We obtain the following duality between higher dimensional left and right adjoints:

\begin{corollary}(\cref{leftrightadj})
Let $S$ be an $\infty$-category.
There is a canonical equivalence of $\infty$-categories
$$ \infty\scat^{\mathrm{L}} \simeq (\infty\scat^{\mathrm{R}})^\coop$$
sending left to right adjoints.

\end{corollary}

We use \cref{THH1} to obtain the following characterization of cocartesian fibrations of $\infty$-categories:

\begin{theorem}(\cref{cocexp})
A functor $p: Y \to X$ is a cocartesian fibration if and only if it the following hold:

\begin{enumerate}[\normalfont(1)]\setlength{\itemsep}{-2pt}
\item The functor $p: Y \to X$ is exponentiable, i.e. the pullback functor $p^*:\infty\scat_{/X}\to\infty\scat_{/Y}$ admits a right adjoint $p_*$.

\item The functor $p: Y \to X$ is a locally cocartesian fibration, i.e.
for every functor $\bD^n \to X$ for $n \geq 0$ the pullback $ \bD^n \times_{X} Y \to \bD^n$ is a cocartesian fibration.
    
\end{enumerate}
\end{theorem}

The only if-direction of this result was also obtained by Loubaton \cite{loubaton2024categorical}, although he develops this theory in the complicial model for $\infty$-categories, and his definitions are somewhat different.

\subsection{Relation to other work}
Lax colimits in higher category theory have been studied by a number of authors. In 2-category theory they have been studied by Street \cite{street1976limits}, Kelly \cite{kelly1989elementary}, Lack \cite{lack20092} and Descotte \cite{descotte2018sigma}, and 
were considered in the homotopical setting by Gepner-Haugseng-Nikolaus \cite{articles}, Berman \cite{berman2024lax}, Gagna-Harpaz-Lanari \cite{gagna2020fibrations} and Abell\'an-Haugseng-Martini \cite{abellan2026free}.
In $n$-category theory and $\infty$-category theory, colimits were studied by Moser-Rasekh-Rovelli 
\cite{moser2023limits}, Loubaton, \cite{loubaton2024categorical} and Heine \cite{heine2024higher}.

The lax colimit of a category-valued functor is computed by the Grothendieck construction, which explains the central role of the
Grothendieck construction in the theory of colimits.
Originally studied by Grothendieck in stack theory, in order to efficiently encode pseudo-functors valued in $(2,1)$-categories,
the Grothendieck construction was extended by Lurie \cite{lurie.HTT}, \cite{lurie2009infinity} to the homotopical setting.
It was further extended to more general versions of fibrations by Ayala-Francis \cite{MR4074276}, Blom \cite{blom2024straightening} and Heine \cite{heine2026local}, to fibrations of operads by Heuts \cite{heuts2011algebras} and to monoidal categories by Ramzi \cite{ramzi2026monoidal}.
Moreover it was generalized to 2-categories by Abell\'an-Stern \cite{abellan20262} and Abell\'an-Gagna-Haugseng \cite{abellan2025straightening}, to $n$-categories by Nuiten \cite{nuiten2023straightening} and Moser-Rasekh-Rovelli \cite{moser2023inftyncategoricalstraighteningunstraighteningconstruction}, and to $\infty$-categories by Loubaton \cite{loubaton2024categorical}.

\subsection{Notation and terminology}

We fix a hierarchy of set-theoretic universes whose objects we call small, large, very large, etc.
We call a space (equivalently, $\infty$-groupoid) $X$ small, large, etc. if for any choice of basepoint and natural number $n$ its homotopy sets $\pi_n X$ are small, large, etc.
We call an $\infty$-category small, large, etc. if its maximal subspace and all its mapping spaces are small, large, etc.

We refer to (not necessarily univalent) weak $(\infty,\n)$-categories for $0 \leq \n \leq \infty$ simply as $\n$-categories, and we refer to (not necessarily univalent) weak $(\n,\n)$-categories as $(\n,\n)$-categories.
In particular, we refer to (not necessarily univalent) $(\infty,1)$-categories as 1-categories, or simply categories.
We will sometimes want to work strictly, which can be viewed as a basechange along the colimit-preserving symmetric monoidal functor $\mS\to\Set$.
In this case, we will refer to strict $(\n,\n)$-categories simply as strict $\n$-categories.\footnote{Note that an $(n,n)$-category need not be a strict $(n,n)$-category if $n>2$.
For instance, the fundamental $\infty$-groupoid of the $2$-sphere $S^2\simeq \mathrm{B}^2\Omega^2 S^2$ is an $(\infty,0)$-category which is not strict, nor are its $n$-truncations for any $n>2$.}

\begin{notation}
We will make use of the following notation and terminology when discussing categories, in the sense of categories enriched in the monoidal category of $\infty$-groupoids under the cartesian product.
\begin{enumerate}[\normalfont(1)]\setlength{\itemsep}{-2pt}
\item We write $\mS$ for the category of spaces, by which we mean small $\infty$-groupoids, homotopy types, or anima, and $\Set$ for the category of small sets.
\item We write $\infty\Cat$ for the large category of small $\infty$-categories.

\item We write $\Delta$ for (a skeleton of) the category of finite, non-empty, partially ordered sets and order preserving maps, whose objects we denote by $[\n] = \{0 < ... < \n\}$ for $\n \geq 0$.\footnote{This should not be confused with the category $\bDelta$ of oriented simplices.}
\item We write $\Map_{\mC}(A,B)$ for the space of maps (equivalently, $1$-morphisms) from $A$ to $B$ in $\mC$, for any category $\mC$ containing an ordered pair of objects $(A,B)\in\mC$.

\item We write $\ast$ for the final object and
$\emptyset$ for the initial object in any category.

\item We call a map of spaces $X \to Y$ an embedding if its fibers are empty or contractible, or likewise if the induced map $X \to X \times_Y X$ is an equivalence, or likewise if it is of the form $X \to X \coprod X'$ for spaces $X,X'.$ 

\item We call a morphism $X \to Y$ in any category $\mC$ a monomorphism if for every $Z \in \mC$ the induced map of spaces $\Map_\mC(Z,X) \to \Map_\mC(Z,Y) $ is an embedding. So embeddings of spaces are precisely monomorphisms in $\mS.$ 

\item We call a fully faithful functor $\mC \to \mD$ an embedding generalizing the notation for spaces.

\item We call a functor an inclusion if it is a monomorphism in $\Cat$. A functor $\mC \to \mD$ is an inclusion if and only if it induces an embedding on maximal subspaces and on all mapping spaces. 

\item For a diagram $X\to Z\leftarrow Y$ in a category $\mC$ we write $X\underset{Z}{\prod} Y$ or $X\underset{Z}{\times} Y$ for the pullback, and given a diagram $X\leftarrow W\to Y$ in a category $\mC$, we write $X\underset{W}{\coprod} Y$ or $X\underset{Z}{+} Y$ for the pushout.
\item If $\mC$ and $\mD$ are categories and $\mC\to\mD$ is a left adjoint functor with right adjoint $\mD\to\mC$, we often write $\mC\rightleftarrows\mD$ for this adjunction, where the left adjoint is understood to be the functor going from left to right.

\item
We write $\mC_*$ or $\mC_{\ast/}$ for the category of pointed objects in a category $\mC$, i.e. the full subcategory of $\Fun([1],\mC)$ of arrows in $\mC$ whose source is a final object.

\item We write $\infty\scat$ for the large $\infty$-category of small $\infty$-categories.\footnote{The morphism $\infty$-categories are formed via enrichment in the cartesian monoidal structure.}
\item We write $\Fun(\mD,\mC)$ for the $\infty$-category of functors from an $\infty$-category $\mD$ to an $\infty$-category $\mC$, the value at $\mC$ of the right adjoint to the functor $(-)\times\mC:\infty\Cat\to\infty\Cat$ for the cartesian product.

\item We write $\bD^1$ for the walking arrow, the category with two objects and a unique non-identity arrow.
\item We write $\partial\bD^1$ and $S^0$ for the maximal subspace in $\bD^1$, the set with two elements.

\item We write $\iota_{n}\mC$ for the $n$-category arising from an $\infty$-category $\mC$ by discarding all noninvertible morphisms above dimension $n.$
\item We write $\tau_{n}\mC$ for the $n$-category arising from an $\infty$-category $\mC$ by inverting all morphisms above dimension $n.$
\end{enumerate}
\end{notation} 

\subsection*{Acknowledgements}
We thank Fernando Abell\'an, Thomas Blom, Tim Campion, Felix Loubaton, Naruki Masuda, and Markus Spitzweck for interesting conversations related to the subject of this paper.
We thank the MPIM for their hospitality while much of this work was carried out.

\section{\mbox{Higher categories}}

\subsection{Enriched categories}

We first recall the notion of homotopy coherent enrichment, as defined and studied in, for instance, \cite{MR3345192}, \cite{heine2024higher}, \cite{heine2025equivalence}, \cite{HINICH2020107129}.
For every presentably monoidal category $\mV$ there is a presentable 2-category $${\mV\mathrm{-}\Cat}$$ of (not necessarily univalent) $\mV$-enriched categories and $\mV$-enriched functors and a forgetful functor $$ \iota:{\mV\mathrm{-}\Cat} \to \Cat$$ to the presentable 2-category $\Cat$ of (not necessarily univalent) categories,
which is an equivalence for $\mV= \mS$ the category of homotopy types \cite[Corollary 3.23.]{heine2024bienriched}. 

Let $${\mV\mathrm{-}\Cat}^\univ \subset {\mV\mathrm{-}\Cat}$$ be the reflexive full subcategory of univalent $\mV$-enriched categories.

\begin{notation}Let $\mV$ be a presentably monoidal category.
A $\mV$-enriched category $\mC$ has an underlying category $\iota(\mC)$, for every objects $X,Y \in \iota(\mC)$ a morphism object 
$$\Mor_\mC(X, Y) \in \mV$$
and for every objects $X,Y,Z \in \iota(\mC)$ a composition morphism in $\mV:$
$$\Mor_\mC(Y, Z) \ot \Mor_\mC(X, Y) \to \Mor_\mC(X, Z).$$
We write $\X \in \mC$ for $\X \in \iota(\mC)$ and usually notationally identify $\mC$ with $\iota(\mC).$

\end{notation}

The following is \cite[Example 2.134]{heine2024bienriched}:

\begin{example}\label{linenr}
Let $\mV$ be a presentably monoidal category. Every presentably left $\mV$-tensored category is a $\mV$-enriched category.
Every $\mV$-linear functor between presentably left $\mV$-tensored categories is a $\mV$-enriched functor.
In particular, $\mV$, which is presentably left tensored over itself, is a $\mV$-enriched category.

\end{example}

\begin{notation}Let $\mV$ be a presentably monoidal category.
Let $B\tu_\mV \subset \mV$ be the full $\mV$-enriched subcategory spanned by the tensor unit. Then $\Mor_{B\tu_\mV}(\tu_\mV, \tu_\mV) \simeq \tu_\mV.$
    
\end{notation}

For every enriched category there is an opposite one:

\begin{notation}
There is an involution $$(-)^\circ: {\mV\mathrm{-}\Cat} \simeq {\mV^\rev\mathrm{-}\Cat}$$ forming the opposite enriched category.
For every $\mC \in {\mV\mathrm{-}\Cat}$ and $X,Y \in \mC$
there are canonical equivalences
$ \iota(\mC^\circ) \simeq \iota(\mC)^\op$
and $$ \Mor_{\mC^\circ}(X,Y) \simeq \Mor_{\mC}(Y,X).$$
\end{notation}

The following is \cite[Proposition 3.72]{heine2024bienriched}:

\begin{proposition}

Let $\mV,\mW$ be presentably monoidal categories and 
$\phi:\mV \to \mW$ a lax monoidal functor.

\begin{enumerate}[\normalfont(1)]\setlength{\itemsep}{-2pt}
\item There is an induced functor
$\phi_!: \mV \mathrm{-}\Cat \to \mW \mathrm{-}\Cat$
that transfers the enrichment, which descends to a functor
$$\mV \mathrm{-}\Cat^\univ \to \mW \mathrm{-}\Cat^\univ.$$

\item For every $\mV$-enriched category $\mC$ and $X,Y \in \iota(\mC)$
there is a canonical equivalence
$$ \Mor_{\phi_!(\mC)}(X,Y) \simeq \phi(\Mor_\mC(X,Y)).$$

\item If $\phi$ is monoidal and admits a right adjoint $\gamma$,
there is an induced adjunction
$\phi_!: \mV \mathrm{-}\Cat \to \mW \mathrm{-}\Cat: \phi^*:=\gamma_!$
that descends to an adjunction 
$\mV \mathrm{-}\Cat^\univ \to \mW \mathrm{-}\Cat^\univ.$

\end{enumerate}

\end{proposition}

The following is \cite[Proposition 2.1.5]{GepnerHeine2026}:

\begin{proposition}\label{monochar}
Let $\mV$ be a presentably monoidal category.
A $\mV$-enriched functor $\phi: \mC \to \mD$ is a monomorphism in $\mV\mathrm{-}\Cat$ if and only it it induces an embedding $\iota(\mC) \to \iota(\mD)$ on underlying spaces and for every $A,B \in \mC$ the induced morphism 
$\Mor_\mC(A, B) \to \Mor_\mD(\phi(A), \phi(B))$ in $\mV$ is a monomorphism.
    
\end{proposition}

There is a close relationship between enriched categories
and tensored and cotensored categories:

\begin{definition}Let $\mV$ be a presentably monoidal category, $\mC$ a $\mV$-enriched category and $X \in \mC, V \in \mV.$	 
\begin{enumerate}[\normalfont(1)]\setlength{\itemsep}{-2pt}
\item The tensor of $V$ and $X$ in $\mC$ is the object $V \ot X \in \mC $ such that there is a morphism
$V \to \Mor_\mC(X, V \ot X) $ in $\mV$ that induces for every $Y \in \mC$ an equivalence
$$ \Mor_\mC(V \ot X,Y) \to \Mor_\mV(V, \Mor_\mC(X,Y)). $$ 
\item The cotensor of $V$ and $X$ in $\mC$ is the object ${^V X} \in \mC $ that is the tensor of $V $ and $X$ in the opposite $\mV^\rev$-enriched category $\mC^\circ.$
\end{enumerate}	
\end{definition}

Since the category of enriched categories forms a 2-category, there is a natural intrinsic notion of adjunction between enriched categories:

\begin{definition}Let $\mV$ be a presentably monoidal category.
A $\mV$-enriched functor $\mC \to \mD$ admits a left (right) adjoint if
it admits a left (right) adjoint in the 2-category $\mV \mathrm{-}\Cat.$

\end{definition}

\begin{notation}Let $\mV$ be a presentably monoidal category and $\mM, \mN$ be $\mV$-enriched categories.

Let $$\mV\mathrm{-}\Fun(\mM,\mN)$$ be the category of $\mV$-enriched functors
$\mM \to \mN.$

Let $$\mV\mathrm{-}\Fun^\L(\mM,\mN) \subset {\mV\mathrm{-}\Fun(\mM,\mN)}$$ be the full subcategory of $\mV$-enriched functors
$\mM \to \mN$ that admit a $\mV$-enriched right adjoint.
    
\end{notation}

The following is \cite[Remark 2.55]{heine2024bienriched}:

\begin{proposition}
Let $\mV$ be a presentably monoidal category.
\begin{enumerate}[\normalfont(1)]\setlength{\itemsep}{-2pt}
\item A $\mV$-enriched functor $\phi: \mC \to \mD$ admits a $\mV$-enriched right adjoint if and only if for every $\Y \in \mD$ the $\mV^\rev$-enriched functor
$\Mor_\mD(\phi(-),\X): \mC^\circ \to \mV$ is representable.
\item A $\mV$-enriched functor $\phi: \mC \to \mD$ admits a $\mV$-enriched left adjoint if and only if the opposite $\mV^\rev$-enriched functor $\phi^\circ: \mC^\circ \to \mD^\circ$ admits a $\mV^\rev$-enriched right adjoint. By (1) this holds if and only if for every $\Y \in \mD$ the $\mV$-enriched functor
$\Mor_\mD(\Y,\phi(-)): \mC \to \mV$ is representable.
\end{enumerate}
\end{proposition}

The following is \cite[Lemma 2.77]{heine2024bienriched}:

\begin{proposition}\label{adj}
Let $\mV$ be a presentably monoidal category.
\begin{enumerate}[\normalfont(1)]\setlength{\itemsep}{-2pt}
\item A $\mV$-enriched functor $\mC \to \mD$ admits a right adjoint if and only if it preserves tensors and the underlying functor admits a right adjoint.
\item A $\mV$-enriched functor $\mC \to \mD$ admits a left adjoint if and only if it preserves cotensors and the underlying functor admits a left adjoint.
\end{enumerate}

\end{proposition}

Next we introduce the enriched category of enriched presheaves \cite{heine2025equivalence}.

\begin{notation}

Let $\mV, \mW$ be presentably monoidal categories and $\mM$ a $\mV$-enriched category and $\mO$ a $\mW$-enriched category.
Let $\langle \mM, \mN \rangle $ be the $\mV \ot \mW$-enriched category that is the transfer of enrichment of the $\mV \times \mW$-enriched category $\mM \times \mN$ along the universal
monoidal functor $\mV \times \mW \to \mV \ot \mW$ preserving small colimits componentwise.
    
\end{notation}

The next theorem follows from \cite[Proposition 4.11 and Theorem 4.86]{heine2024bienriched}:

\begin{theorem}\label{psinho} Let $\mV, \mW$ be presentably monoidal categories and $\mN$ a univalent $(\mV, \mW)$-bienriched category. 
\begin{enumerate}[\normalfont(1)]\setlength{\itemsep}{-2pt}
\item Let $\mM$ be a small univalent $\mV$-enriched category. The
category ${\mV\mathrm{-}\Fun}(\mM, \mN)$ refines to an univalent $\mW$-enriched category characterized by an 
equivalence
$$ \mW\mathrm{-}\Fun(\mO,{\mV\mathrm{-}\Fun}(\mM, \mN)) \to {\mV \ot \mW\mathrm{-}\Fun}(\langle\mM,\mO\rangle,\mN)$$
natural in any univalent $\mW$-enriched category $\mO$.
\item Let $\mO$ be a small univalent $\mW$-enriched category. 
The category ${\mW\mathrm{-}\Fun}(\mO, \mN)$ refines to an univalent $\mV$-enriched category characterized by an equivalence
$$ \mV\mathrm{-}\Fun(\mM,{\mW\mathrm{-}\Fun}(\mO, \mN)) \to {\mV \ot \mW\mathrm{-}\Fun}(\langle\mM,\mO\rangle,\mN) $$
natural in any univalent $\mV$-enriched category $\mM$.

\item If $\mN$ is a presentably $\mV, \mW$-bitensored category, then ${\mV\mathrm{-}\Fun}(\mM, {\mN}) $ is a presentably right $\mW$-tensored category and ${\mW\mathrm{-}\Fun}(\mO, {\mN}) $ is a presentably left $\mV$-tensored category.

\end{enumerate}

\end{theorem}
    
\begin{definition}

Let $\mV$ be a presentably monoidal category and $\mC$ a small univalent $\mV$-enriched category.
The presentably left $\mV$-tensored category of $\mV$-enriched presheaves on $\mC$ is
$$ \mP_\mV(\mC):= \mV\mathrm{-}\Fun(\mC^\op,\mV). $$
\end{definition}

The next theorem, which follows from \cite[Theorem 3.41 and Theorem 4.70]{heine2024bienriched}, is an enriched version of the universal property of the category of presheaves as the free cocompletion under small colimits \cite[Theorem 5.1.5.6]{lurie.HTT}:

\begin{theorem}
\label{Yonedaext}
Let $\mV$ be a presentably monoidal category, $\mC$ a small univalent $\mV$-enriched category and $\mD$ a presentably left $\mV$-tensored category.
\begin{enumerate}[\normalfont(1)]\setlength{\itemsep}{-2pt}
\item There is a $\mV$-enriched embedding $\iota_\mC : \mC \to \mP_\mV(\mC)$ that sends $X$ to $\L\Mor_\mC(-,X)$ and induces
for every presentably left $\mV$-tensored category $\mD$ an equivalence
$$ {\mV\mathrm{-}\Fun^\L}(\mP_\mV(\mC),\mD)\to \mV\mathrm{-}\Fun(\mC,\mD).$$
\item Let $F: \mC \to \mD$ be a $\mV$-enriched functor and
$\bar{F}: \mP_\mV(\mC) \to \mD $ the unique $\mV$-enriched left adjoint extension of $F$.
For every $\mV$-enriched functor $G: \mP_\mV(\mC) \to \mD $
the induced morphism $$\Map_{{\mV\mathrm{-}\Fun^\L}(\mP_\mV(\mC),\mD)}(\bar{F},G) \to \Map_{\mV\mathrm{-}\Fun(\mC,\mD)}(F,G \circ \iota_\mC) $$ is an equivalence.
\item Let $F: \mC \to \mD$ be a $\mV$-enriched functor and
$\bar{F}: \mP_\mV(\mC) \to \mD $ the unique $\mV$-enriched left adjoint extension of $F$.
The $\mV$-enriched right adjoint $\mD \to \mP_\mV(\mC)$ of $\bar{F}$ sends $Y$ to $\Mor_\mD(-,Y) \circ F. $
\end{enumerate}
\end{theorem}

The following is the enriched Yoneda lemma, proven in 
\cite[Corollary 4.44]{heine2024bienriched}:

\begin{lemma}
Let $\mV$ be a presentably monoidal category and $\mC$ a small univalent $\mV$-enriched $\infty$-category. For every object $X \in \mC$ and $F \in \mP_\mV(\mC)$ the induced morphism $$\Mor_{\mP_\mV(\mC)}(\Mor_\mC(-,X),F) \to F(X) $$ is an equivalence.

\end{lemma}

We will use the following terminology to describe compatible enrichments in two monoidal categories:

\begin{definition}\label{bienr}
Let $\mV$ and $\mW$ be presentably monoidal categories.	 
\begin{enumerate}[\normalfont(1)]\setlength{\itemsep}{-2pt}
\item A left $\mV$-enriched category is a $\mV$-enriched category.
\item A left $\mV$-enriched functor is a $\mV$-enriched functor
\item A right $\mV$-enriched category is a $\mV^\rev$-enriched category.
\item A right $\mV$-enriched functor is a $\mV^\rev$-enriched functor.
\item A $\mV,\mW$-bienriched category is a $\mV \ot \mW^\rev$-enriched category.
\item A $\mV,\mW$-enriched functor is a $\mV \ot \mW^\rev$-enriched functor.
\end{enumerate}
\end{definition}

We often refer to a $(\mV,\mV)$-bienriched category simply as $\mV$-bienriched category.

\begin{notation}
Let $\mV, \mW$ be presentably monoidal categories.
\begin{enumerate}[\normalfont(1)]\setlength{\itemsep}{-2pt}
\item Let $${_\mV \Cat}:= {\mV}\mathrm{-}\Cat$$
be the 2-category of left $\mV$-enriched categories and left $\mV$-enriched functors.
\item Let $${\Cat_\mW}:= {\mW^\rev}\mathrm{-}\Cat$$
be the 2-category of right $\mW$-enriched categories and right $\mW$-enriched functors.
\item Let $${_\mV \Cat_\mW}:= {\mV \ot \mW^\rev}\mathrm{-}\Cat$$
be the 2-category of $\mV,\mW$-bienriched categories and $\mV,\mW$-enriched functors.
\end{enumerate}
\end{notation}

\begin{notation}

We will write $\L\Mor, \R\Mor, \mathrm{B}\Mor$ for the morphism objects in a left, right and bienriched category, respectively.    
\end{notation}

\begin{notation}
Let $\mV, \mW$ be presentably monoidal categories and $\mM, \mN$ be $(\mV,\mW)$-bienriched categories.
Let $$_\mV\Fun_\mW(\mM,\mN):= {\mV \ot \mW^\rev \mathrm{-}\Fun(\mM,\mN)}$$ be the category of $\mV,\mW$-enriched functors $\mM \to \mN.$

\end{notation}

The following is \cite[Example 2.134]{heine2024bienriched}:

\begin{example}
Let $\mV,\mW$ be presentably monoidal categories. Every presentably $\mV,\mW$-bitensored category is a $\mV,\mW$-enriched category.
Every presentably right $\mW$-tensored category is a right $\mW$-enriched category.
This follows also from \cref{linenr} identifying
presentably $\mV,\mW$-bitensored categories with presentably left $\mV \ot \mW^\rev$-tensored categories, and identifying
presentably right $\mW$-tensored categories with presentably left $\mW^\rev$-tensored categories.
In particular, $\mV$, which is presentably bitensored over itself, is a $\mV,\mV$-bienriched category.

Similarly, every $\mV,\mW$-linear functor between presentably $\mV,\mW$-bitensored categories is a $\mV,\mW$-enriched functor.

\end{example}

\begin{definition}Let $\mV$ and $\mW$ be presentably monoidal categories, let $\mC$ be a  
$\mV,\mW$-enriched category, and let $X \in \mC, V \in \mV, W \in \mW.$	 
\begin{enumerate}[\normalfont(1)]\setlength{\itemsep}{-2pt}
\item The left (co)tensor of $V$ and $X$ in $\mC$ is the (co)tensor of
$V \ot \tu_\mW \in \mV \ot \mW^\rev$ and $X$ in $\mC.$

\item The right (co)tensor of $W$ and $X$ in $\mC$ is the (co)tensor of
$\tu_\mV \ot W \in \mV \ot \mW^\rev$ and $X$ in $\mC.$

\end{enumerate}

\end{definition}

In the following we consider enriched slice categories.
To define bienriched slice categories we use the following lemma, which is \cite[Lemma 2.1.23.]{oriented}: 

\begin{lemma}\label{tensorunit}
Let $\mV, \mW$ be presentably monoidal categories whose tensor unit is final.
The tensor unit of $\mV \ot \mW$ is final.
	
\end{lemma}

\begin{corollary}\label{finality}
	
Let $\mV, \mW$ be presentably monoidal categories whose tensor unit is final.
Then ${_\mV \Cat_\mW}$ admits a final object $*$.
	
\end{corollary}

\begin{corollary}Let $\mV, \mW$ be presentably monoidal categories whose tensor unit is final.
For every $(\mV,\mW)$-bienriched category $\mC$
the induced functor ${_\mV\Fun_\mW}(*,\mC) \to \mC$ is an equivalence.
    
\end{corollary}

For the next notation we that the 2-category $\mV\mathrm{-}\Cat$ admits cotensors.

\begin{notation}
Let $\mV$ be a presentably monoidal category whose tensor unit is final, $\mC$ a $\mV$-enriched category and $X $ an object of $\mC.$
Let $$ \mC_{\X/}:= \{X\} \times_{\mC^{\{0\}}} \mC^{\bD^1} $$ be the pullback of the $\mV$-enriched functor $\mC^{\bD^1} \to \mC^{\{0\}}$ evaluating at the source along the $\mV$-enriched functor $* \to \mC$ classifying $X$. 
By definition there is a $\mV$-enriched functor $\mC_{\X/}\to \mC$.
    
\end{notation}

\begin{remark}\label{init}
Let $\mV$ be a presentably monoidal category whose tensor unit is final, $\mC$ a $\mV$-enriched category and $\X \to \Y, \X \to \Z$ morphisms in $\mC.$
The induced morphism $$\tu \to \Mor_{\mC_{\X/}}(\X,\X) \to \Mor_{\mC_{\X/}}(\X,\Z)$$ is an equivalence and the resulting commutative square
$$\begin{xy}
\xymatrix{
\Mor_{\mC_{\X/}}(\Y,\Z) \ar[d]^{} \ar[r]
& \Mor_{\mC}(\Y,\Z) \ar[d] \ar[d]^{}
\\ 
\tu \simeq \Mor_{\mC_{\X/}}(\X,\Z) \ar[r] & \Mor_{\mC}(\X,\Z)
}
\end{xy}$$
is a pullback square.
    
\end{remark}

The following is \cite[Lemma 2.2.3]{oriented}:

\begin{lemma}

Let $\mV$ be a presentably monoidal category whose tensor unit is final, $\mC$ a $\mV$-enriched category and $X $ an object of $\mC.$
If $\mC$ admits small weakly contractible conical colimits, then $\mC_{\X/}$ admits small colimits.

\end{lemma}

\begin{remark}
Let $\mV$ be a presentably monoidal category whose tensor unit is final, $\mC, \mD$ be $\mV$-enriched categories, $F: \mC \to \mD$ a $\mV$-enriched functor 
and $Y \to X $ a morphism in $\mC.$
There is an induced $\mV$-enriched functor 
$\mC_{\X/} \to \mD_{F(\Y)/}$ that fits into a commutative square 
$$\begin{xy}
\xymatrix{
\mC_{\X/} \ar[d]^{} \ar[r]
&  \mD_{F(\Y)/} \ar[d] \ar[d]^{}
\\ 
\mC
\ar[r]^F & \mD
}
\end{xy}$$	
of $\mV$-enriched categories.
    
\end{remark}

The following is \cite[Lemma 2.2.5]{oriented}:

\begin{lemma}\label{bien}
Let $\mV$ be a presentably monoidal category whose tensor unit is final, $\mC$ a $\mV$-enriched category and $X \to Y$ a morphism in $\mC.$
If $\mC$ admits conical pushouts, the $\mV$-enriched functor $\mC_{\Y/} \to \mC_{\X/}$ admits a $\mV$-enriched left adjoint.
	
\end{lemma}

The following is \cite[Corollary 2.2.2]{GepnerHeine2026}:

\begin{corollary}\label{susp}

Let $\mV$ be a presentably monoidal category and $\A \in \mV$.
There is a $\mV$-enriched category $S(\A) $, which we call the suspension of $A$, satisfying the following properties:
\begin{enumerate}[\normalfont(1)]\setlength{\itemsep}{-2pt}
\item The space of objects of $S(\A) $ is the set $\{0,1\}.$

\item For every $0 \leq \ell \leq 1$ the unit $\tu \to \Mor_{S(\A)}(\ell,\ell)$ is an equivalence.

\item The morphism object $\Mor_{S(\A)}(1,0)$ is initial.

\item There is an equivalence $ \Mor_{S(\A)}(0,1) \simeq \A$. 

\item For every $\mV$-enriched category $\mC $ and objects $\X,\Y$ of $\mC$
the induced map is an equivalence $$ \Map_{\mV\mathrm{-}\Cat_{B(\tu) \coprod B(\tu)/}}(S(\A), (\mC; \X,\Y)) \to \Map_\mV(\A, \Mor_\mC(\X,\Y)).$$
\end{enumerate}
    
\end{corollary}

\begin{definition}
Let $\mV$ be a presentably monoidal category, $\mC, \mD$ be $\mV$-enriched categories and $X,Y \in \mC, Y,Z \in \mD$
objects. The bipointed wedge $(\mC; X,Y) \vee (\mD; Y,Z) $
is the pushout $ (\mC \coprod_{\{Y\}} \mD; X,Z).$
    
\end{definition}

\begin{corollary}

Let $\mV$ be a presentably monoidal category, $n \geq 1$ and $\A_1, ..., \A_\n \in \mV$.
\begin{enumerate}[\normalfont(1)]\setlength{\itemsep}{-2pt}
\item The space of objects of the $\mV$-enriched category $S(A_1) \vee ...\vee S(A_n)$ is the set $\{0,...,\n\}.$
\item For every $0 \leq \ell \leq \n$ the unit $\tu \to \Mor_{S(A_1) \vee ...\vee S(A_n)}(\ell,\ell)$ is an equivalence.
\item For every $0 \leq \bk < \ell \leq \n$ the morphism object $\Mor_{S(A_1) \vee ...\vee S(A_n)}(\ell,\bk)$ is initial.
\item For every $0 \leq \ell < \n$ there is an equivalence $ \Mor_{S(A_1) \vee ...\vee S(A_n)}(\ell,\ell+1) \simeq \A_{\ell+1}$. 
\item For every $0 \leq \bk < \m \leq \n$ the following induced morphism is an equivalence $$\bigotimes_{\bk \leq \ell < \m}  \A_{\ell+1} \simeq \bigotimes_{\bk \leq \ell < \m} \Mor_{S(A_1) \vee ...\vee S(A_n)}(\ell,\ell+1) \to \Mor_{S(A_1) \vee ...\vee S(A_n)}(\bk,\m).$$
\end{enumerate}
\end{corollary}

\begin{notation}\label{Thetasgen}Let $\mV$ be a small monoidal category. Let $\Theta(\mV) \subset \mV\mathrm{-}\Cat$ be the full subcategory spanned by the wedges of suspensions of objects of $\mV$ and $B\tu_\mV$.

\end{notation}

\begin{definition}

Let $\mV$ be a small monoidal category.
A presheaf on $\Theta(\mV) $ satisfies the Segal condition 
if for every $n \geq 2$ and $X_1,...,X_n \in \mV$
the following canonical map is an equivalence:
$$ F(S(X_1) \vee\cdots\vee S(X_n)) \to F(S(X_1)) \times_{F(B\tu_\mV)}\times\cdots \times_{F(B\tu_\mV)} F(S(X_n)).$$
    
\end{definition}

\begin{definition}Let $\mV$ be a small monoidal category, $F$ a presheaf on $\Theta(\mV)$ and $A,B \in F(B\tu_\mV). $ 
The presheaf of morphisms $A$ to $B$ in $F$ is the following presheaf on $\mV$, where $S: \mV \to \mV\mathrm{-}\Cat_{B\tu_\mV \coprod B\tu_\mV/}$ is the suspension:
$$\Mor_F(A,B):= (F \circ S) \times_{F(B\tu_\mV) \times F(B\tu_\mV)} \{(A,B)\}.$$
    
\end{definition}

The following is \cite[Theorem 2.4.8]{GepnerHeine2026}:

\begin{theorem}\label{denseinherited} Let $\mW$ be a presentably monoidal category and $\mV$ a small dense full monoidal subcategory of $\mW$.
The full subcategory $\Theta(\mV)$ of $\mV\mathrm{-}\Cat$ spanned by the wedges of suspensions of objects of $\mV$ and $B\tu_\mV$ is dense in $\mW\mathrm{-}\Cat.$
A presheaf on $\Theta(\mV) $ belongs to the essential image of the $\Theta(\mV)$-nerve
if and only if it satisfies the Segal condition and 
all morphism presheaves belong to the essential image of the $\mW$-nerve.

\end{theorem}

\subsection{Higher categories}

Our main application of theory of enriched categories will be to $\infty$-categories and oriented categories.
We recall these basic notions in the following two subsections.

\begin{definition}
For every $\n \geq 0$ we inductively define the presentable cartesian closed category $\n\Cat$ of small (not necessarily univalent) $\n$-categories by setting
$$(\n+1)\Cat:= {{\n\Cat}\mathrm{-}\Cat} $$
starting with $$ 0\Cat :=\mS.$$
\end{definition}

\begin{notation}
For every $\n \geq 0$ we inductively define colocalizations
$$\n\Cat \rightleftarrows (\n+1)\Cat: \iota_\n,$$
where both adjoints  preserve finite products and filtered colimits.
Let
$$0\Cat= \mS \rightleftarrows 1\Cat = {\mS\mathrm{-}\Cat} : \iota_0 $$ be the canonical colocalization whose right adjoint assigns the space of objects. Let $$(\n+1)\Cat= {{\n\Cat}\mathrm{-}\Cat} \rightleftarrows (\n+2)\Cat= {{(\n+1)\Cat}\mathrm{-}\Cat}:\iota_{\n+1}:= (\iota_\n)_! $$
be the induced adjunction.	
	
\end{notation}

\begin{definition}The presentable category $\infty\Cat$ of small (non-univalent) $\infty$-categories is the limit
$$\infty\Cat:= \lim(\cdots\xrightarrow{\iota_{\n}} \n \Cat \xrightarrow{\iota_{\n-1}}\cdots \xrightarrow{\iota_0} 0 \Cat) $$
of presentable categories and right adjoint functors.

\end{definition}

The next proposition follows from the fact that the functor ${(-)\mathrm{-}\Cat}$ preserves small limits:

\begin{proposition}\label{fix}
	
There is a canonical equivalence
$$ \infty\Cat \simeq {\infty\Cat}\mathrm{-}\Cat. $$
	
\end{proposition}

\begin{notation}
Let $\partial\bD^1$ denote the two element set $\{0,1\}$.
\end{notation}

\begin{notation}
Let $\Mor: \infty\Cat_{\partial\bD^1/} \to \infty\Cat$
be the canonical functor $$\infty\Cat_{\partial\bD^1/} \simeq \mS_{\partial\bD^1/}\times_\mS {\infty\Cat}\mathrm{-}\Cat \to \infty\Cat $$
sending $(\mC,\X,\Y)$ to $\Mor_\mC(\X,\Y).$
\end{notation}

\begin{remark}\label{homfil}
The functor $\Mor: \infty\Cat_{\partial\bD^1/} \to \infty\Cat$
preserves small filtered colimits and limits.
	
\end{remark}

\begin{definition}

A functor $X \to Y$ of $\infty$-categories is an inclusion -- or subcategory inclusion-- if is a monomorphism in the category $\infty\Cat,$ i.e. for every $\infty$-category $Z$ the induced map $$\Map_{\infty\Cat}(Z,X) \to \Map_{\infty\Cat}(Z,Y)$$ is an embedding.
In this case we also say that $X$ is a subcategory of $Y$.
 
\end{definition}

\cref{monochar} implies the following:

\begin{corollary}

A functor $\phi: X \to Y$ of $\infty$-categories is an inclusion if and only if it induces an embedding $\iota_0(X) \to \iota_0(Y)$ on underlying spaces and for every $A,B \in X$ the induced functor
$$\Mor_X(A, B) \to \Mor_Y(\phi(A), \phi(B))$$ is an inclusion.

\end{corollary}

\begin{definition}\label{ruik} Let $\n \geq 0.$
We inductively define involutions $$(-)^\op_\n, (-)^\co_\n: \n\Cat \to \n\Cat$$ by setting
$(-)^\co_0, (-)^\op_0:  0\Cat \to 0\Cat$ are the identities, $$(-)^\op_{\n+1}: {(\n+1)}\Cat \xrightarrow{(-)^\circ} {(\n+1)}\Cat \xrightarrow{((-)^\co_\n)_!}{(\n+1)}\Cat, $$ $$(-)^\co_{\n+1}:=((-)^\op_\n)_!: {(\n+1)}\Cat \xrightarrow{ }{(\n+1)}\Cat.$$
There are commutative squares:
$$\begin{xy}
\xymatrix{
{(\n+1)}\Cat \ar[d]^{\iota_\n}  \ar[r]^{(-)_{\n+1}^\op} & {(\n+1)}\Cat  \ar[d]^{\iota_\n}
\\ 
{\n}\Cat \ar[r]^{(-)_{\n}^\op} & {\n}\Cat
}
\end{xy}
\qquad
\begin{xy}
\xymatrix{
{(\n+1)}\Cat \ar[d]^{\iota_\n}  \ar[r]^{(-)_{\n+1}^\co} & {(\n+1)}\Cat  \ar[d]^{\iota_\n}
\\ 
{\n}\Cat \ar[r]^{(-)_{\n}^\co} & {\n}\Cat
}
\end{xy}
$$
and so induced involutions on the limit $$(-)^\op, (-)^\co: \infty\Cat \to \infty\Cat. $$

\end{definition}

\begin{remark}\label{oppo} By \cref{ruik} there are commutative squares, where $\sigma$ permutes the distinguished objects:

$$\begin{xy}
\xymatrix{
\infty\Cat_{\partial\bD^1/} \ar[d]^{\Mor}  \ar[r]^{(-)^\co} & \infty\Cat_{\partial\bD^1/}  \ar[d]^{\Mor}
\\ 
\infty\Cat \ar[r]^{(-)^\op} & \infty\Cat
}
\end{xy}\qquad
\begin{xy}
\xymatrix{
\infty\Cat_{\partial\bD^1/} \ar[d]^{\Mor}  \ar[r]^{\sigma \circ (-)^\op} & \infty\Cat_{\partial\bD^1/}  \ar[d]^{\Mor}
\\ 
\infty\Cat \ar[r]^{(-)^\co} & \infty\Cat.
}
\end{xy}
$$

\end{remark}

The following follows from \cref{susp}:

\begin{proposition}\label{suspi}

The functor $\Mor: \infty\Cat_{\partial\bD^1/} \to \infty\Cat$ admits a left adjoint
$S: \infty\Cat \to \infty\Cat_{\partial\bD^1/}$ such that for every $\infty$-category $\mC$ the $\infty$-category $S(\mC),$ called categorical suspension of $\mC$, has two objects 0,1, and morphism $\infty$-categories:
$$\Mor_{S(\mC)}(1,0)\simeq \emptyset, \ \Mor_{S(\mC)}(0,1)\simeq \mC, \ \Mor_{S(\mC)}(0,0)\simeq \Mor_{S(\mC)}(1,1) \simeq *.$$
    
\end{proposition}

\begin{notation}
For every $0 \leq \n \leq \m$ the left adjoint embeddings $\n\Cat \leftrightarrows \m\Cat$ preserve small limits and thus induce a left adjoint embedding $\n\Cat \leftrightarrows \infty\Cat: \iota_\n$ that preserves small limits
and so admits a left adjoint $\tau_\n: \infty\Cat \to \n\Cat$ by presentability.

\end{notation}

\begin{remark}
The $\iota_n$ collectively filter every $\infty$-category $X$ as a colimit $X\simeq\colim_{n}\iota_n(X)$.
\end{remark}

The next is \cite[Lemma 2.4.16.]{GepnerHeine2026}:

\begin{lemma}\label{carclo}
The presentable category $\infty\Cat$ is cartesian closed.
\end{lemma}

\begin{notation}

For any $\infty$-category $\mC$ let 
$\Fun(\mC,-): \infty\Cat \to \infty\Cat$ be the right adjoint of the functor $(-) \times \mC:\infty\Cat \to \infty\Cat.$
    
\end{notation}

\begin{remark}
Since $\infty\Cat$ is cartesian closed, it refines to an $\infty$-category $\infty\scat$ such that for every two objects $X$ and $Y$:
\[
\Mor_{\infty\scat}(X,Y)=\Fun(X,Y).
\]
\end{remark}

\begin{definition}Let $\n \geq 0$.
The $\n$-disk is the $n$-fold iterated suspension $\bD^\n:= S^{\n}(\bD^0)$ of the terminal $\infty$-category $\bD^0$.
\end{definition}

\begin{definition}Let $\n \geq 0$.
The boundary of the $\n$-disk is the $n$-fold iterated suspension $\partial\bD^\n:= S^{\n}(\emptyset)$ of the initial $\infty$-category $\emptyset$.
	
\end{definition}

\begin{remark}Let $\n \geq 0.$
The functor $\emptyset=\partial\bD^0 \subset \bD^0=*$ induces an inclusion $\partial\bD^\n \subset \bD^\n$.
	
\end{remark}

\begin{definition}
The bipointed wedge of an ordered pair $(\mC,\mD)$ of bipointed $\infty$-categories  $\mC\in\infty\Cat_{/\partial\bD^1}$ and $\mD\in\infty\Cat_{/\partial\bD^1}$ is the bipointed $\infty$-category obtained by taking the pushout along the last basepoint of $\mC$ and the first basepoint of $\mD$.
\end{definition}

\begin{definition}\label{Theta}
	
Let $\Theta' \subset \infty\Cat_{\partial\bD^1/}$ be the full subcategory generated by $\bD^0$ under suspensions and bipointed wedges and $\Theta \subset \infty\Cat$ the essential image of $\Theta'$ under the forgetful functor.	
	
\end{definition}

\begin{remark}

An $\infty$-category belongs to $\Theta$ if and only if it is of the form $$ \bD^{i_0} \coprod_{\bD^{j_1}} \bD^{i_1} \coprod_{\bD^{j_2}} \cdots \coprod_{\bD^{j_n}} \bD^{i_n} $$
for any sequence of natural numbers $n, i_0,\ldots, i_n, j_1,\ldots, j_n$ and monomorphisms 
$ \bD^{j_\ell} \rightarrowtail \bD^{i_\ell},\bD^{j_\ell} \rightarrowtail \bD^{i_{\ell-1}}$, $ 1 \leq \ell \leq n$.

\end{remark}

The next result is \cite[Theorem 2.5.8]{GepnerHeine2026}:

\begin{theorem}\label{theta}

The restricted Yoneda embedding $\N:
\infty\Cat \to \Fun(\Theta^\op,\mS)$ is fully faithful and admit left adjoints that preserves finite products.
A $\Theta$-space is in the essential image of the restricted Yoneda embedding if and only if it satisfies the Segal condition, i.e. it is local with respect to the map 
$$ \N(\bD^{i_0}) \coprod_{\N(\bD^{j_1})} \N(\bD^{i_1})\coprod_{\N(\bD^{j_2})} \cdots \N(\coprod_{\bD^{j_n}}) \N(\bD^{i_n}) \to \N(\bD^{i_0} \coprod_{\bD^{j_1}} \bD^{i_1} \coprod_{\bD^{j_2}} \cdots \coprod_{\bD^{j_n}} \bD^{i_n}) $$
for any sequence of natural numbers $n, i_0,\ldots, i_n, j_1,\ldots, j_n$ and monomorphisms 
$ \bD^{j_\ell} \rightarrowtail \bD^{i_\ell},\bD^{j_\ell} \rightarrowtail \bD^{i_{\ell-1}}$, $ 1 \leq \ell \leq n$.
For every $ 0 \leq n \leq \infty$ the full subcategory $n\Cat \subset \infty\Cat$ is generated under small colimits by the disk of dimension less than or equal to $n$, or simply the $n$-disk itself (since the disks of strictly lower dimension are retracts of the $n$-disks).

\end{theorem}

By \cite[Definition 2.6.25]{GepnerHeine2026} for every $n \geq 0$ there is an oriented $n$-cube $\cube^n$
whose 1-truncation is $(\bD^1)^{\times n}.$

\begin{notation}\label{sqGray}

Let $\cube \subset \infty\Cat$ be the full subcategory of oriented cubes.
\end{notation}

The following is \cite[Theorem 2.8.12]{GepnerHeine2026}:

\begin{theorem}\label{cubedense}

The full subcategory $\cube \subset \infty\Cat$ is dense.
    
\end{theorem}

By \cite[Definition 2.6.32]{GepnerHeine2026} for every $n \geq 0$ there is an oriented $n$-simplex $\bDelta^n$
whose 1-truncation is the totally ordered set $[n]= \{0 < ...< n \}.$

\begin{notation}

Let $\bDelta \subset\infty\Cat$ be the full subcategory spanned by the oriented simplices.
  
Let $\bDelta^+ \subset\infty\Cat$ be the full subcategory spanned by the oriented simplices and the initial object.
  
\end{notation}

The next is \cite[Theorem 2.7.8]{GepnerHeine2026}:

\begin{theorem}\label{orientdense}

The full subcategory $\bDelta\subset\infty\Cat$ is dense.
    
\end{theorem}

In the following we introduce the Gray tensor product
and categories enriched in the Gray tensor product.

By \cite[Definition 2.6.23]{GepnerHeine2026} there is a canonical monoidal structure on $\cube$ whose tensor unit is
$\bD^0$ and such that $\cube^n \boxtimes \cube^m= \cube^{n+m}$ for every $n,m \geq 0.$

We have the following result of Campion \cite{campion2022cubesdenseinftyinftycategories}, which is also proven in \cite[Corollary 3.4.2]{GepnerHeine2026}:

\begin{corollary}\label{locmon2}
Let $\cube \subset \infty\Cat$ denote the full subcategory of oriented cubes, equipped with the Gray monoidal structure.
The convolution monoidal structure descends along 
the localization $$\mP(\cube) \rightleftarrows \infty\Cat.$$

\end{corollary}

\begin{notation}
	
Since the Gray tensor product defines a presentably monoidal structure on $\infty\Cat$, it is closed: for every $\infty$-category $\mC$ the functor $$ \mC \boxtimes (-): \infty\Cat \to \infty\Cat$$ admits a right adjoint $\Fun^\lax(\mC,-)$, and the functor $$ (-) \boxtimes \mC : \infty\Cat \to \infty\Cat$$ admits a right adjoint $\Fun^\oplax(\mC,-)$.
\end{notation}

\begin{definition}Let $\F,\G:\mC \to \mD $ be functors of $\infty$-categories.
\begin{enumerate}[\normalfont(1)]\setlength{\itemsep}{-2pt}
\item A lax natural transformation $\F \to \G$ is a morphism in $\Fun^\lax(\mC,\mD)$.

\item An oplax natural transformation $\F \to \G$ is a morphism in $\Fun^\oplax(\mC,\mD)$.
\end{enumerate}
\end{definition}

\begin{remark}\label{funcompa}
	
Let $\mC,\mD$ be $\infty$-categories. The canonical functors $\mC \to *, \mD \to *$
give rise to functors $\mC \boxtimes \mD \to \mC \boxtimes * \simeq \mC$ and $\mC \boxtimes \mD \to * \boxtimes \mD \simeq \mD$ and so to a functor $\mC \boxtimes \mD \to \mC \times \mD.$
By adjointess the latter functor induces functors
$ \Fun(\mC,\mD)\to \Fun^\lax(\mC,\mD)$ and $\Fun(\mC,\mD)\to \Fun^\oplax(\mC,\mD).$
\end{remark}

\begin{remark}\label{lao}\label{grayspace}
If $\mC$ is an $\n$-category and $\mD$ an $\m$-category for $\n,\m \geq 0$,
then $\mC \boxtimes \mD$ is an $\n+\m$-category.
This holds since $\n\Cat$ is closed under small colimits in $\infty\Cat$ and $\n\Cat$ is generated under small colimits by the oriented $\ell$-cubes for $1 \leq \ell \leq \n$.
For $n=m=0$ one finds that the Gray-monoidal structure restricts to the 
full subcategory of spaces $\mS \subset \infty\Cat$.
The restricted Gray-monoidal structure on $\mS$ is the cartesian structure since $\mS$ is generated in $\infty\Cat$ under small colimits by the final category, the tensor unit for the Gray tensor product and the cartesian product.
Moreover the left adjoint $\tau_0: \infty\Cat \to \mS$ of the monoidal embedding $\mS \subset \infty\Cat$ is monoidal since
$\infty\Cat$ is generated by the oriented cubes under small colimits and the image of any oriented cube under $\tau_0$ is contractible.

\end{remark}

The following is \cite[Proposition 3.4.8]{GepnerHeine2026}: 
\begin{proposition}\label{dua}
	
There are canonical monoidal involutions
$$(-)^\op, (-)^\co: (\infty\Cat, \boxtimes) \simeq (\infty\Cat, \boxtimes)^\rev $$
refining the involutions of \cref{ruik}.
\end{proposition}

\begin{corollary}\label{grayhoms}
Let $\mC,\mD $ be $\infty$-categories.
There are canonical equivalences $$ \Fun^\oplax(\mC,\mD)^\op \simeq \Fun^\lax(\mC^\op,\mD^\op), $$$$\Fun^\oplax(\mC,\mD)^\co \simeq \Fun^\lax(\mC^\co,\mD^\co).$$	
\end{corollary}

Next we define join and slice.

\begin{notation} For every small category $\mC$ that admits an initial object let $ \mP_{\mathrm{red}}(\mC) \subset \mP(\mC) $  
be the full subcategory of presheaves on $\mC$
that are reduced, i.e. send the initial object to the final space.
\end{notation}

\begin{notation}

Let $\bDelta^+\subset\infty\Cat$ denote the full subcategory consisting of the oriented simplices and the initial object. 

\end{notation}

By \cite[Definition 2.6.30]{GepnerHeine2026} the category $\bDelta^+$ carries a monoidal structure, the join,
whose tensor unit is the empty $\infty$-category and such that $$\bDelta^n \star \bDelta^m = \bDelta^{n+m+1}$$ for every $n,m \geq -1.$
The following is \cite[Corollary 3.5.8.]{GepnerHeine2026}:

\begin{theorem}\label{locmon4}

The join monoidal structure descends along the localization $\mP_{\mathrm{red}}(\bDelta^+) \rightleftarrows \infty\Cat. $

\end{theorem}

We also define the antijoin:

\begin{definition}
Let $X, Y \in \infty\Cat.$ The antijoin of $X, Y$ is
$$X \bar{\star} Y := (X^{\co} \star Y^{\co})^{\co} .$$

\end{definition}

\begin{definition}Let $X$ be a small $\infty$-category.
\begin{enumerate}[\normalfont(1)]\setlength{\itemsep}{-2pt}
\item The oplax slice, or oplax over $\infty$-category, functor is the right adjoint $$ \infty\Cat_{X/ } \to \infty\Cat, \qquad (F:X \to Y) \mapsto Y_{//^\oplax F}$$ of the functor $ (-) \star X: \infty\Cat \to \infty\Cat_{X/ }.$

\item The lax coslice, or lax under $\infty$-category, functor is the right adjoint $$ \infty\Cat_{X/ } \to \infty\Cat,\qquad (F:X \to Y) \mapsto Y_{F//^\lax }$$ of the functor $ X \star (-): \infty\Cat \to \infty\Cat_{\X / }.$

\item The lax slice, or lax over $\infty$-category, functor is the right adjoint $$ \infty\Cat_{\X / } \to \infty\Cat,\qquad (F:X \to Y) \mapsto Y_{//^\lax F} := (Y^\co_{//^\oplax F^\co})^\co $$
of the functor $ (-) \bar{\star} X: \infty\Cat \to \infty\Cat_{X/ }.$

\item The oplax coslice, or oplax under $\infty$-category, functor is the right adjoint $$ \infty\Cat_{X/ } \to \infty\Cat,\qquad (F:X \to Y) \mapsto Y_{F//^\oplax}:=(Y^\co_{F^\co//^\lax })^\co $$
of the functor $ X \bar{\star} (-): \infty\Cat \to \infty\Cat_{\X / }.$

\end{enumerate}

\end{definition}

\begin{notation}Let $F: X \to Y $ be a functor.
If unspecified, $Y_{//F}$ will refer to the oplax slice, and $Y_{F//}$ will refer to the oplax coslice.
\end{notation}

The following is \cite[Lemma 3.7.13]{GepnerHeine2026}:

\begin{lemma}Let $F: X \to Y $ be a functor.
There is a canonical equivalence of $\infty$-categories
$$ (Y_{//^\oplax F})^\op \simeq (Y^\op)_{F^\op//^\lax }.$$
\end{lemma}

In the following we apply the theory of bienriched $\infty$-categories to the Gray monoidal structure on $\infty\Cat$ \cite{oriented}, \cite{heine2026stable}.
See \cite{heine2024bienriched} for a detailed discussion of bienriched category theory.

\begin{definition}
\begin{enumerate}[\normalfont(1)]\setlength{\itemsep}{-2pt}
\item An oriented category is a category right enriched in $(\infty\Cat,\boxtimes).$

\item An oriented functor is a functor right enriched in $(\infty\Cat,\boxtimes).$

\item An antioriented category is a category left enriched in $(\infty\Cat,\boxtimes).$

\item An antioriented functor is a functor left enriched in $(\infty\Cat,\boxtimes).$

\item A bioriented category is a category bienriched in  $((\infty\Cat,\boxtimes), (\infty\Cat,\boxtimes))$.

\item A bioriented functor is a functor bienriched in  $((\infty\Cat,\boxtimes), (\infty\Cat,\boxtimes))$.

\end{enumerate}

\end{definition}

\begin{notation}
\begin{enumerate}[\normalfont(1)]\setlength{\itemsep}{-2pt}
\item Let $\mC$ be an oriented category and $\X,\Y \in \mC.$	
By the structure of a right $(\infty\Cat,\boxtimes)$-enriched category there is a right morphism $\infty$-category $\R\Mor_\mC(\X,\Y)$.

\item Let $\mC$ be an antioriented category and $\X,\Y \in \mC.$	
By the structure of a left $(\infty\Cat,\boxtimes)$-enriched category there
is a left morphism $\infty$-category $\L\Mor_\mC(\X,\Y)$.

\item Let $\mC$ be a bioriented category and $\X,\Y \in \mC.$	
By the structure of a $(\infty\Cat,\boxtimes)$-bienriched category there
is a morphism $\infty$-category $\Mor_\mC(\X,\Y) \in \infty\Cat \otimes \infty\Cat.$

\end{enumerate}
	
\end{notation}

\begin{example}
The Gray monoidal structure on $\infty\Cat$ is closed and so endows $\infty\Cat$ as bienriched in $(\infty\Cat,\boxtimes).$ This way we see
$\infty\Cat$ as a large bioriented category, which we denote by $ \infty\fcat.$
Every full subcategory of $\infty\Cat$ inherits the structure of a bioriented category.

\end{example}	

\begin{definition}
We refer to morphims in $\boxtimes\Cat$ as {\em antioriented functors}, to morphims in $\Cat\boxtimes$ as {\em oriented functors}, and to morphims in $\boxtimes\Cat\boxtimes$ as {\em bioriented functors}.
\end{definition}

\begin{notation}

Let $$\Cat\boxtimes,\qquad \boxtimes\Cat,\qquad \boxtimes\Cat\boxtimes$$
denote the categories of oriented categories, antioriented categories and bioriented categories, respectively.
\end{notation}

\begin{remark}
An oriented, antioriented, or bioriented category is {\em presentable} if the respective left, right, or bienriched category is presentable in the sense of enriched $\infty$-category theory, which is likewise a presentable category endowed with a closed left action, closed right action, or closed biaction of the enriching monoidal category or categories, respectively.
\end{remark}

\begin{notation}
\begin{enumerate}[\normalfont(1)]\setlength{\itemsep}{-2pt}
\item Let $\mC,\mD \in \boxtimes\Cat$. We write $${\boxtimes\Fun}(\mC,\mD)$$ for the category of antioriented functors $\mC \to \mD.$	
		
\item Let $\mC,\mD \in \Cat\boxtimes$. We write $${\Fun\boxtimes}(\mC,\mD)$$ for the category of oriented functors $\mC \to \mD.$	
		
\item Let $\mC,\mD \in \boxtimes\Cat\boxtimes $. We write $${\boxtimes\Fun\boxtimes}(\mC,\mD)$$ for the category of bioriented functors $\mC \to \mD.$		

\end{enumerate}
\end{notation}

\begin{remark}
We refer to adjunctions of oriented, antioriented, and bioriented categories as oriented, antioriented, or bioriented adjunctions.
\end{remark}

\begin{remark}\label{adj2}
\begin{enumerate}[\normalfont(1)]\setlength{\itemsep}{-2pt}
\item An antioriented (oriented) functor $\mC \to \mD$ admits a right adjoint if and only if it preserves left (right) tensors and the underlying functor admits a right adjoint.

\item A bioriented functor $\mC \to \mD$ admits a right adjoint if and only if it preserves left and right tensors and the underlying functor admits a right adjoint.

\item An antioriented (oriented) functor $\mC \to \mD$ admits a left adjoint if and only if it preserves left (right) cotensors and the underlying functor admits a left adjoint.
	
\item A bioriented functor $\mC \to \mD$ admits a left adjoint if and only if it preserves left and right cotensors and the underlying functor admits a left adjoint.

\end{enumerate}

\end{remark}

Next we define the appropriate notions of opposite oriented category.

\begin{definition}
    
Let
\begin{align*}
&(-)^\circ: {\Cat\boxtimes} \simeq {\boxtimes\Cat}\\
&(-)^\circ: {\boxtimes\Cat} \simeq {\Cat\boxtimes}\\
&(-)^\circ:  {\boxtimes\Cat\boxtimes} \simeq {\boxtimes\Cat\boxtimes}
\end{align*}
be the opposite enriched category involutions.	
	
\end{definition}

\begin{definition}
We define the following involutions, which reverse the even dimensional cells.
Note that we are also forced to reverse the orientation.
The equivalences of \cref{dua} give rise to the equivalences
\begin{align*}
&(-)^\co:= (-)^\op_!: {\boxtimes\Cat} \simeq {\Cat\boxtimes}\\
&(-)^\co:= (-)^\op_!: {\Cat\boxtimes} \simeq {\boxtimes\Cat}\\
&(-)^\co:= ((-)^\op, (-)^\op)_!: {\boxtimes\Cat\boxtimes} \simeq {\boxtimes\Cat\boxtimes}.
\end{align*}
\end{definition}

\begin{definition}
We define the following involutions, which reverse the odd dimensional cells.
Note that we are also forced to keep the orientation.
\begin{align*}
&(-)^\op:= (-)^\circ\circ (-)^\co_!: {\boxtimes\Cat} \simeq {\boxtimes\Cat}\\
&(-)^\op:= (-)^\circ\circ (-)^\co_!: {\Cat\boxtimes} \simeq {\Cat\boxtimes}\\
&(-)^\op :=(-)^\circ \circ ((-)^\co, (-)^\co)_!: {\boxtimes\Cat\boxtimes} \simeq {\boxtimes\Cat\boxtimes}
\end{align*}
\end{definition}

\begin{definition}
Combining the latter two types of involutions gives rise to the following sort of involution, which reverses all cells:
\begin{align*}
&{(-)^{\co\op}}:= {(-)^{\co}} \circ {(-)^{\op}} \simeq {(-)^{\op}} \circ{(-)^{\co}}:{\boxtimes\Cat} \simeq {\Cat\boxtimes}\\
&{(-)^{\co\op}}:= {(-)^{\co}} \circ {(-)^{\op}} \simeq {(-)^{\op}} \circ {(-)^{\co}}:{\Cat\boxtimes} \simeq {\boxtimes\Cat}\\
&{(-)^{\co\op}}:= {(-)^{\co}} \circ {(-)^{\op}} \simeq {(-)^{\op}} \circ {(-)^{\co}}: {\boxtimes\Cat\boxtimes} \simeq {\boxtimes\Cat\boxtimes},
\end{align*}
where all equivalences are involutions.
\end{definition}

Next we define oriented pushouts and oriented pullbacks in oriented categories and the dual notions of antioriented pushouts and antioriented pullbacks in antioriented categories.

\begin{definition}
Let $\mC$ be an oriented category.
\begin{enumerate}[\normalfont(1)]\setlength{\itemsep}{-2pt}
\item 
The oriented pullback of the diagram $X \to Z\leftarrow Y$ in $\mC$ is a diagram
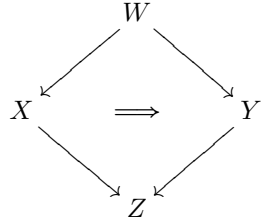
\[
\xymatrix{& W\ar[rd]\ar[ld] &\\
X\ar[rd] & \Longrightarrow & Y\ar[ld]\\
& Z &}
\]
in $\mC$ such that for all objects $T$ of $\mC$ the induced functor
\[
\R\Mor_{\mC}(T,W)\to \R\Mor_{\mC}(T,X) \underset{\R\Mor_{\mC}(T,Z)}{\times}
\Fun^\oplax(\bD^1,\R\Mor_{\mC}(T,Z)) \underset{\R\Mor_{\mC}(T,Z)}{\times} \R\Mor_{\mC}(T,Y)
\]
is an equivalence.
In this case, we will write $X\underset{Z}{\vec{\times}} Y$ or $ Y\underset{Z}{{\cev\times}} X$ for $W.$

\item 
The oriented pushout of the oriented diagram $X \leftarrow Z\to Y$ in $\mC$ is the oriented pullback of the corresponding diagram in the oriented category $\mC^\op,$ which we denote by $ X\underset{Z}{{\vec{+}}} Y $ or $Y\underset{Z}{\cev{+}} X$.

\end{enumerate}

\end{definition}

\begin{remark}
An oriented pullback square in an oriented category $\mC$ is a diagram $\cube^2\to\mC$ which satisfies the universal property above.
\end{remark}

\begin{definition}
Let $\mC$ be an antioriented category.
\begin{enumerate}[\normalfont(1)]\setlength{\itemsep}{-2pt}
\item 
The antioriented pullback of the diagram $X \to Z\leftarrow Y$ in $\mC$ is the oriented pushout of the corresponding diagram in the oriented category $\mC^\circ$, which we denote by $X\underset{Z}{\bar{\vec{\times}}} Y$ or $ Y\underset{Z}{{\bar{\cev\times}}} X$.

\item 
The antioriented pushout of the diagram $X \leftarrow Z\to Y$ in $\mC$
is the oriented pullback of the corresponding diagram in the oriented category $\mC^\circ$, which we denote by $ X\underset{Z}{{\bar{\vec{+}}}} Y $ or $Y\underset{Z}{\bar{\cev{+}}} X.$

\end{enumerate}

\end{definition}

\begin{definition}Let $\mC$ be a bioriented category.
\begin{enumerate}[\normalfont(1)]\setlength{\itemsep}{-2pt}

\item The (anti)oriented pullback of a diagram $A \to C\leftarrow B$ in $\mC$ is the (anti)oriented pullback in the underlying (anti)oriented category of $\mC$.

\item The (anti)oriented pushout of a diagram $A \leftarrow C \to B$ in $\mC$ is the (anti)oriented pushout in the underlying (anti)oriented category of $\mC$.

\end{enumerate}

\end{definition}

\begin{remark}
Note that the oriented pullback is not symmetric since it makes use of a specified ordering of the sources of the maps $A\to C$ and $B\to C$.
\end{remark}

The following is \cite[Lemma 3.8.6.]{oriented}:

\begin{lemma}\label{0desc}\label{adesc}
Let $\mC$ be an oriented category.
\begin{enumerate}[\normalfont(1)]\setlength{\itemsep}{-2pt}
\item Let $A \leftarrow C \to \B$ be morphisms in $\mC$. If $\mC$ admits pushouts and right tensors with $\bD^1$, 
there is a canonical equivalence $$\A\,\underset{C}{\vec{+}}\, \B \simeq \A\!\!\underset{C \otimes \{0\}}{+} \!\!(\C \ot \bD^1) \!\!\underset{C \otimes \{1\}}{+}\!\!\B.$$	

\item Let $A \to C\leftarrow B$ be morphisms in $\mC$. If $\mC$ admits pullbacks and right cotensors with $\bD^1$, 
there is a canonical equivalence $${\A \,\underset{\C}{\vec{\times}}}\, \B \simeq \A \!\!\underset{\C^{\{0\}}}{\times} \!\!{\C^{\bD^1}}\!\!\underset{{\C^{\{1\}}}}{\times}  \!\!\B.$$	
\end{enumerate}

\end{lemma}

Dually, we obtain the following:

\begin{corollary}\label{bdesc}
Let $\mC$ be an antioriented category.
\begin{enumerate}[\normalfont(1)]\setlength{\itemsep}{-2pt}
\item 
Let $A \leftarrow C \to \B$ be morphisms in $\mC$. If $\mC$ admits pushouts and left tensors with $\bD^1$, 
there is a canonical equivalence $$\A\,\underset{C}{\bar{\vec{+}}}\, \B \simeq \A \underset{\{0\}\otimes\C}{+} ( \bD^1\ot \C) \underset{\{1\}\otimes\C}{+} \B.$$	

\item Let $A \to C\leftarrow B$ be morphisms in $\mC$. If $\mC$ admits pullbacks and left cotensors with $\bD^1$, 
there is a canonical equivalence $$\A\underset{C}{\bar{\vec{\times}}} \B \simeq \A \underset{^{\{0\}}\C}{\times}{^{\bD^1}\C} \underset{^{\{1\}}\C}{\times} \B.$$	
\end{enumerate}

\end{corollary}

We have the following pasting law, which is \cite[Lemma 3.8.11.]{oriented}:

\begin{lemma}\label{pasting}

Consider the following diagram in any oriented category $\mC$, where the left hand square is a commutative square:
\[
\begin{tikzcd}
\Q \ar{d} \ar{r} & \P \ar{r}{} \ar{d}[swap]{} & \B  \ar{d}{} \\
\E \ar{r} & \A \ar[double]{ur}{}  \ar{r}[swap]{} & \C
\end{tikzcd}
\]
If the right hand square is an oriented pullback square, the left hand square is a pullback square if and only if the outer square is an oriented pullback square.

\end{lemma}

The following is \cite[Proposition 3.8.12.]{oriented}:

\begin{proposition}\label{homs} 
Let $\F: \mA \to \mC,\G: \mB \to \mC$ be functors and $\A,\A'\in \mA, B,B' \in \mB$ and $\sigma: \F(\A)\to \G(B), \sigma': \F(\A')\to \G(B')$ morphisms. There is a canonical equivalence $$\Mor_{{\mA \,\underset{\mC}{\vec{\times}}}\, \mB}((\A,B, \sigma),(\A',B', \sigma')) \simeq {\Mor_\mB(B,B') \,\underset{\Mor_\mC(\F(\A),\G(B'))}{\vec{\times}}}\, \Mor_{\mA}(\A,\A').$$

\end{proposition}

\begin{corollary}\label{homso}

Let $\F: \mA \to \mC,\G: \mB \to \mC$ be functors and $\A,\A'\in \mA, B,B' \in \mB$ and $\sigma: \F(\A)\to \G(B), \sigma': \F(\A')\to \G(B')$ morphisms. There is a canonical equivalence $$\Mor_{\mA\underset{\mC}{\bar{\vec{\times}}} \mB}((\A,B, \sigma),(\A',B', \sigma')) \simeq \Mor_\mA(\A,\A')\underset{\Mor_\mC(\F(\A),\G(B'))}{\bar{\vec{\times}}} \Mor_\mB(B,B').$$

\end{corollary}

\begin{corollary}\label{homs2}\label{homso2} Let $\mB$ be an $\infty$-category and $\sigma: \A \to B, \sigma': \A' \to B'$ morphisms in $\mB$. There are canonical equivalences $$\Mor_{\Fun^\lax(\bD^1,\mB)}(\sigma, \sigma') \simeq \Mor_\mA(\A,\A')\underset{\Mor_\mC(\A,B')}{\bar{\vec{\times}}} \Mor_\mB(B,B'), $$
$$\Mor_{\Fun^\oplax(\bD^1,\mB)}(\sigma, \sigma') \simeq {\Mor_\mB(B,B') \,\underset{\Mor_\mC(\A,B')}{\vec{\times}}}\, \Mor_{\mA}(\A,\A'). $$\end{corollary}

The following is \cite[Corollary 2.3.25.]{gepner2026fibrations}:

\begin{corollary}\label{dimensio}

Let $ n \geq 0$ and $\mA,\mB,\mC$ be $n$-categories.
Let $\F: \mA \to \mC,\G: \mB \to \mC$ be functors.
The oriented pullback $\mA\underset{\mC}{\vec{\times}} \mB$
is an $n$-category.
    
\end{corollary}

The following is \cite[Theorem 4.2.7.]{oriented}.

\begin{theorem}\label{interchange}
The induced functor $$\infty\Cat\simeq\Cat_{\infty\Cat} \to {\Cat\boxtimes} $$ of 2-categoriesis fully faithful. The essential image precisely consists of the oriented categories $\mC$ such that 
for every $X,Y,Z \in \mC$ and functors $\bD^n \to \R\Mor_\mC(Y,Z)$ and $\bD^m \to \R\Mor_\mC(X,Y)$
the functor $$\bD^n \boxtimes \bD^m \to \R\Mor_\mC(Y,Z) \boxtimes \R\Mor_\mC(X,Y) \to \R\Mor_\mC(X,Z)$$ factors through the epimorphism
$\bD^n \boxtimes \bD^m \to \bD^n \times \bD^m.$ 
There is a similar statement for antioriented categories.
\end{theorem}

The following is \cite[Theorem 4.1.9.]{oriented}, where we use the functors of \cref{funcompa}.

\begin{theorem}\label{funcompo}

Let $X,Y$ be $\infty$-categories. The canonical functors
$$ \Fun(X, Y) \to \Fun^\lax(X, Y), \ \Fun(X, Y) \to \Fun^\oplax(X, Y)$$
are inclusions.
    
\end{theorem}

The following is \cite[Theorem 4.2.8.]{oriented}.

\begin{theorem}\label{presheaves}
Let $\mD$ be an $\infty$-category.
The following canonical functors induce fully faithful functors on underlying 1-categories:

$$\Fun(\mD,\infty\scat) \to {\boxtimes\Fun}(\mD,\infty\fcat) $$
and 
$$\Fun(\mD,\infty\scat) \to {\Fun\boxtimes}(\mD,\infty\fcat). $$
\end{theorem}

\subsection{\mbox{Fibrations of higher categories}}

\begin{notation}Let $\phi: \mC \to \mD$ be a functor and $X,Y \in \mC.$ 
We write
\begin{align*}
\phi_{X,Y}:\Mor_\mC(X,Y)\to\Mor_\mD(\phi(X),\phi(Y))
\end{align*}
for the induced functor on morphism $\infty$-categories.
\end{notation}

\begin{definition}Let $1 \leq \n \leq \infty$ and $\phi: \mC \to \mD$ a functor of $\infty$-categories.

\begin{enumerate}[\normalfont(1)]\setlength{\itemsep}{-2pt}

\item A 1-morphism $f: X \to Y$ in $\mC$ is $\phi$-cocartesian if for every $Z\in \mC$ the commutative square 
\begin{equation}\label{filler}
\begin{xy}
\xymatrix{
\Mor_\mC(Y,Z) \ar[d] \ar[r]
& \Mor_\mC(X,Z) \ar[d]^\phi
\\ 
\Mor_\mD(\phi(Y), \phi(Z)) \ar[r] & \Mor_\mD(\phi(X), \phi(Z))
}
\end{xy}\end{equation}
is a pullback square.

\item An $n$-morphism $\alpha:\bD^n\to\mC$ is $\phi$-cocartesian if
for every pair of morphisms $(X \to \alpha(0), \alpha(1) \to Y)$ in $\mC$ the composite $n-1$-morphism
\[
\bD^{n-1}\to\Mor_\mC(\alpha(0),\alpha(1)) \to \Mor_\mC(X,Y)
\]
is $\phi_{X,Y}$-cocartesian.
\end{enumerate}
\end{definition}

\begin{example}Let $\phi: \mC \to \mD$ be a functor and $n \geq 0$.
Every $n$-morphism in $\mC$ that is an equivalence, is $\phi$-cocartesian.
An $n$-morphism in $\mC$ is an equivalence if and only if it is 
cocartesian for the functor $\mC \to \bD^0$.
    
\end{example}

\begin{remark}\label{elemen}
Let $n \geq 1$ and $\phi: \mC \to \mD$ a functor.
By pasting of pullbacks the composite of two composable $\phi$-cocartesian 1-morphisms is $\phi$-cocartesian.
Therefore for any three parallel $n-1$-morphisms $f,g,h$ in $\mC$
the composite of any two $\phi$-cocartesian $n$-morphisms $\alpha : f \to g, \beta: g \to h$ is again $\phi$-cocartesian.

Moreover the pasting law for pullbacks implies that for two functors
$\phi: \mC \to \mD, \kappa: \mD \to \mE$
an $n$-morphism of $\mC$ lying over a $\kappa$-cocartesian $n$-morphism of $\mD$, is $\phi$-cocartesian if and only if it is $\kappa \circ \phi$-cocartesian.
In particular, an $n$-morphism of $\mC$ that is inverted by $\phi$, is $\phi$-cocartesian if and only if it is an equivalence.

Also note that since forming morphism objects preserves pullbacks,
for any functors $\phi:\mC \to \mD$ and $\mB \to \mD$
an $n$-morphism in the pullback $\mB \times_\mD \mC $ is cocartesian for the projection to $\mB$ if it is $\phi$-cocartesian.
    
\end{remark}

\begin{definition} Let $0 \leq n \leq \infty.$
A functor $\phi: \mC \to \mD$ is an $n$-anticocartesian fibration if for every $1 \leq k \leq n$ 
every commutative square 
\begin{equation}\label{filler2}
\begin{xy}
\xymatrix{
\bD^{k-1} \ar[d] \ar[r]
& \mC \ar[d]^\phi
\\ 
\bD^k \ar[r] & \mD
}
\end{xy}\end{equation}
admits a filler by a $\phi$-cocartesian $k$-morphism,
where the left vertical functor is the source inclusion.

An anticocartesian fibration is an $\infty$-anticocartesian fibration.

\end{definition}

\begin{remark}
It follows immediately from the definition and \cref{elemen} that $n$-anticocartesian fibrations are stable under base change and composition.
\end{remark}

\begin{definition}Let $1 \leq n \leq \infty. $

\begin{enumerate}[\normalfont(1)]\setlength{\itemsep}{-2pt}

\item A functor $\phi: \mC \to \mD$ is an $n$-cocartesian fibration if $\phi^\co$ is a $n$-anticocartesian fibration.

\item A functor $\phi: \mC \to \mD$ is an $n$-cartesian fibration if $\phi^{\op}$ is an $n$-anticocartesian fibration.

\item A functor $\phi: \mC \to \mD$ is an $n$-anticartesian fibration if $\phi^{\co\op}$ is an $n$-anticocartesian fibration.

For $n=\infty$ we drop $n.$

\end{enumerate}
    
\end{definition}

\begin{remark}
Since the $(-)^\co$ involution fixes the orientation of 1-morphisms, a functor is a 1-(co)cartesian fibration if and only if it is a 1-anti(co)cartesian fibration.

\end{remark}

We have the following inductive definition of anticocartesian fibrations, which is \cite[Proposition 3.2.12.]{gepner2026fibrations}:

\begin{proposition}\label{cocarto}
Let $1 \leq n \leq \infty$ and $\phi: \mC \to \mD$ a functor.
The following are equivalent:

\begin{enumerate}[\normalfont(1)]\setlength{\itemsep}{-2pt}
\item The functor $\phi$ is an $n$-anticocartesian fibration.

\item The functor $\phi$ is a 1-anticocartesian fibration, and for every $\X,\Y \in \mC$ the functor
\[
\phi_{X,Y}: \Mor_\mC(\X,\Y) \to \Mor_\mD(\phi(\X),\phi(\Y))
\]
is a $n-1$-anticocartesian fibration
and for every pair of morphisms $X' \to X$ and $Y \to Y'$ in $\mC$
the induced functor 
$\Mor_\mC(\X,\Y) \to \Mor_\mC(X',Y')$ sends $\phi_{X,Y}$-cocartesian morphisms to $\phi_{X',Y'}$-cocartesian morphisms.
\end{enumerate}

\end{proposition}

\begin{definition}Let $n \geq 0$ and $\sigma$ a commutative square
\begin{equation}\label{sqco}
\begin{xy}
\xymatrix{
X \ar[d]^\psi \ar[r]^\kappa
& Y \ar[d]^\phi
\\ 
S \ar[r]^\rho & T
}
\end{xy}\end{equation}

\begin{enumerate}[\normalfont(1)]\setlength{\itemsep}{-2pt}
\item A commutative square \ref{sqco} is a map of $n$-anticocartesian fibrations if $\psi, \phi$ are $n$-anticocartesian fibrations and $\kappa$ sends $\psi$-cocartesian morphisms of dimension smaller or equal $n$ to $\phi$-cocartesian morphisms.

\item A commutative square \ref{sqco}
is a map of $n$-cocartesian fibrations if $\sigma^\co$ is a map of $n$-anticocartesian fibrations.

\item A commutative square \ref{sqco}
is a map of $n$-anticartesian fibrations if $\sigma^\op$ is a map of $n$-anticocartesian fibrations.

\item A commutative square \ref{sqco}
is a map of $n$-cartesian fibrations if $\sigma^\coop$ is a map of $n$-anticocartesian fibrations.

\item A commutative square \ref{sqco}
is a map of (anti)(co)cartesian fibrations if it is a map of
$n$-(anti)(co)cartesian fibrations for every $n \geq 1.$

\item A map of (anti) $n$-(co)cartesian fibrations over $\mD$ is a map of $n$-(anti)(co)cartesian fibrations such that $\rho$ is the identity.

\item A map of (anti)(co)cartesian fibrations over $\mD$ is a map of (anti) (co)cartesian fibrations such that $\rho$ is the identity.
\end{enumerate}

\end{definition}

\begin{notation}\label{notfib1}

Let $$\coCart, \Cart, \co\overline{\Cart}, \overline{\Cart} \subset \Fun(\bD^1,\infty\scat)$$
be the respective subcategories of cocartesian, cartesian, anticocartesian, anticartesian fibrations and morphisms of such. 

\end{notation}

\begin{notation}\label{notfib2} Let $S$ be an $\infty$-category.
Let $$ \infty\scat^\cocart_{/S}, \infty\scat^\cart_{/S}, \infty\scat^{\overline{\cocart}}_{/S}, \infty\scat^{\overline{\cart}}_{/S} \subset \infty\scat_{/S} $$ be the respective subcategories of cocartesian, cartesian, anticocartesian, anticartesian fibrations and maps of such.

\end{notation}

\begin{remark}\label{seqcol}
It follows immediately from the definition that the subcategories of
\cref{notfib1} and \cref{notfib2} admit small filtered colimits preserved by the respective inclusions to
$\Fun(\bD^1,\infty\scat), \infty\scat_{/S}.$
\end{remark}

The following is \cite[Corollary 3.3.3]{gepner2026fibrations}:

\begin{corollary}\label{fiberwiseeq}

A map $\kappa: \mB \to \mC$ of cocartesian fibrations over $\mD$ is an equivalence if and only for every $Z \in \mD$ the following induced functor is an equivalence:
$$ \{Z\} \times_\mD \mB \to  \{Z\} \times_\mD \mC. $$
    
\end{corollary}

The following is \cite[Corollary 3.3.6]{gepner2026fibrations}:

\begin{corollary}\label{paracocart}
A map $\kappa: \mB \to \mC$ of cocartesian fibrations over $\mD$ is a cocartesian fibration if and only if for every $Z \in \mD$ the induced functor 
$
\{Z\} \times_\mD \mB \to  \{Z\} \times_\mD \mC
$
is a cocartesian fibration and for every morphism $Z \to W$ in $\mD$ the induced functor
$\{W\} \times_\mD \mB \to \{Z\} \times_\mD \mB$
is a map of cocartesian fibrations.
   
\end{corollary}

We have the following local definition of cocartesian fibrations, which is \cite[Corollary 3.3.8]{gepner2026fibrations}:

\begin{corollary}\label{thetalocal} Let $\phi: \mC \to \mD$ be a functor.
The following are equivalent:

\begin{enumerate}[\normalfont(1)]\setlength{\itemsep}{-2pt}
\item The functor $\phi$ is a cocartesian fibration.

\item For every $\theta \in \Theta$ the pullback of $\phi$ along any functor $\theta \to \mD$ is a cocartesian fibration.

\item The pullback of $\phi$ along any functor $\bbDelta^n \to \mD$
for $n \geq 0 $ is a cocartesian fibration.

\item The pullback of $\phi$ along any functor $\cube^n \to \mD$
for $n \geq 0 $ is a cocartesian fibration.

\end{enumerate}

\end{corollary}

\begin{definition}

Let $ n \geq 0 $ and $\mA,\mB,\mC$ be $\infty$-categories.

\begin{enumerate}[\normalfont(1)]\setlength{\itemsep}{-2pt}
\item An $n$-bifibration 
is a functor $\mC \to \mA \times \mB $ that is a map of $n$-cartesian fibrations over $\mA$ and $n$-cocartesian fibrations over $\mB.$

\item A $n$-antibifibration 
is a functor $\mC \to \mA \times \mB $ that is a map of $n$-anticartesian fibrations over $\mA$ and $n$-anticocartesian fibrations over $\mB.$

\item A bifibration 
is a functor $\mC \to \mA \times \mB $ that is a $n$-bifibration for every $n \geq 0.$

\item An antibifibration 
is a functor $\mC \to \mA \times \mB $ that is a $n$-antibifibration for every $n \geq 0.$
\end{enumerate}

\end{definition}

\begin{remark}
Let $X$ be an $\infty$-category.
We will show in \cref{targetfibr} that the source and target projections
\[
\Fun^{\oplax}(\bD^1,X)\to X\times X\qquad\text{and}\qquad \Fun^{\lax}(\bD^1,X)\to X\times X
\]
form a bifibration and antibifibration, respectively.
\end{remark}

\begin{definition}
Let $ n \geq 0 $ and $\mA,\mB,\mC, \mA',\mB',\mC'$ be $\infty$-categories.
A map of (anti) $n$-bifibrations is a commutative square
\begin{equation}\label{sqfib}
\xymatrix{
\mC \ar[r] \ar[d] & \mC' \ar[d] \\
\mA \times \mB \ar[r] & \mA' \times \mB'
}
\end{equation}
whose vertical functors are (anti) $n$-bifibrations,
that induces a map of (anti) $n$-cartesian fibrations after projection to the first factor and induces a map of (anti) $n$-cocartesian fibrations after projection to the second factor.

A map of (anti) bifibrations is a commutative square
\ref{sqfib} that is a map of (anti) $n$-bifibrations for every $n \geq 0.$

\end{definition}

The following is \cite[Corollary 4.1.7.]{gepner2026fibrations}:

\begin{corollary}\label{targetfibr}
Let $\mA, \mC$ be $\infty$-categories.
The functor $$\Fun^\oplax(S(\mA),\mC) \simeq \mC \,{{\vec{\times}}}_{\Fun^\oplax(\mA,\mC)} \mC \to \mC \times \mC $$
is a bifibration.
\end{corollary}

\begin{corollary}\label{targetfi2}
The functor $\Fun^{\oplax}(\bD^1,\mC)\to\mC\times\mC$ is a bifibration whose fiber over $(A,B)$ is $\Mor_\mC(A,B)$.
\end{corollary}

\begin{corollary}\label{slicecocart}
Let $\mC$ be an $\infty$-category and $X \in \mC.$

\begin{enumerate}[\normalfont(1)]\setlength{\itemsep}{-2pt}
\item The functor $ \mC_{//^\oplax X}\to \mC$ is a cartesian fibration. 

\item The functor $ \mC_{X //^\oplax}\to \mC$ is a cocartesian fibration.

\end{enumerate}
\end{corollary}

The following is \cite[Lemma 4.6.8.]{gepner2026fibrations}:

\begin{lemma}\label{Grorep}
Let $X$ be a small $\infty$-category and $Z \in X.$
The following commutative square is a pullback square:
$$
\begin{xy}
\xymatrix{
X_{Z //^\oplax} \ar[d] \ar[r]
& \infty\scat_{*//^\oplax } \ar[d]
\\ 
X  \ar[r]^{\Mor_\X(Z,-)} & \infty\scat
}
\end{xy}$$

\end{lemma}

The following is \cite[Corollary 4.2.7.]{gepner2026fibrations}:

\begin{corollary}\label{indulax}

Let $\mB$ be an $\infty$-category and $\mC \to \mD$ a cocartesian fibration. The induced functor
$$\Fun^\lax(\mB, \mC) \to \Fun^\lax(\mB, \mD) $$
is a cocartesian fibration.

\end{corollary}

\begin{notation}
Let $X \to S$ be a cocartesian fibration
and $K$ an $\infty$-category.
Let $X^K_\lax \to S $ be the pullback of the induced functor
$$ \Fun^\lax(K,X) \to \Fun^\lax(K,S) $$
along the diagonal functor $S \to \Fun^\lax(K,S). $

\end{notation}

\cref{indulax} gives the following:

\begin{corollary}\label{restcoca}

Let $X \to S$ be a cocartesian fibration
and $K$ an $\infty$-category.
The functor $X^K_\lax \to S $ is a cocartesian fibration.
   
\end{corollary}

\begin{notation}
Let $\alpha: X \to S, \beta: Y \to S$ be functors.
Let $\Fun^\oplax_S(X,Y)$ be the fiber of the induced functor
$$ \Fun^\oplax(X,Y) \to \Fun^\oplax(X,S)$$ over $\alpha.$
    
\end{notation}

\begin{remark}\label{univeqr}
    
For every $\infty$-category $K$ and functors $X \to S, Y \to S$ there is a canonical equivalence
$$ \Fun^\oplax_S(X,Y^K_\lax) \simeq \Fun^\lax(K,\Fun^\oplax_S(X,Y)).$$
\end{remark}

The following is \cite[Proposition 5.1.4.]{gepner2026fibrations}:

\begin{lemma}\label{repren}
Let $X \to S, Y \to S$ be cocartesian fibrations.
There is a subcategory $$ \Fun^{\oplax, \cocart}_S(X,Y) \subset \Fun^\oplax_S(X,Y)$$ whose $n$-morphisms for some $n \geq 0$ are the $n$-morphisms of $ \Fun^\oplax_S(X,Y)$ corresponding to a map $X \to Y^{\bD^n}_\lax $ of cocartesian fibrations over $S.$

\end{lemma}

The following is \cite[Lemma 5.1.5.]{gepner2026fibrations}:
 
\begin{lemma}\label{cotensol}

Let $X \to S, Y \to S$ be cocartesian fibrations and $K$ an $\infty$-category. The canonical equivalence
$$ \Fun^\oplax_S(X,Y^K_\lax) \simeq \Fun^\lax(K,\Fun^\oplax_S(X,Y)) $$
restricts to an equivalence
$$ \Fun^{\oplax, \cocart}_S(X,Y^K_\lax) \simeq \Fun^\lax(K,\Fun^{\oplax, \cocart}_S(X,Y)). $$
    
\end{lemma}

The following is \cite[Theorem 5.1.8.]{gepner2026fibrations}:

\begin{theorem}

Let $S$ be an $\infty$-category.
There is antioriented subcategory 
$$ \infty\mathfrak{Cat}^\cocart_{/S} \subset \infty\mathfrak{Cat}_{/S}$$
such that for every cocartesian fibrations $X \to S, Y \to S$
there is a canonical equivalence
$$ \L\Mor_{\infty\mathfrak{Cat}^\cocart_{/S}}(X,Y) \simeq \Fun^{\oplax,\cocart}_S(X,Y).$$

\end{theorem}

\begin{definition}
Let $\phi: \mC \to \mD $ be a functor.
The enveloping cocartesian fibration or free cocartesian fibration
is the cocartesian fibration $$\Env(\mC):=\Fun^\oplax(\bD^1,\mD) \times_{\Fun^\oplax(\{0\},\mD)} \mC \to \Fun^\oplax(\{1\},\mD)\simeq \mD.$$

\end{definition}

\begin{remark}

The unique functor $\bD^1 \to \bD^0$ induces an inclusion
$ \mD \simeq \Fun^\oplax(\bD^0,\mD) \to \Fun^\oplax(\bD^1,\mD)$.
The latter gives rise to an inclusion
$$ \mC \simeq \mD \times_{\mD} \mC \to \Fun^\oplax(\bD^1,\mD) \times_{\Fun^\oplax(\{0\},\mD)} \mC = \Env_\mD(\mC) $$
over $\Fun^\oplax(\{1\},\mD).$

\end{remark}

\begin{remark}
For every $X \in \mD$ the fiber $\{X\} \times_\mD \Env(\mC)$
is canonically equivalent to $$ \mC \times_{\mD} \Fun^\oplax(\bD^1,\mD) \times_\mD \{X\} \simeq \mC \underset{\mD}{{\,\vec{\times}}} \{X\} .$$

\end{remark}

\begin{notation}

Let $\mB \to \mD, \mC \to \mD$ be functors.
\begin{enumerate}[\normalfont(1)]\setlength{\itemsep}{-2pt}

\item Let $\Fun_\mD^\cocart(\mB,\mC) \subset \Fun_\mD(\mB,\mC) $
be the full subcategory of maps of cocartesian fibrations over $\mD.$

\item Let $\Fun_\mD^\cart(\mB,\mC) \subset \Fun_\mD(\mB,\mC) $
be the full subcategory of maps of cartesian fibrations over $\mD.$

\end{enumerate}

\end{notation}

The following is \cite[Theorem 4.5.6.]{gepner2026fibrations}:

\begin{theorem}\label{envelo}

Let $\phi: \mC \to \mD $ be a functor and
$\rho: \mB \to \mD $ a cocartesian fibration.
The induced functor
$$ \Fun^\cocart_{\Fun^\oplax(\{1\},\mD)}(\Env(\mC),\mB) \to \Fun_\mD(\mC,\mB) $$
is an equivalence.

\end{theorem}

\begin{corollary}\label{envelo2}
Let $\phi: \mC \to \mD $ be a functor and
$\rho: \mB \to \mD $ a cartesian fibration.
The induced functor
$$ \Fun^\cart_{\Fun^\oplax(\{0\},\mD)}(\Env(\mC^\coop)^\coop,\mB) \to \Fun_\mD(\mC,\mB) $$
is an equivalence.
\end{corollary}

\begin{definition}Let $\mC$ be a small $\infty$-category.
Let $$ \infty\scat^\cocart_{/\mD}\subset \infty\scat_{/\mD} $$
be the subcategory of cocartesian fibrations over $\mC$ and maps of cocartesian fibrations over $\mD$.
\end{definition}

The following is \cite[Corollary 4.5.9.]{gepner2026fibrations}:

\begin{corollary}\label{huip}
The inclusion of $\infty$-categories $$\infty\scat^\cocart_{/\mD}\subset \infty\scat_{/\mD}$$
admits a left adjoint that sends $\mC \to \mD$ 
to $\Env(\mC) \to \mD$.
\end{corollary}

The following is \cite[Theorem 5.1.17.]{gepner2026fibrations}:

\begin{corollary}
The inclusion of antioriented $\infty$-categories $$\infty\fcat^\cocart_{/\mD}\subset \infty\fcat_{/\mD}$$
admits a left adjoint that sends $\mC \to \mD$ 
to $\Env(\mC) \to \mD$.
\end{corollary}

\begin{definition}
A left fibration is a cocartesian fibration whose fibers are spaces.
    
\end{definition}

\begin{remark}

A cocartesian fibration is a left fibration if and only if
every odd dimensional cell is cocartesian and every even dimensional cell is cartesian.
One sees this easily by induction on $n \geq 0$ proving the following more general statement: a cocartesian fibration has fibers $n$-categories 
if and only if every odd dimensional cell of dimension larger $n$ is cocartesian and every even dimensional cell of dimension larger $n$ is cartesian.
    
\end{remark}

\begin{notation}
Let $\LFib \subset \coCart$ be the full subcategory of left fibrations.
    
\end{notation}

\begin{remark} Let $S$ be an $\infty$-category.
Then $$ \LFib_S \subset \infty\scat^\cocart_{/S}$$ is the full reflexive subcategory of left fibrations over $S.$
We write $ L: \infty\scat^\cocart_{/S} \to \LFib_S $ for the left adjoint.
    
\end{remark}

\begin{notation}\label{fibinv}
Let $X \to S$ be a cocartesian fibration and $\mE $ a collection of cells of $X.$
Let $X^{\mE^{-1}}$ be the pushout
$$  X \times _{(\mE_n \times \Env_S(\bD^n))} (\mE_n \times L(\Env_S(\bD^n))).$$
    
\end{notation}

\begin{lemma}

Let $X \to S, Y \to S$ be cocartesian fibrations and $\mE $ a collection of cells of $X.$
The induced functor
$$ \Fun^\cocart_S(X^{\mE^{-1}},Y) \to \Fun^\cocart_S(X,Y) $$ is fully faithful and the essential image precisely consists of the maps $X \to Y$ of cocartesian fibrations over $S$ sending $\mE$ to cocartesian morphisms.

\end{lemma}

\begin{proof}
The induced functor
$ \Fun^\cocart_S(X^{\mE^{-1}},Y) \to \Fun^\cocart_S(X,Y) $ identifies with the functor
$$ \Fun^\cocart_S(X^{\mE^{-1}},Y) \simeq \Fun^\cocart_S(X,Y) \times_{\prod_{n > 0}\Fun(\mE_n, \Fun_S(\bD^n,Y))} \prod_{n > 0}\Fun(\mE_n,\Fun_S(\bD^{n-1},Y')) \to \Fun^\cocart_S(X,Y). $$
This functor is fully faithful by \cite[Corollary 4.1.4.]{oriented} and the essential image precisely consists of the maps $X \to Y$ of cocartesian fibrations over $S$ sending $\mE$ to cocartesian morphisms.
\end{proof}

\subsection{The Grothendieck construction}

For the next definition we use \cref{envelo2}:

\begin{definition}

The Grothendieck construction is the unique map
$$\int: \infty\widehat{\Cat}_{//^\oplax \infty\scat} \to \co\widehat{\Cart}$$
of cartesian fibrations over
$\infty\widehat{\Cat}$ sending $\infty\scat$ to the universal cocartesian fibration $ \infty\scat_{*//^\oplax} \to \infty\scat.$

The map $\int: \infty\widehat{\Cat}_{//^\oplax \infty\scat} \to \co\widehat{\Cart}$ 
restricts to a map 
\begin{equation}\label{GroCon}
\int: \infty\scat \times_{\infty\widehat{\Cat}} \infty\widehat{\Cat}_{//^\oplax \infty\scat} \to \coCart
\end{equation} of cartesian fibrations over $\infty\scat$,
which we also call the Grothendieck construction.

\end{definition}

\begin{remark}
The map $$\int: \infty\widehat{\Cat}_{//^\oplax \infty\scat} \to \co\widehat{\Cart}$$ of cartesian fibrations over
$\infty\widehat{\Cat}$ induces on the fiber over every $\mC \in \infty\widehat{\Cat}$ the functor
$$\int_\mC: \Fun(\mC, \infty\scat) \to \infty\widehat{\Cat}^\coCart_{/\mC}, $$
which takes the pullback along the universal cocartesian fibration.
The map 
\begin{equation}\label{GroCon}
\int: \infty\scat \times_{\infty\widehat{\Cat}} \infty\widehat{\Cat}_{//^\oplax \infty\scat} \to \coCart
\end{equation} of cartesian fibrations over $\infty\scat$ induces on the fiber over every $\mC \in \infty\scat$ the functor
$$\int_\mC : \Fun(\mC, \infty\scat) \to \infty\scat^\cocart_{/\mC}$$
that takes the pullback along the universal cocartesian fibration.

\end{remark}

\begin{remark}

For every $ 0 \leq n \leq \infty$ the map (\ref{GroCon}) restricts to a map $$\int: n\Cat \times_{\infty\widehat{\Cat}} \infty\widehat{\Cat}_{//^\oplax n\Cat} \to {n\coCart}$$ of cartesian fibrations over $n\Cat,$ which induces on the fiber over every $\mC \in {n\Cat}$ a functor 
$$\int_\mC: \Fun(\mC, n\Cat) \to {n\Cat}^\cocart_{/\mC}.$$
\end{remark}

\begin{definition}
Let $ 0 \leq n \leq \infty$ and $\mC \in {n\Cat}$.
\begin{enumerate}[\normalfont(1)]\setlength{\itemsep}{-2pt}
\item The Grothendieck construction for cartesian fibrations is the 
functor $$\int_\mC: \Fun(\mC^\circ, n\Cat) \to {n\Cat}^\cart_{/\mC}$$
corresponding to the composition
$$\Fun(\mC^\circ, n\Cat)^\cop \simeq \Fun(\mC, n\Cat^\circ)^{\co\op} \xrightarrow{(-)^{\co\op}_!} \Fun(\mC, n\Cat^{\co\op})^{\co\op} \simeq $$$$ \Fun(\mC^{\co\op}, n\Cat) \xrightarrow{\int_{\mC^{\co\op}}} {n\Cat}^\cocart_{/\mC^{\co\op}} \xrightarrow{(-)^{\co\op}} ({n\Cat}^\cart_{/\mC})^\cop$$

\item The Grothendieck construction for anticartesian fibrations is the functor $$\int_\mC: \Fun(\mC^{\co\op}, n\Cat) \to {n\Cat}^{\overline{\cart}}_{/\mC}$$
corresponding to the composition $$ \Fun(\mC^{\co\op}, n\Cat)^\co \simeq \Fun(\mC, n\Cat^{\co\op})^\op \xrightarrow{(-)^{\op}_!} \Fun(\mC, n\Cat^\op)^\op \simeq $$$$ \Fun(\mC^\op, n\Cat) \xrightarrow{\int_{\mC^\op}} {n\Cat}^\cocart_{/\mC^\op} \xrightarrow{(-)^{\op}} ({n\Cat}^{\overline{\cart}}_{/\mC})^\co$$
\item The Grothendieck construction for anticocartesian fibrations is the
functor $$\int_\mC: \Fun(\mC^{\cop}, n\Cat) \to {n\Cat}^{\co\overline{\cart}}_{/\mC}$$
corresponding to the composition $$ (\Fun(\mC^\cop, n\Cat)^\op)^\circ \simeq \Fun(\mC, n\Cat^\cop)^\co \xrightarrow{(-)^{\co}_!} \Fun(\mC, n\Cat^\co)^\co \simeq $$$$ \Fun(\mC^\co, n\Cat) \xrightarrow{\int_{\mC^\co}} {n\Cat}^\cocart_{/\mC^\co} \xrightarrow{(-)^{\co}} (({n\Cat}^{\co\overline{\cart}}_{/\mC})^\op)^\circ.$$
\end{enumerate}
\end{definition}

\begin{definition}Let $ 0 \leq n \leq \infty.$
We say that the $n$-categorical Grothendieck construction is an equivalence if for every $\mC \in n\Cat$ the following functor is an equivalence:
$$\int_\mC: \Fun(\mC, n\Cat) \to {n\Cat}^\cocart_{/\mC}. $$
    
\end{definition}

The following is \cite[Corollary 4.6.9.]{gepner2026fibrations}:

\begin{corollary}\label{Grorep2}

Let $X$ be an $\infty$-category and $Z \in X.$
The Grothendieck construction for cartesian fibrations sends 
$$ \Mor_\X(-,Z): X^\circ \to \infty\scat $$ 
to the cartesian fibration $$ X_{//^\oplax Z} \to X. $$

\end{corollary}

\begin{example}\label{Gro0}

The 0-categorical Grothendieck construction is an equivalence since 
for every small space $X$ the canonical functor
$$\Fun(X, \mS) \to {0\Cat}^\cocart_{/X} = \mS_{/X} $$
taking the fiber over the contractible space is an equivalence.

\end{example}

\begin{remark}\label{clas}
Let $ 0 \leq n \leq \infty$ and assume the $n$-categorical Grothendieck construction is an equivalence.
Then the map $$\int: n\Cat \times_{\infty\widehat{\Cat}} \infty\widehat{\Cat}_{//^\oplax n\Cat} \to {n\coCart}$$ of cartesian fibrations over $n\Cat$ is fiberwise an equivalence and so an 
equivalence by \cref{fiberwiseeq}.
So by \cref{Grorep2} the cartesian fibration
$ {n\coCart} \to n\Cat$ is the Grothendieck construction associated to the functor $$\Fun(-,n\Cat): n\Cat^\circ \to \infty\scat.$$

\end{remark}

\begin{proposition}\label{Grothendieck-char}
Let $ 0 \leq n \leq \infty$ and $\mC$ an $n$-category. The following are equivalent:

\begin{enumerate}[\normalfont(1)]\setlength{\itemsep}{-2pt}

\item The $n$-categorical Grothendieck construction $\int_\mC: \Fun(\mC,n\Cat) \to n\Cat^\cocart_{/\mC} $ is an equivalence.

\item The $\infty$-category $ n\Cat^\cocart_{/\mC}$ admits small weighted colimits and for every
$X \in \mC$ the functor $$ n\Cat^\cocart_{/\mC} \to n\Cat$$
taking the fiber over $X$ preserves small weighted colimits.

\item The category $\iota_1(n\Cat^\cocart_\mC)$ admits small colimits and for every
$X \in \mC$ the functor $$ \iota_1(n\Cat^\cocart_{/\mC}) \to \iota_1(n\Cat) $$ taking the fiber over $X$ preserves small colimits.
    
\end{enumerate}
    
\end{proposition}

\begin{proof}

(2) and (3) are equivalent since $n\Cat^\cocart_{/\mC}$ admits tensors and for any $X \in \mC$ the functor $ n\Cat^\cocart_{/\mC}\to n\Cat$ preserves tensors.

(1) evidently implies (3) since the category
$\iota_1(\Fun(\mC, n\Cat))$ admits small colimits and for every $X \in \mC$ the composition
of the functor $ \Fun(\mC, n\Cat) \to n\Cat^\cocart_{/\mC}$
and the functor taking the fiber over $X$ is the functor evaluating at $X.$

We prove that (2) implies (1).
If $ n\Cat^\cocart_{/\mC}$ admits small weighted colimits, by \cref{Yonedaext} the composition 
$$ \mW: \mC^\circ \subset \Fun(\mC, n\Cat) \xrightarrow{\int} n\Cat^\cocart_{/\mC} $$
of the Yoneda embedding and the Grothendieck construction
uniquely extends to a left adjoint functor 
$F: \Fun(\mC, n\Cat) \to n\Cat^\cocart_{/\mC}.$
The right adjoint $\N$ of $F$ is the restricted Yoneda embedding
$$n\Cat ^\coCart_{/\mC} \to \Fun((n\Cat^\cocart_{/\mC})^\circ, \infty\scat) \xrightarrow{\mW^*} \Fun(\mC, \infty\scat), $$
which lands in $ \Fun(\mC, n\Cat).$
In other words the functor $\N$ sends a cocartesian fibration $\mD \to \mC $ to the functor $ \Mor_{n\Cat^\cocart_{/\mC}}(-,\mD) \circ \mW $. By \cref{envelo} for every $X \in \mC$ there is a canonical equivalence
$$\N(\mD)(X)= \Mor_{n\Cat^\cocart_{/\mC}}(\mC_{X//^\oplax},\mD) \simeq \mD_{X}. $$
Hence by \cref{fiberwiseeq} the functor $\N$ is conservative.

If (2) holds, for every $X \in \mC$ the functor
$n\Cat^\cocart_{/\mC} \to n\Cat$ taking the fiber over $X$ preserves small weighted colimits.
This functor identifies with the functor
\[
\Mor_{n\Cat^\cocart_{/\mC}}(\mC_{X//^\oplax},-) \simeq \Mor_{n\Cat_{/\mC}}(\{X\},-) : n\Cat^\cocart_{/\mC} \to n\Cat.
\]
Consequently, the cocartesian fibration $ \mC_{X//^\oplax} \to \mC$
corepresents a functor preserving small weighted colimits.
Since the functor $\mW$ is fully faithful by \cite[Corollary 4.54. (2)]{heine2024bienriched}, this implies that the functor $F$ is fully faithful.
The right adjoint $\N$ of $F$ is conservative so that $F$ is an equivalence.

It is not clear that the functor $F$ agrees with the Grothendieck construction since we do not know that the Grothendieck construction preserves small weighted colimits.
However by \cref{Yonedaext} (2) there is a unique natural transformation
$F \to \int $ whose restriction along the Yoneda embedding is the identity.
The latter and the unit $\id \to \N \circ F$ give rise to a natural transformation $ \id \to \N \circ F \to \N \circ \int$
whose component at any functor $\alpha: \mC \to \infty\scat$
is the canonical map of cocartesian fibrations over $\mC$ whose fiber at any $X \in \mC$ is the identity:  
$$\alpha(X) \to \Fun^\cocart(\mC_{X//^\oplax}, \alpha^*(\infty\scat_{*//^\oplax})) \simeq \{\alpha(X)\}\times_{\infty\scat} \infty\scat_{*//^\oplax} \simeq\alpha(X).$$
Thus the functor $\int: \Fun(\mC, n\Cat) \to n\Cat^\cocart_{/\mC} $ is 
a right inverse of the equivalence $\N$ and so an equivalence.
Since $\int$ preserves weighted colimits, by \cref{Yonedaext} the unique natural transformation $F \to \int $ whose restriction along the Yoneda embedding is the identity, is an equivalence.
\end{proof}

\begin{lemma}
Let $\mC$ be an $\infty$-category.

\begin{enumerate}[\normalfont(1)]\setlength{\itemsep}{-2pt}
\item
The canonical functor $$\infty\scat_{/\mC} \to \lim_{n \geq 1} n\Cat_{/\iota_n(\mC)} $$ is an equivalence and restricts to an 
equivalence $$ \infty\scat^\cocart_{/\mC} \to \lim_{n \geq 1} n\Cat^\cocart_{/\iota_n(\mC)} .$$

\item There is a canonical commutative square:

$$\begin{xy}
\xymatrix{
\Fun(\mC,\infty\scat) \ar[d]^{} \ar[rr]^{\int_\mC}
&& \infty\scat^\cocart_{/\mC} \ar[d] \ar[d]^{}
\\ 
\lim_{n \geq 1} \Fun(\iota_n(\mC),n\Cat)
\ar[rr]^{\lim_{n \geq 1} \int_{\iota_n(\mC)}} && \lim_{n \geq 1} n\Cat^\cocart_{/\iota_n(\mC)}
}
\end{xy}$$	
\end{enumerate}
\end{lemma}

\begin{proof}
We first prove (1). The canonical functor $$\infty\scat \to \lim_{n \geq 1} n\Cat $$ 
that induces on the $\n$-th factor the functor $\iota_n$ is an equivalence by the definition of $\infty\scat$.
The latter equivalences induces an equivalence
$$ \infty\scat_{/\mC} \to \lim_{n \geq 1} n\Cat_{/\iota_n(\mC)} $$
that restricts to an equivalence $$ \infty\scat^\cocart_{/\mC} \to \lim_{n \geq 1} n\Cat^\cocart_{/\iota_n(\mC)}.$$
The latter follows immediately from the fact that forming morphism $\infty$-categories commutes with filtered colimits.
Assertion (2) follows by the Yoneda lemma from the fact that
the cocartesian fibration $ \iota_n(\infty\scat_{*//^\oplax}) \to \iota_n(\infty\scat) $ classifies the functor 
$ \iota_n: \iota_n(\infty\scat) \to \n\Cat .$
\end{proof}

\begin{corollary}\label{Groind}

The $\infty$-categorical Grothendieck construction is an equivalence if and only if for every $n \geq 0$ the $n$-categorical Grothendieck construction is an equivalence.
   
\end{corollary}

The following is \cite[
Proposition 4.6.1.]{gepner2026fibrations}:

\begin{proposition}

Evaluation at the target $\coCart \to \infty\scat$
is a cartesian fibration.
  
\end{proposition}

For the following notation we use \cref{slicecocart}.
\begin{notation}

Let $$\Rep\Cart \subset \Cart$$ be the full subcategory of cartesian fibrations of the form
$\mC_{//^\oplax X} \to \mC $ for some small $\infty$-category $\mC.$

\end{notation}

The following is \cite[Proposition 4.6.4.]{gepner2026fibrations}:

\begin{proposition}\label{univ2}
Evaluation at the target $\Rep\Cart \to \infty\scat$
is a cocartesian fibration.
The unique map $$ \infty\scat_{*//^\oplax} \to \Rep\Cart $$
of cocartesian fibrations over $ \infty\scat$ sending $* $ to $*$ is an equivalence.

\end{proposition}

\begin{proposition}\label{univ3}
Let $n \geq 0$ and $S$ an $n$-category. Let $\alpha: S \to n\Cat$ be a functor and $X:=S \times_{\infty\scat} \infty\scat_{*//^\oplax} \to S $ the pullback along $\alpha.$
We assume that the $n$-categorical Grothendieck construction is an equivalence.
The canonical map 
$$ n\Cat_{*//^\oplax} \to n \Cat \times_{\infty\scat} \Rep\Cart $$
of cocartesian fibrations over $ n\Cat$ sending $*$ to $*$
induces a map
$$
\Map_{\infty\widehat{\Cat}_{/S}}(S,X) \simeq \Map_{\infty\widehat{\Cat}_{/n\Cat}}(S,n\Cat_{*//^\oplax}) \to $$$$\Map_{\infty\widehat{\Cat}_{/n\Cat}}(S,n \Cat \times_{\infty\scat} \Rep\Cart) \subset \Map_{\infty\widehat{\Cat}_{/n\Cat}}(S,\Fun(\bD^1, n\Cat)) \simeq $$$$\iota_0(\Fun(S,n\Cat)_{/\alpha}) \xrightarrow{\int_S}\iota_0((n\Cat^{\cocart}_{/S})_{/X})$$
which sends a section of the cocartesian fibration $X:=S \times_{\infty\scat} \infty\scat_{*//^\oplax} \to S $ to the map of cocartesian fibrations
$$ \mC^{\bD^1}_\oplax \times_{\mC_\oplax^{\{1\}}} S \to \mC_\oplax^{\{0\}}$$
over $S.$

\end{proposition}

\begin{proof}

We consider the map $$ \theta: \Map_{\infty\widehat{\Cat}_{/S}}(S,X) \to \iota_0((n\Cat^{\cocart}_{/S})_{/X}), $$ which sends a section of the cocartesian fibration $X \to S $ to the map of cocartesian fibrations
$$ X^{\bD^1}_\oplax \times_{X_\oplax^{\{1\}}} S \to X_\oplax^{\{0\}}$$
over $S$.
The map $$\Map_{\infty\widehat{\Cat}_{/n\Cat}}(S,n\Cat_{*//^\oplax}) \simeq \Map_{\infty\widehat{\Cat}_{/S}}(S,X) \xrightarrow{\theta} $$$$\iota_0((n\Cat^{\cocart}_{/S})_{/X}) \xrightarrow{\int_S^{-1}}\iota_0(\Fun(S,n\Cat)_{/\alpha})$$$$ \simeq \Map_{\infty\widehat{\Cat}_{/n\Cat}}(S,\Fun(\bD^1, n\Cat)) $$
represents a functor $$ \rho: n\Cat_{*//^\oplax} \to \Fun(\bD^1, n\Cat) $$
over $ n\Cat$. By \cref{targetfibr} the functor $\rho$ lands in $ n \Cat \times_{\infty\scat} \Rep\Cart \subset \Fun(\bD^1, n\Cat)$. The functor $\rho $ is a map of cocartesian fibrations over $ n\Cat$
and so by \cref{envelo} is uniquely determined by the image of $*$ in the fiber over $*.$ But $\rho$ sends the image of $*$ in the fiber over $*$ to $*.$
\end{proof}

Let $n \geq 0$. The unique map $$ n\Cat_{*//^\oplax} \to n \Cat \times_{\infty\scat} \Rep\Cart $$
of cocartesian fibrations over $ n\Cat$ sending $* $ to $*$
induces on the fiber over every $X \in n\Cat$ a functor $$X \to \infty\scat^\cart_{/X}.$$

\cref{univ3} implies the following:

\begin{corollary}\label{univ4} Let $n \geq 0$ and $X $ an $n$-category. We assume that the $n$-categorical Grothendieck construction is an equivalence.
The Grothendieck construction sends the functor 
$X \to \Fun(X^\circ, \infty\scat) \xrightarrow{\int_X} \infty\scat_{/X}$
to the canonical map of cocartesian fibrations
over $X:$
$$ \Fun^\oplax(\bD^1,X) \to X \times X. $$

\end{corollary}

\section{\mbox{Oriented colimits}}

\subsection{Weighted colimits}
Although the oriented pullback admits an elementary definition, it is also an instance of a weighted colimit.
In the following we define weighted colimits following \cite{heine2024higher}.

\begin{definition}
Let $\mV$ be a presentably monoidal category and $\mC$ a small $\mV$-enriched category.

\begin{itemize}
\item 
A $\mV$-enriched weight on $\mC$ is an object of $ \mP_\mV(\mC)$.

\item A small $\mV$-enriched weight is a $\mV$-enriched weight on some small $\mV$-enriched category.

\end{itemize}

\end{definition}

\begin{definition}
Let $\mV$ be a presentably monoidal category, $\mC$ a small $\mV$-enriched category, $\phi: \mC \to \mD$ a $\mV$-enriched functor, $X \in \mD$ and $H \in \mP_\mV(\mC).$

\begin{itemize}

\item A $H$-weighted cocone on $X$ is a morphism in $\mP_\mV(\mC):$
$$ H \to \Mor_\mD(-,X) \circ \phi. $$

\item A $\mV$-enriched cocone on $X$ is a $H$-weighted cocone on $X$ for some $H \in \mP_\mV(\mC).$

\item A $\mV$-enriched cocone in $\mD$ is a $\mV$-enriched cocone on $X$ for some $X \in \mD.$

\end{itemize}

\end{definition}

\begin{definition}
Let $\mV$ be a presentably monoidal category, $\mC$ a small $\mV$-enriched category and $H \in \mP_\mV(\mC).$ 
Let $\phi: \mC \to \mD$ be a $\mV$-enriched functor.
The $H$-weighted colimit of $\phi$ if it exists, is the object 
$\colim^H(\phi) \in \mD$ such that there is a $H$-weighted cocone on $\colim^H(\phi)$ that induces for every $Y \in \mD$ an equivalence
$$ \Mor_{\mD}(\colim^H(\phi),Y) \to \Mor_{\mP_\mV(\mC)}(H, \Mor_\mD(-,Y) \circ \phi).$$

\end{definition}

\begin{definition}
Let $\mV$ be a presentably monoidal category, $\mC$ a small $\mV$-enriched category and $H \in \mP_{\mV^\rev}(\mC^\circ).$ 
Let $\phi: \mC \to \mD $ be a $\mV$-enriched functor.
The $H$-weighted limit of $\phi$ is the $H$-weighted colimit of the opposite $\mV^\rev$-enriched functor $\phi^\circ : \mC^\circ \to \mD^\circ. $

\end{definition}

\begin{definition}
Let $\mV$ be a presentably monoidal category.

\begin{itemize}

\item Let $\mC$ be a small $\mV$-enriched category and $H \in \mP_\mV(\mC).$ A $\mV$-enriched category $\mD$ admits $H$-weighted colimits if it admits the $H$-weighted colimit of every $\mV$-enriched functor $\mC \to \mD.$

\item Let $\mH$ be a set of small $\mV$-enriched weights. 
A $\mV$-enriched category $\mD$ admits $\mH$-weighted colimits if it admits $H$-weighted colimits for every $H \in \mH.$ 
  
\end{itemize}

\end{definition}

We make a similar definition for weighted limits.

\vspace{1mm}

For presentably monoidal categories $\mV, \mW$ let 
$\alpha_\mV: \mV \to \mV \ot \mW, \alpha_\mW: \mW \to \mV \ot \mW$ be the canonical left adjoint functors.
The next is \cite[Theorem 3.102.]{heine2024higher}:

\begin{proposition}\label{colimenrfun}
Let $\mV, \mW $ be presentably monoidal categories, $\mC$ a small $\mV$-enriched category, $\mD$ a $\mV \ot \mW$-enriched category
and $\mH$ a collection of $\mW$-enriched weights such that $\mD$ admits
$(\alpha_\mW)_!(\mH)$-weighted (co)limits.
The $\mW$-enriched category ${\mV\mathrm{-}\Fun}(\mC,\mD) $
admits $\mH$-weighted (co)limits and for every $X \in \mC$
the $\mW$-enriched functor ${\mV\mathrm{-}\Fun}(\mC,\mD) \to \alpha^*_\mW(\mD) $ evaluating at $X$ preserves $\mH$-weighted (co)limits.
    
\end{proposition}

\begin{notation}
    
Let $\mV$ be a presentably monoidal category and $\mH$ a set of small $\mV$-enriched weights. Let $\mV \mathrm{-}\Cat(\mH) \subset \mV \mathrm{-}\Cat$ be the subcategory of $\mV$-enriched categories that admit $\mH$-weighted colimits and $\mV$-enriched functors preserving $\mH$-weighted colimits.
    
\end{notation}

The following is \cite[Proposition 4.47.]{heine2024higher}:

\begin{proposition}\label{weiadj} Let $\mV$ be a presentably monoidal category and $\mH$ a set of small $\mV$-enriched weights.
The inclusion $\mV \mathrm{-}\Cat(\mH) \subset \mV \mathrm{-}\Cat$ admits a left adjoint.

\end{proposition}

The following is \cite[Corollary 3.68.]{heine2024higher}:

\begin{proposition}\label{weiadj} Let $\mV$ be a presentably monoidal category and $\mC$ a small $\mV$-enriched category.
The $\mV$-enriched Yoneda embedding $\mC \to \mP_\mV(\mC)$ preserves weighted limits.

\end{proposition}

\begin{notation}Let $\mV$ be a presentably monoidal category, $\mC$ a small $\mV$-enriched category and $\Lambda$ a set of small $\mV$-enriched cocones in $\mC$. Let $\mP^\Lambda_\mV(\mC) \subset \mP_\mV(\mC)$ be the full $\mV$-enriched subcategory of $\mV$-enriched presheaves sending $\mV$-enriched cocones in $\mC$ of $\Lambda$ to weighted colimits in $\mV^\circ$.
    
\end{notation}

The following is \cite[Lemma 3.80., Remark 3.85.]{heine2024higher}:

\begin{proposition}\label{weiloc} Let $\mV$ be a presentably monoidal category, $\mC$ a small $\mV$-enriched category and $\Lambda$ a set of small $\mV$-enriched cocones in $\mC$.

\begin{enumerate}[\normalfont(1)]\setlength{\itemsep}{-2pt}

\item The $\mV$-enriched embedding $\mP^\Lambda_\mV(\mC) \subset \mP_\mV(\mC)$ admits a $\mV$-enriched left adjoint.

\item The restricted $\mV$-enriched localization $\mC \to \mP_\mV(\mC) \to \mP^\Lambda_\mV(\mC)$ sends $\mV$-enriched cocones of $\Lambda$ to weighted colimits.

\item If $\Lambda$ consists of weighted colimit cocones, the 
$\mV$-enriched Yoneda embedding of $\mC$ lands in $\mP^\Lambda_\mV(\mC)$
and so the restricted $\mV$-enriched localization $\mC \to \mP_\mV(\mC) \to \mP^\Lambda_\mV(\mC)$ is a $\mV$-enriched embedding and preserves 
weighted limits.

\end{enumerate}

\end{proposition}

The following is \cite[Proposition 3.19.]{heine2024higher}:

\begin{proposition}\label{weifunc}

Let $\mV, \mW$ be presentably monoidal categories, $\mC$ a small $\mV$-enriched category, $H$ a $\mV$-enriched weight on $\mC$ and $\mD$ a $\mV \ot \mW$-enriched category that admits $(\alpha_\mV)_!(H)$-weighted colimits.
There is a $\mW$-enriched functor 
$$ \mV\mathrm{-}\Fun(\mC,\mD) \to \alpha^*_\mW(\mD)$$
that sends a $\mV$-enriched functor to its $H$-weighted colimit.

\end{proposition}

Next we consider pushouts along enriched embeddings.
Pushouts of categories are typically quite difficult to describe, as they generalize the amalgamated product from monoid, or group, theory to the many-object setting.
However, in we restrict to fully faithful functors, the pushout admits an easier description.

\begin{lemma}\label{pulla}
Let $\mV$ be a presentably monoidal category and $\theta: \mC' \to \mC$
a $\mV$-enriched functor.
Let $\mD \to \mC$ be a $\mV$-enriched functor that admits a left
$\mV$-enriched fully faithful left adjoint $\phi$.
Let $\theta': \mD':= \mD \times_\mC \mC' \to \mD$ be the projection.
The projection $\mD' \to \mC' $ admits a $\mV$-enriched fully faithful left adjoint $\phi'$ and the canonical natural transformation
$$ \phi \circ \theta \to \theta' \circ \phi' $$
of functors $\mC'\to \mD$ is an equivalence.
    
\end{lemma}

\begin{proof}

We use the following terminology.
A $\mV$-enriched adjunction relative to $\mC$ is an adjunction in the 2-category $ {\mV\mathrm{-}\Cat}_{/\mC}.$
There is a forgetful 2-functor $ {\mV\mathrm{-}\Cat}_{/\mC} \to {\mV\mathrm{-}\Cat}$ so that any $\mV$-enriched adjunction relative to $\mC$ lies over a $\mV$-enriched adjunction.
Similarly, for every $\mV$-enriched functor
$\theta: \mC' \to \mC$ there is a base change 2-functor $ \theta^*:{\mV\mathrm{-}\Cat}_{/\mC'} \to {\mV\mathrm{-}\Cat}_{/\mC},$ which as a 2-functor preserves adjunctions.
In other words the pullback of any adjunction relative to $\mC$ along $ \theta:\mC' \to \mC$ is an adjunction relative to $\mC'$. 
Moreover note that $\theta^*$ sends fully faithful $\mV$-enriched functors over $\mC$ to fully faithful $\mV$-enriched functors over $\mC'$ since fully faithful functors are stable under base change (because equivalences are stable under base change and forming morphism objects commutes with pullbacks).

Any $\mV$-enriched functor $\mB \to \mA$ over $\mC$ admits a $\mV$-enriched left adjoint relative to $\mC$ if and only if it admits a  $\mV$-enriched left adjoint and the unit is sent by the functor $\mA \to \mC$ to an equivalence. In particular, the unique $\mV$-enriched functor $\mD \to \mC$ over $\mC$, which admits a fully faithful $\mV$-enriched left adjoint, also admits a fully faithful $\mV$-enriched left adjoint relative to $\mC$,
which is preserved by base change along $\theta: \mC' \to \mC.$
\end{proof}

\begin{proposition}\label{fullyfaith}

Let $\mV$ be a presentably monoidal category,
$\beta: \mC \to \mD$ a $\mV$-enriched functor, and $\alpha: \mC \to \mC'$
a fully faitful $\mV$-enriched functor.
\begin{enumerate}[\normalfont(1)]\setlength{\itemsep}{-2pt}
\item The cobase change
$\alpha': \mD \to \mD' := \mC' \coprod_\mC \mD$ is fully faithful.
\item Let $\beta': \mC' \to \mD' $ be the cobase change.
The canonical natural transformation
$$ \alpha_! \circ \beta^* \to \beta'^* \circ \alpha'_! $$
of functors $\mP_\mV(\mD) \to \mP_\mV(\mC') $ is an equivalence.
\end{enumerate}   
\end{proposition}

\begin{proof}

By \cite[Corollary 1.8.]{heine2024bienriched} every $\mV$-enriched functor $\phi: \mA \to \mB$ 
gives rise to a $\mV$-enriched adjunction
$$\phi_! : \mP_\mV(\mA) \rightleftarrows \mP_\mV(\mB) : \phi^* $$ on $\mV$-enriched presheaves, i.e. an adjunction in the 2-category ${\mV\mathrm{-}\Cat}$, where $\phi^*$ restricts along $\phi$ and
$\phi_!$ extends $\phi$ along the $\mV$-enriched Yoneda embeddings.
Since $\phi^* $ preserves small weighted colimits and
$\mP_\mV(\mB)$ is generated by the representable presheaves 
under small weighted colimits and the $\mV$-enriched Yoneda embeddings are fully faithful, $\phi$ is fully faithful
if and only if $\phi_!$ is fully faithful.
So by uniqueness of left adjoints the $\mV$-enriched functor $\phi$ is fully faithful if and only if $\phi^*$ admits a fully faithful $\mV$-enriched left adjoint.
\cref{psinho} implies that the functor $ \mP_\mV(-): {\mV\mathrm{-}\Cat}^\op \to {_\mV\widehat{\Cat}} $ sends small colimits to limits. 
Hence $$\alpha'^* : \mP_\mV(\mD') \to \mP_\mV(\mD) $$ is the base change of $\alpha^* : \mP_\mV(\mC') \to \mP_\mV(\mC)$ along $\beta^*: \mP_\mV(\mD) \to \mP_\mV(\mC)$.
By what we have proven, it suffices to see that $\alpha'^*$ admits a fully faithful $\mV$-enriched left adjoint if $\alpha^*$ does.
This follows from \cref{pulla}.
(2) follows from the second part of \cref{pulla}.
\end{proof}

\begin{proposition}\label{homsus}
Let $\mV$ be a presentably monoidal category,
$\beta: \mC \to \mD$ a $\mV$-enriched functor, and $\alpha: \mC \to \mC'$
a fully faithful $\mV$-enriched functor.
Let $\alpha': \mD \to \mD' := \mC' \coprod_\mC \mD, 
\beta': \mC' \to \mD' := \mC' \coprod_\mC \mD$ be the cobase change.
For every $X \in \mC', Y \in \mD$ 
there is a canonical equivalence
$$ \Mor_{\mD'}(\beta'(X),\alpha'(Y)) \simeq \int_{T \in \mC} \Mor_{\mC'}(X,\alpha(T)) \ot \Mor_{\mD}(\beta(T),Y) . $$ 

\end{proposition}

\begin{proof}
Since $\alpha'_!: \mP_\mV(\mD) \to \mP_\mV(\mD') $ preserves representable $\mV$-enriched presheaves, there is an equivalence
$$ \Mor_{\mD'}(\beta'(X),\alpha'(Y)) \simeq \beta'^*(\Mor_{\mD'}(-,\alpha'(Y)))(X) \simeq \beta'^*(\alpha'_!(\Mor_{\mD}(-,Y)))(X). $$
Since $\alpha$ is fully faithful, by \cref{fullyfaith} (2) the canonical $\mV$-enriched natural transformation
$$ \alpha_! \circ \beta^* \to \beta'^* \circ \alpha'_! $$
of $\mV$-enriched functors $\mP_\mV(\mD) \to \mP_\mV(\mC') $ is an equivalence.
So there is a canonical equivalence
$$ \beta'^*(\alpha'_!(\Mor_{\mD}(-,Y)))(X) \simeq \alpha_!(\beta^*(\Mor_{\mD}(-,Y)))(X)= \alpha_!(\Mor_{\mD}(\beta(-),Y))(X).$$

By \cite[Theorem 4.119., Theorem 4.43.]{heine2024bienriched} for every $Z \in \mP_\mV(\mC) $ there is a canonical equivalence
$$ \alpha_!(Z) \simeq \int_{T \in \mC} \Mor_{\mC'}(-,\alpha(T)) \ot Z(T) .$$
Hence there is a canonical equivalence
$$ \alpha_!(\Mor_{\mD'}(\beta(-),Y)) \simeq \int_{T \in \mC} \Mor_{\mC'}(-,\alpha(T))\ot \Mor_{\mD}(\beta(T),Y) $$
and so a canonical equivalence
$$ \alpha_!(\Mor_{\mD'}(\beta(-),Y))(X) \simeq \int_{T \in \mC} \Mor_{\mC'}(X,\alpha(T) \ot \Mor_{\mD}(\beta(T),Y)). $$
\end{proof}

\subsection{Oriented colimits}

\begin{definition}Let $\mC$ be an $\infty$-category.
\begin{enumerate}[\normalfont(1)]\setlength{\itemsep}{-2pt}
\item The oplax weight on $\mC$ is the functor $$\W^\mC_\oplax: \mC \to \Fun(\mC^\circ, \infty\scat) \xrightarrow{\int_\mC} \infty\scat. $$

\item The lax weight on $\mC$ is the functor $$\W^\mC_\lax: \mC^{\circ} \to \Fun(\mC, \infty\scat) \xrightarrow{\int_\mC} \infty\scat.$$
\end{enumerate}

\end{definition} 

\begin{remark}\label{weigrot}
Let $0 \leq n \leq \infty$ and $\mC $ an $n$-category. We assume that the $n$-categorical Grothendieck construction is an equivalence.
By \cref{univ4} the Grothendieck construction sends 
$\W^\mC_\oplax: \mC \to \infty\scat $ to the cocartesian fibration $ \Fun^\oplax(\bD^1,\mC) \to \mC
$ evaluating at the target.
    
\end{remark}

\begin{definition}
Let $\mC$ be an $\infty$-category.

\begin{enumerate}[\normalfont(1)]\setlength{\itemsep}{-2pt}

\item Let $\phi: \mC \to \infty\scat$ be an oriented functor.
The oriented limit of $\phi$ is the $\infty$-category
$$ \overset{\to}{\lim}(\phi):=\L\Mor_{\Fun\boxtimes(\mC,\infty\mathfrak{Cat})}(\W^{\mC}_\oplax, \phi)$$
(the right $\W^{\mC}_\oplax$-weighted limit of $\phi$).

\item Let $\phi: \mC \to \infty\scat$ be an antioriented functor.
The antioriented limit of $\phi$ is the $\infty$-category
$$ \overset{\bar{\to}}{\lim}(\phi):=\R\Mor_{\boxtimes\Fun(\mC,\infty\mathfrak{Cat})}(\W^{\mC^\circ}_\lax, \phi)$$
(the left $\W^{\mC^\circ}_\lax$-weighted limit of $\phi$).

\end{enumerate}

Now we use oriented weights to define oriented colimits.

\begin{definition}

Let $\mC$ be an $\infty$-category, $\mD$ an oriented category and $\phi: \mC \to \mD$ an oriented functor. 

\begin{enumerate}[\normalfont(1)]\setlength{\itemsep}{-2pt}

\item The oriented colimit of $\phi$ is the object of $\mD$ representing the following copresheaf on $\mD:$
$$ X \mapsto \mathrm{-} \overset{\to}{\lim}(\mathrm{Mor}_\mD(-,X)\circ \phi)$$
(the right $\W^{\mC^\circ}_\oplax$-weighted colimit of $\phi$).
Precisely, this copresheaf is the following composition of oriented functors
$$ \mD \to {\boxtimes\Fun}(\mD^\circ,\infty\mathfrak{Cat}) \xrightarrow{\phi^*} {\boxtimes\Fun}(\mC^\circ,\infty\mathfrak{Cat}) \xrightarrow{\R\Mor_{\boxtimes\Fun(\mC^\circ,\infty\mathfrak{Cat})}(\W^{\mC^\circ}_\oplax, -)} \infty\mathfrak{Cat}.$$

\item The oriented limit of $\phi$ is the object of $\mD$ representing the following presheaf on $\mD:$
$$ X \mapsto \overset{\to}{\lim}(\mathrm{Mor}_\mD(X,-)\circ \phi)$$
(the right $\W^{\mC}_\oplax$-weighted limit of $\phi$).
Precisely, this presheaf is the following composition of antioriented functors
$$ \mD^\circ \to {\Fun\boxtimes}(\mD,\infty\mathfrak{Cat}) \xrightarrow{\phi^*} {\Fun\boxtimes}(\mC,\infty\mathfrak{Cat}) \xrightarrow{\L\Mor_{\Fun\boxtimes(\mC,\infty\mathfrak{Cat})}(\W^{\mC}_\oplax, -)} \infty\mathfrak{Cat}.$$

\end{enumerate}

\end{definition}

\end{definition}

\begin{definition}

Let $\mC$ be an $\infty$-category, $\mD$ an antioriented category  and $\phi: \mC \to \mD$ an antioriented functor. 

\begin{enumerate}[\normalfont(1)]\setlength{\itemsep}{-2pt}

\item The antioriented colimit of $\phi$ is the object of $\mD$ representing the following copresheaf on $\mD:$
$$ X \mapsto \overset{\bar{\to}}{\lim}(\mathrm{Mor}_\mD(-,X)\circ \phi)$$
(the left $\W^{\mC}_\lax$-weighted colimit of $\phi$).
Precisely, this copresheaf is the following composition of antioriented functors
$$ \mD \to {\Fun\boxtimes}(\mD^\circ,\infty\mathfrak{Cat}) \xrightarrow{\phi^*} {\Fun\boxtimes}(\mC^\circ,\infty\mathfrak{Cat}) \xrightarrow{\L\Mor_{\Fun\boxtimes(\mC^\circ,\infty\mathfrak{Cat})}(\W^{\mC}_\lax, -)} \infty\mathfrak{Cat}.$$

\item The antioriented limit of $\phi$ is the object of $\mD$ representing the following presheaf on $\mD:$
$$ X \mapsto \overset{\bar{\to}}{\lim}(\mathrm{Mor}_\mD(X,-)\circ \phi)$$
(the left $\W^{\mC^\circ}_\lax$-weighted limit of $\phi$).
Precisely, this presheaf is the following composition of oriented functors
$$ \mD^\circ \to {\boxtimes\Fun}(\mD,\infty\scat) \xrightarrow{\phi^*} {\boxtimes\Fun}(\mC,\infty\mathfrak{Cat}) \xrightarrow{\R\Mor_{\boxtimes\Fun(\mC,\infty\mathfrak{Cat})}(\W^{\mC^\circ}_\lax, -)} \infty\mathfrak{Cat}.$$

\end{enumerate}

\end{definition}

The following is an important example:

\begin{example}

Let $\mC$ be an oriented $\infty$-category that admits right tensors and pullbacks.
Let $Y \to X \leftarrow Z$ be morphisms in $\mC.$

The oriented limit of the functor $(\partial\bD^1)^{\triangleright} \to \mC $
corresponding to $Y \to X \leftarrow Z$ is the pullback
$$Y  \times_{X} X^{\bD^1} \times_{X} X^{\bD^1} \times_{X} Z $$
    
\end{example}

\begin{proof}

There is a canonical equivalence
$$ W:= W^{(\partial\bD^1)^{\triangleright}}_{\oplax} \simeq $$$$ \Map_{(\partial\bD^1)^{\triangleright}}(0,-) \coprod_{\Map_{(\partial\bD^1)^{\triangleright}}(*,-)} \Map_{(\partial\bD^1)^{\triangleright}}(*,-) \boxtimes \bD^1 \coprod_{\Map_{(\partial\bD^1)^{\triangleright}}(*,-)} \Map_{(\partial\bD^1)^{\triangleright}}(*,-) \boxtimes \bD^1 $$$$ \coprod_{\Map_{(\partial\bD^1)^{\triangleright}}(*,-)}\Map_{(\partial\bD^1)^{\triangleright}}(1,-)  $$
in $\Fun((\partial\bD^1)^{\triangleright}, \infty\Cat) .$

Let $F: (\partial\bD^1)^{\triangleright} \to \mC $ be the functor 
corresponding to $Y \to X \leftarrow Z$.

Hence for every $T \in \mC$ there is a canonical equivalence
$$ \RMor_{\Fun((\partial\bD^1)^{\triangleright}, \infty\fcat)}(W, \RMor_{\mC}(T,-) \circ F) \simeq $$
$$ \RMor_{\mC}(T, Y) \times_{\RMor_{\mC}(T, X)} \Fun^\oplax(\bD^1, \RMor_{\mC}(T, X)) \times_{\RMor_{\mC}(T, X)} $$$$ \Fun^\oplax(\bD^1, \RMor_{\mC}(T, X)) \times_{\RMor_{\mC}(T, X)} \RMor_{\mC}(T, Z).$$
This proves the result.
\end{proof}

\subsection{Partial oriented colimits}

\begin{definition}

A marked $\infty$-category is a pair $(\mC,\mE)$ consisting of an $\infty$-category $\mC$ and a collection  $\mE$ of cells of $\mC.$
    
\end{definition}

\begin{notation}Let $X$ be an $\infty$-category.
We assume the Grothendieck construction is an equivalence.
Let $F: S \to \infty\scat$ be a functor classified by a cocartesian fibration $X \to S.$ Let $\mE$ be a collection of cells of $X.$
Let $F^{\mE^{-1}}: S \to \infty\scat$ be the functor classified by the cocartesian fibration $X^{\mE^{-1}} \to S$ of \cref{fibinv}.

\end{notation}

\begin{definition}Let $(\mC, \mE)$ be a marked $\infty$-category.
We assume the Grothendieck construction is an equivalence.
\begin{enumerate}[\normalfont(1)]\setlength{\itemsep}{-2pt}
\item The $\mE$-oplax weight on $\mC$ is $$ \W^{\mC,\mE}_\oplax:= (\W^{\mC}_\oplax)^{\mE^{-1}}.$$

\item The $\mE$-lax weight on $\mC$ is $$ \W^{\mC,\mE}_\lax:=(\W_\lax ^{\mC})^{\mE^{-1}}.$$

\end{enumerate}

\end{definition} 

\begin{remark}
Our argument that the Grothendieck construction is an equivalence is logically independent of this notion of partially (op)lax colimit.
Specifically, we do not use this notion at all in the proof that the Grothendieck construction is an equivalence.    
\end{remark}

\begin{definition}
Let $(\mC, \mE)$ be a marked $\infty$-category.

\begin{enumerate}[\normalfont(1)]\setlength{\itemsep}{-2pt}

\item Let $\phi: \mC \to \infty\scat$ be an oriented functor.
The $\mE$-oriented limit of $\phi$ is the $\infty$-category
$$ \mE-\overset{\to}{\lim}(\phi):=\L\Mor_{\Fun\boxtimes(\mC,\infty\mathfrak{Cat})}(\W^{\mC,\mE}_\oplax, \phi)$$
(the right $\W^{\mC,\mE}_\oplax$-weighted limit of $\phi$).

\item Let $\phi: \mC \to \infty\scat$ be an antioriented functor.
The $\mE$-antioriented limit of $\phi$ is the $\infty$-category
$$ \mE-\overset{\bar{\to}}{\lim}(\phi):=\R\Mor_{\boxtimes\Fun(\mC,\infty\mathfrak{Cat})}(\W^{\mC^\circ,\mE}_\lax, \phi)$$
(the left $\W^{\mC^\circ,\mE}_\lax$-weighted limit of $\phi$).

\end{enumerate}

Now we use oriented weights to define oriented colimits.

\begin{definition}

Let $(\mC, \mE)$ be a marked $\infty$-category, $\mD$ an oriented category and $\phi: \mC \to \mD$ an oriented functor. 

\begin{enumerate}[\normalfont(1)]\setlength{\itemsep}{-2pt}

\item The $\mE$-oriented colimit of $\phi$, denoted by $ \mE-\overset{\to}{\colim} $, is the object of $\mD$ representing the following copresheaf on $\mD:$
$$ X \mapsto \mE \mathrm{-} \overset{\to}{\lim}(\mathrm{Mor}_\mD(-,X)\circ \phi)$$
(the right $\W^{\mC^\circ,\mE}_\oplax$-weighted colimit of $\phi$).
Precisely, this copresheaf is the following composition of oriented functors
$$ \mD \to {\boxtimes\Fun}(\mD^\circ,\infty\mathfrak{Cat}) \xrightarrow{\phi^*} {\boxtimes\Fun}(\mC^\circ,\infty\mathfrak{Cat}) \xrightarrow{\R\Mor_{\boxtimes\Fun(\mC^\circ,\infty\mathfrak{Cat})}(\W^{\mC^\circ,\mE}_\oplax, -)} \infty\mathfrak{Cat}.$$

\item The $\mE$-oriented limit of $\phi$, denoted by $ \mE-\overset{\to}{\lim} $, is the object of $\mD$ representing the following presheaf on $\mD:$
$$ X \mapsto  \mE-\overset{\to}{\lim}(\mathrm{Mor}_\mD(X,-)\circ \phi)$$
(the right $\W^{\mC,\mE}_\oplax$-weighted limit of $\phi$).
Precisely, this presheaf is the following composition of antioriented functors
$$ \mD^\circ \to {\Fun\boxtimes}(\mD,\infty\mathfrak{Cat}) \xrightarrow{\phi^*} {\Fun\boxtimes}(\mC,\infty\mathfrak{Cat}) \xrightarrow{\L\Mor_{\Fun\boxtimes(\mC,\infty\mathfrak{Cat})}(\W^{\mC,\mE}_\oplax, -)} \infty\mathfrak{Cat}.$$

\end{enumerate}

\end{definition}

\end{definition}

\begin{definition}

Let $\mC$ be an $\infty$-category, $\mD$ an antioriented category  and $\phi: \mC \to \mD$ an antioriented functor. 

\begin{enumerate}[\normalfont(1)]\setlength{\itemsep}{-2pt}

\item The $\mE$-antioriented colimit of $\phi$, denoted by $ \mE-\overset{\bar{\to}}{\colim} $, is the object of $\mD$ representing the following copresheaf on $\mD:$
$$ X \mapsto \mE-\overset{\bar{\to}}{\lim}(\mathrm{Mor}_\mD(-,X)\circ \phi)$$
(the left $\W^{\mC,\mE}_\lax$-weighted colimit of $\phi$).
Precisely, this copresheaf is the following composition of antioriented functors
$$ \mD \to {\Fun\boxtimes}(\mD^\circ,\infty\mathfrak{Cat}) \xrightarrow{\phi^*} {\Fun\boxtimes}(\mC^\circ,\infty\mathfrak{Cat}) \xrightarrow{\L\Mor_{\Fun\boxtimes(\mC^\circ,\infty\mathfrak{Cat})}(\W^{\mC,\mE}_\lax, -)} \infty\mathfrak{Cat}.$$

\item The $\mE$-antioriented limit of $\phi$, denoted by $ \mE-\overset{\bar{\to}}{\lim} $, is the object of $\mD$ representing the following presheaf on $\mD:$
$$ X \mapsto \mE-\overset{\bar{\to}}{\lim}(\mathrm{Mor}_\mD(X,-)\circ \phi)$$
(the left $\W^{\mC^\circ,\mE}_\lax$-weighted limit of $\phi$).
Precisely, this presheaf is the following composition of oriented functors
$$ \mD^\circ \to {\boxtimes\Fun}(\mD,\infty\scat) \xrightarrow{\phi^*} {\boxtimes\Fun}(\mC,\infty\mathfrak{Cat}) \xrightarrow{\R\Mor_{\boxtimes\Fun(\mC,\infty\mathfrak{Cat})}(\W^{\mC^\circ,\mE}_\lax, -)} \infty\mathfrak{Cat}.$$

\end{enumerate}

\end{definition}

\begin{proposition}\label{weichara} Let $\mC$ be an $\infty$-category, $\mE$ a collection of 1-morphisms and $H: \mC \to \infty\Cat$ a functor and
$$ \phi: \W^{\mC}_\oplax \to H $$ a natural transformation of functors $\mC \to \infty\Cat$.
The following are equivalent:

\begin{enumerate}[\normalfont(1)]\setlength{\itemsep}{-2pt}
\item The natural transformation $\phi: \W^{\mC}_\oplax \to H $ of functors $\mC \to \infty\Cat$ 
factors through $\W^{\mC}_\oplax \to \W^{\mC,\mE}_\oplax$.

\item For every $X \in \mC$ the induced functor
$$ (\W^{\mC}_\oplax)_X \simeq \mC_{//^\oplax X} \to H(X) $$ 
inverts morphisms which lie over morphisms of $\mE$ and belong to
$\mC_{/X}.$

\item For every $X \in \mC$ the induced functor
$$ (\W^{\mC}_\oplax)_X \simeq \mC_{//^\oplax X} \to H(X) $$ 
inverts morphisms which lie over morphisms of $\mE$ and are the unique morphism to the final object in $\mC_{/X}.$

\end{enumerate}

\end{proposition}

\begin{proof}

(2) trivially implies (3). (3) implies (2) since every morphism 
$A \to B $ in $\mC_{/X}$ which lies over a morphism of $\mE$
is the image of the unique morphism $A \to B $ in $\mC_{/B}$, 
which also lies over a morphism of $\mE$, under the induced functor
$\mC_{//^\oplax A} \to \mC_{//^\oplax B}$, which restricts to a functor
$\mC_{/ A} \to \mC_{/ B}$.

It remains to see that (1) is equivalent to (3).
(1) is equivalent to say that for every morphism $A \to B$ in $\mC$ that belongs to $\mE$ the induced map $$\int \phi: \Fun^\oplax(\bD^1,\mC) \simeq \int \W^{\mC}_\oplax \to \int H $$ of cocartesian fibrations over $\mC$ sends the commutative square 
\begin{equation}\label{sqppj}
\xymatrix{
A \ar[r] \ar[d]^= & B \ar[d]^= \\
A \ar[r] & B
}
\end{equation}
to a cocartesian morphism.
The latter commutative square factors as 
\begin{equation}\label{gbbnmi}
\xymatrix{
A \ar[r]^= \ar[d]^= & A \ar[r] \ar[d] & B \ar[d]^= \\
A \ar[r] & B \ar[r]^= & B
}
\end{equation}
in $\Fun^\oplax(\bD^1,\mC).$
The left hand commutative square is cocartesian with respect to evaluation at the target $\Fun^\oplax(\bD^1,\mC) \to \mC$ and so is sent by $\int \phi: \Fun^\oplax(\bD^1,\mC) \to \int H $ to a cocartesian morphism.
The right hand commutative square is in $\mC_{/B}$ and so is sent by $\int \phi: \Fun^\oplax(\bD^1,\mC) \to \int H $ to a morphism in $H(B).$
Hence for every morphism $A \to B$ in $\mC$ that belongs to $\mE$
the commutative square (\ref{sqppj}) is sent by $\int \phi: \Fun^\oplax(\bD^1,\mC) \to \int H $ to a cocartesian morphism if and only if for every morphism $A \to B$ in $\mC$ that belongs to $\mE$
the right hand commutative square of (\ref{gbbnmi}), which is the unique morphism $A \to B$ in $\mC_{/B}$ is inverted by
$(\W^{\mC}_\oplax)_B \simeq \mC_{//^\oplax B} \to H(B) $.
In other words (1) is equivalent to (3).
\end{proof}

\subsection{Lax colimits}

\begin{definition}
Let $(\mC, \mE)$ be a marked $\infty$-category and $\phi: \mC \to \infty\scat$ a functor.

\begin{enumerate}[\normalfont(1)]\setlength{\itemsep}{-2pt}

\item The $\mE$-oplax colimit of $\phi$, denoted by $\mE-\overset{\oplax}{\colim}(\phi)$, is the $\mE$-oriented colimit of $\phi $ viewed as oriented functor.
\item The $\mE$-lax colimit of $\phi$, denoted by $\mE-\overset{\lax}{\colim}(\phi)$, is the $\mE$-antioriented colimit of $\phi $ viewed as antioriented functor.
\item The $\mE$-oplax limit of $\phi$, denoted by $\mE-\overset{\oplax}{\lim}(\phi)$, is the $\mE$-oriented limit of $\phi $ viewed as oriented functor.
\item The $\mE$-lax limit of $\phi$, denoted by $\mE-\overset{\lax}{\lim}(\phi)$, is the $\mE$-antioriented limit of $\phi $ viewed as antioriented functor.

\end{enumerate}

\end{definition}

\begin{remark}

We can also form the $\mE$-oriented colimit of the oriented functor
$\phi: \mC \to \infty\scat \subset \infty\mathfrak{Cat},$
which is very different from the $\mE$-oplax colimit.
  
\end{remark}

\begin{proposition}Let $(\mC, \mE)$ be a marked $\infty$-category and $\phi: \mC \to \infty\scat$ a functor.

\begin{enumerate}[\normalfont(1)]\setlength{\itemsep}{-2pt}
\item There is a canonical epimorphism from the $\mE$-oriented colimit of $\phi$ to the $\mE$-oplax colimit of $\phi.$ 

\item There is a canonical epimorphism from the $\mE$-antioriented colimit of $\phi$ to the $\mE$-lax colimit of $\phi.$

\item There is a canonical monomorphism from the $\mE$-oplax limit of $\phi$ to the $\mE$-oriented limit of $\phi$.

\item There is a canonical monomorphism from the $\mE$-lax limit of $\phi$
to the $\mE$-antioriented limit of $\phi$.

\end{enumerate}
    
\end{proposition}

\begin{proof}

This follows immediately from \cref{funcompo}.
\end{proof}

\begin{proposition} Let $(\mC, \mE)$ be a marked $\infty$-category and $\phi: \mC \to \infty\scat$ a functor.
For every $X \in \infty\scat$ there are canonical equivalences
$$ \Fun(X,\mE-\overset{\oplax}{\lim}(\phi)) \simeq \Mor_{\Fun(\mC, \infty\scat)}(W^{\oplax,\mC},\Fun(X,-) \circ \phi), $$
$$ \Fun(X,\mE-\overset{\lax}{\lim}(\phi)) \simeq \Mor_{\Fun(\mC, \infty\scat)}(W^{\lax,\mC^\circ},\Fun(X,-) \circ \phi). $$
$$ \Fun(\mE-\overset{\oplax}{\colim}(\phi),X) \simeq \Mor_{\Fun(\mC^\circ, \infty\scat)}(W^{\oplax,\mC^\circ},\Fun(-,X) \circ \phi), $$
$$ \Fun(\mE-\overset{\lax}{\colim}(\phi),X) \simeq \Mor_{\Fun(\mC^\circ, \infty\scat)}(W^{\lax,\mC},\Fun(-,X) \circ \phi). $$

\end{proposition}

\begin{proof}

This follows immediately from \cref{presheaves}
and that $\Fun(\mC, \infty\scat)$ is tensored and cotensored so that for every $K \in \infty\scat$ there is a canonical equivalence 
$$\Map_{\infty\scat}(K,\Mor_{\Fun(\mC, \infty\scat)}(W^{\oplax,\mC},\Fun(X,-) \circ \phi)) \simeq \Mor_{\Fun(\mC, \infty\scat)}(W^{\oplax,\mC},\Fun(K \times X,-) \circ \phi),$$
and similar for the other cases.
\end{proof}

\begin{proposition}\label{repre}

Let $0 \leq \n \leq \infty.$
We assume that the $n$-categorical Grothendieck construction is an equivalence.
Let $\mA, \mB, \mC,\mD$ be $n$-categories and $\mC \to \mD$ a cocartesian fibration classifying a functor $F: \mD \to n\Cat$.
Let $ \langle \mC, \mA \rangle \to \mD$ be the cartesian fibration classifying 
the functor $\Fun(-,\mA) \circ F: \mD^{\circ} \to n\widehat{\Cat}.$
For every functor $\alpha: \mB \to \mD$ there is a canonical equivalence
$$ \Fun_\mD(\mB, \langle \mC, \mA \rangle) \simeq \Fun(\mB\times_\mD \mC, \mA).$$
Moreover, an $n$-morphism $\bD^n \to \langle \mC, \mA \rangle$
is (co)cartesian if the induced functor 
$\bD^n \times_\mD \mC \to \mA$ inverts (co)cartesian cells. 

\end{proposition}

\begin{proof}

Let $ \mD \times_{n\Cat} n\coCart$ be the pullback along the functor $$\mD \xrightarrow{F} n \Cat \xrightarrow{(-)\times \mA^\circ} n \Cat.$$ 

Let $ \langle \mC, \mA \rangle \subset \mD \times_{n\Cat} n\coCart$ be the full cartesian subfibration spanned by the cocartesian fibrations of $n$-categories over $\mY \to F(X) \times \mA^\circ $ for some $X \in \mD$ 
whose fibers over objects of $F(X)$ are representable cocartesian fibrations over $\mA^\circ.$
Since the Grothendieck construction for $n$-categories holds,
the fiber of the cartesian fibration $\langle \mC, \mA \rangle$ over any $X \in \mD$ is the $n$-category $$\Fun(F(X),\mA) \subset \Fun(F(X),\Fun(\mA^\circ, n\Cat)) \simeq \Fun(F(X) \times \mA^\circ, n\Cat) \simeq n \Cat^{\cocart}_{/ F(X) \times \mA^\circ}. $$
Since $\mD$ is an $n$-category, also $ \langle \mC, \mA \rangle $ is an $n$-category.
In particular, the embedding $\langle \mC, \mA \rangle \subset \mD \times_{n\Cat} n\coCart$ induces an embedding 
$$\langle \mC, \mA \rangle \subset \mD \times_{\iota_n(n\Cat)} \iota_n(n\coCart).$$

\cref{clas} implies that the cartesian fibration $ \iota_n(n\coCart) \to \iota_n(n\Cat)$
classifies the functor $$\iota_n \circ \Fun(-,n\Cat) : \iota_n(n\Cat) \to \iota_n(n \widehat{\Cat}).$$
Thus the pullback $\mD \times_{\iota_n(n\Cat)} \iota_n(n\coCart) \to \mD$
classifies the functor $$\iota_n \circ \Fun(-,n\Cat) \circ (-\times \mA^\circ) \circ F \simeq \iota_n \circ \Fun(-,\Fun(\mA^\circ,n\Cat)) \circ F: \mD^\circ \to n \widehat{\Cat},$$
and the full cartesian subfibration $\langle \mC, \mA \rangle \to \mD$ classifies the functor $$\iota_n \circ \Fun(-,\mA) \circ F \simeq \Fun(-,\mA) \circ F: \mD^\circ \to n\widehat{\Cat}.$$

We will construct a canonical equivalence
$$ \Fun_\mD(\mB, \langle \mC, \mA \rangle) \simeq \Fun(\mB\times_\mD \mC, \mA).$$
The Grothendieck construction $$\Fun(\mB,n\Cat) \simeq n\Cat^\cocart_{/\mB}$$ provides an equivalence
$$ \Fun_\mD(\mB, \mD \times_{n\Cat}\Fun(\bD^1,n\Cat)) \simeq \Fun_{n\Cat}(\mB, \Fun(\bD^1,n\Cat)) $$$$ \simeq \Fun(\mB,n\Cat)_{/((-)\times \mA^\circ) \circ F \circ \alpha} \simeq (n\Cat^\cocart_{/\mB})_{/ \mB\times_\mD \mC \times \mA^\circ}.$$
By \cref{paracocart} the latter equivalence restricts to an equivalence
$$ \Fun_\mD(\mB, \mD \times_{n\Cat} n\coCart) \simeq n\Cat^\cocart_{/ \mB\times_\mD \mC \times \mA^\circ}.$$
The Grothendieck construction provides an equivalence
$$n\Cat^\cocart_{/ \mB\times_\mD \mC \times \mA^\circ}\simeq 
\Fun(\mB\times_\mD \mC \times \mA^\circ,n\Cat)
 \simeq \Fun(\mB\times_\mD \mC, \Fun(\mA^\circ,n\Cat)).$$
The resulting equivalence 
$$ \Fun_\mD(\mB, \mD \times_{n\Cat} n\coCart) \simeq \Fun(\mB\times_\mD \mC, \Fun(\mA^\circ,n\Cat)) $$
restricts to an equivalence 
$$ \Fun_\mD(\mB, \langle \mC, \mA \rangle) \simeq \Fun(\mB\times_\mD \mC, \mA)$$
by definition of $\langle \mC, \mA \rangle$.
\end{proof}

\begin{theorem}\label{laxgro}
Let $0 \leq \n \leq \infty.$ 
We assume that the $n$-categorical Grothendieck construction is an equivalence.
Let $(\mD, \mE) $ be a marked $n$-category, $\mC$ an $n$-category and
$\mC \to \mD$ a cocartesian fibration classifying a functor $F: \mD \to n\Cat$. Let $\bar{\mE} $ be the class of (co)cartesian cells in $\mC$ whose image in $\mD$ belongs to $\mE.$
There is a canonical equivalence
$$ \mE-\overset{\lax}{\colim}(F) \simeq \mC[\bar{\mE}^{-1}].$$    
\end{theorem}

\begin{proof}
By \cref{dimensio} the $\infty$-category $\Fun^\oplax(\bD^1,\mD)$ is an $n$-category. By \cref{slicecocart} evaluation at the source $\Fun^\oplax(\bD^1,\mD)\to \mD$ is a cartesian fibration and by \cref{weigrot} classifies the functor $ \W^{\mD}_\lax: \mD^\circ \to n\Cat$.

Let $\mA$ be an $n$-category.
Let $\langle \mC, \mA \rangle \to \mD$ be the cartesian fibration 
of $n$-categories classifying the functor $$\mD^\circ \to n\Cat, \Fun(-,\mA) \circ F.$$
The Grothendieck construction provides a canonical equivalence 
$$ \Mor_{\Fun(\mD^\circ,\infty\scat)}(\W^{\mD}_\lax, \Fun(-,\mA) \circ F)\simeq \Fun^\cart_{\Fun^\oplax(\{0\},\mD)}(\Fun^\oplax(\bD^1,\mD), \langle \mC, \mA \rangle).$$

By \cref{envelo} there is a canonical equivalence 
$$ \Fun^\cart_{\Fun^\oplax(\{0\},\mD)}(\Fun^\oplax(\bD^1,\mD), \langle \mC, \mA \rangle) \simeq \Fun_{\mD}(\mD, \langle \mC, \mA \rangle).$$
By \cref{repre} there is a canonical equivalence
$ \Fun_{\mD}(\mD, \langle \mC, \mA \rangle) \simeq \Fun(\mC,\mA).$
We obtain an equivalence 
$$ \Mor_{\Fun(\mD^\circ,\infty\scat)}(\W^{\mD}_\lax, \Fun(-,\mA) \circ F)\simeq \Fun(\mC,\mA).$$
The latter equivalence restricts to an equivalence 
$$\Fun(\overset{\mE-\lax}{\colim}(F), \mA) \simeq \Mor_{\Fun(\mD^\circ,\infty\scat)}(\W^{\mD, \mE}_\lax, \Fun(-,\mA) \circ F) \simeq \Fun(\mC[\bar{\mE}^{-1}],\mA). $$
\end{proof}

\begin{corollary}\label{cocartcolim}

Let $0 \leq \n \leq \infty$ and $\mC$ an $n$-category.
We assume that the $n$-categorical Grothendieck construction is an equivalence.
The inclusion $n\Cat^\cocart_{/\mC} \subset {n\Cat}_{/\mC}$
preserves small colimits.
    
\end{corollary}

\begin{proof}

By assumption the Grothendieck construction gives an 
equivalence
$\Fun(\mC, n\Cat) \simeq n\Cat^\cocart_{/\mC}.$
The forgetful functor ${n\Cat}_{/\mC} \to {n\Cat}$
preserves and so detects small colimits. 
So it suffices to see that the composition 
$\Fun(\mC, n\Cat) \simeq n\Cat^\cocart_{/\mC} \subset {n\Cat}_{/\mC} \to {n\Cat}$
preserves colimits. 
By \cref{laxgro} the latter composition identifies with
the left adjoint functor $\overset{\lax}{\colim}: \Fun(\mC, n\Cat) \to {n\Cat}.$  
\end{proof}

Next we study cofinality of (op)lax colimits.

\begin{definition}Let $(\mC,\mE)$ be a marked $\infty$-category.
A functor $\phi: \mC \to \mD$ of $\infty$-categories is $\mE$-cofinal if for every functor $H: \mD \to \mD'$ that admits a $\mE$-lax colimit
such that $H \circ \phi$ also admits a $\mE$-lax colimit,
the canonical morphism $$ \colim^\mE(H \circ \phi) \to \colim^\mE(H) $$ is an equivalence.

A functor $\phi: \mC \to \mD$ of $\infty$-categories is $\mE$-lax
final if $\phi^{\co\op}: \mC^{\co\op} \to \mD^{\co\op}$ is $\mE$-lax cofinal.
    
\end{definition}

\begin{remark}
A functor $\phi: \mC \to \mD$ of $\infty$-categories is $\mE$-lax final 
if and only if for every functor $H: \mD \to \infty\scat$ 
the canonical morphism $${\lim}^\mE(H) \to {\lim}^\mE(H \circ \phi) $$ is an equivalence.
This holds because for every functor $H: \mD \to \mD'$ that admits a limit
such that $H \circ \phi$ also admits a limit, and every $Z \in \mD'$ 
the induced morphism $$ \Map_{\mD'}(Z, {\lim}^\mE(H)) \to \Map_{\mD'}(Z, {\lim}^\mE(H \circ \phi)) $$ 
identifies with the canonical morphism $$\lim(\Map_{\mD'}(Z,-)\circ H)  \to \lim(\Map_{\mD'}(Z,-)\circ H \circ \phi). $$ 

Dually, a functor $\phi: \mC \to \mD$ of $\infty$-categories is $\mE$-lax cofinal 
if and only if for every functor $H: \mD \to \infty\scat$ 
the canonical morphism $$\colim(H \circ \phi) \to \colim(H) $$ is an equivalence.
    
\end{remark}

\begin{definition}Let $(X,\mE)$ be a marked $\infty$-category.
The localization of $X$ with respect to $\mE$ is the 
pushout $$ X[\mE^{-1}]:= X \coprod_{(\coprod_{n > 0}\mE_n \times \bD^n)} (\coprod_{n > 0} \mE_n \times \bD^{n-1}).$$

\end{definition}

\begin{lemma}

Let $(X,\mE)$ be a marked $\infty$-category.
For every $\infty$-category $Y$ the induced functor
$$ \Fun(X[\mE^{-1}],Y) \to \Fun(X,Y) $$ is fully faithful and the essential image precisely consists of the functors $X \to Y$ sending $\mE$ to invertible cells.

\end{lemma}

\begin{proof}
The induced functor
$ \Fun(X[\mE^{-1}],Y) \to \Fun(X,Y) $ identifies with the functor
$$ \Fun(X[\mE^{-1}],Y) \simeq \Fun(X,Y) \times_{\prod_{n > 0}\Fun(\mE_n, \Fun(\bD^n,Y))} \prod_{n > 0}\Fun(\mE_n,\Fun(\bD^{n-1},Y)) \to \Fun(X,Y). $$
This functor is fully faithful by \cite[Corollary 4.1.4.]{oriented} and the essential image precisely consists of the functors $X \to Y$ sending $\mE$ to invertible cells.
\end{proof}

\cref{laxgro} implies the following $\infty$-categorical version of Quillen's Theorem A:

\begin{theorem}\label{QuillenA}
Let $(Y,\mE)$ be a marked $\infty$-category.
A functor $\phi: Y \to X $ is $\mE$-lax cofinal if
and only if for every $t \in X$
the induced functor $$\{ t \} \,{\vec{\times}}_X Y \to X_{t//^\oplax} $$ induces an equivalence 
$$ (\{ t \} \,{\vec{\times}}_X Y)[{\bar{\mE}^{-1}}] \to (X_{t//^\oplax})
[(\overline{\phi(\mE)})^{-1}].$$ 

\end{theorem}

\begin{proof}

The functor $\phi$ is $\mE$-lax cofinal 
if and only if for every functor $H: X \to \infty\scat$ 
the canonical functor $$ \colim^\mE(H \circ \phi) \to \colim^\mE(H) $$ is an equivalence.
The functor $\phi^*: \Fun(X, \infty\scat) \to \Fun(Y, \infty\scat)$ and the functors $$ \Fun(X, \infty\scat) \to \infty\scat, \Fun(Y, \infty\scat) \to \infty\scat $$ forming the colimit preserve small colimits and tensors. 
Moreover the $\infty$-category $ \Fun(X, \infty\scat)$ is generated under small colimits and tensors by the corepresentable functors.
Hence we can reduce to the case that $H= \Mor_X(t,-): X \to \infty\scat$ for some $t \in X.$
By \cref{laxgro} the $\mE$-lax colimit of a functor $F: Y \to \infty\scat$ is $(\int_Y F)[{\bar{\mE}^{-1}}].$
Thus $\phi$ is $\mE$-lax cofinal if and only if the induced functor $$\{ t \} \,{\vec{\times}}_X Y \simeq Y \times_X X_{t//^\oplax} \simeq Y \times_X \int_X \Mor_X(t,-) \simeq \int_Y \Mor_\X(t,-) \circ \phi \to \int_X \Mor_X(t,-) \simeq X_{t//^\oplax} $$ induces an equivalence from the localization with respect to $\bar{\mE}$ 
to the localization with respect to $\overline{\phi(\mE)}$.
\end{proof}

\subsection{Oriented pullbacks}

\begin{proposition}

Let $\mC$ be an oriented $\infty$-category.

\begin{enumerate}[\normalfont(1)]\setlength{\itemsep}{-2pt}

\item Let $Y \to X \leftarrow Z$ morphisms in $\mC.$

The oriented pullback is the $\{0\}^{\triangleright} $-oplax pullback 
of the functor $(\partial\bD^1)^{\triangleright} \to \mC $
corresponding to $Y \to X \leftarrow Z$.

\item Let $Y \leftarrow X \to Z$ morphisms in $\mC.$

The oriented pushout is the $\{0\}^{\triangleleft} $-oplax pushout 
of the functor $(\partial\bD^1)^{\triangleleft} \to \mC $
corresponding to $Y \leftarrow X \to Z. $
    
\end{enumerate}
    
\end{proposition}

\begin{proof}

(1): The $\{0\}^{\triangleright} $-oplax weight is the functor 
$W: (\partial\bD^1)^{\triangleright} \to \infty\Cat$ classifying the diagram 
$$ (\partial\bD^1)^{\triangleright}_{/0} \to (\partial\bD^1)_{/*}^{\triangleright}[(\{0\}^{\triangleright})^{-1}] \leftarrow (\partial\bD^1)^{\triangleright}_{/1}.$$ The latter identifies with the diagram $ \{0\} \subset \bD^1 \supset \{1\}$.
We prove that the oriented pullback of $Y \to X \leftarrow Z$ is the $W$-weighted colimit of the functor $F: (\partial\bD^1)^{\triangleright} \to \mC $
corresponding to $Y \to X \leftarrow Z$.

There is a canonical equivalence
$$ W \simeq \Map_{(\partial\bD^1)^{\triangleright}}(0,-) \coprod_{\Map_{(\partial\bD^1)^{\triangleright}}(*,-) \boxtimes \{ 0\}} \Map_{(\partial\bD^1)^{\triangleright}}(*,-) \boxtimes \bD^1 \coprod_{\Map_{(\partial\bD^1)^{\triangleright}}(*,-) \boxtimes \{ 1\}} \Map_{(\partial\bD^1)^{\triangleright}}(1,-)$$
in $\Fun((\partial\bD^1)^{\triangleright}, \infty\Cat).$

Hence for every $T \in \mC$ there is a canonical equivalence
$$ \RMor_{\Fun((\partial\bD^1)^{\triangleright}, \infty\fcat)}(W, \RMor_{\mC}(T,-) \circ F) \simeq $$
$$ \RMor_{\mC}(T, Y) \times_{\Fun^\oplax(\{0 \},\RMor_{\mC}(T, X))}  \Fun^\oplax(\bD^1, \RMor_{\mC}(T, X))  \times_{\Fun^\oplax(\{1 \},\RMor_{\mC}(T, X))} \RMor_{\mC}(T, Z).$$

This proves (1). The oriented pushout in $\mC$ is the oriented pullback in $\mC^\op.$
\end{proof}

\begin{corollary}

Let $\mC$ be an antioriented $\infty$-category.

\begin{enumerate}[\normalfont(1)]\setlength{\itemsep}{-2pt}

\item Let $Y \leftarrow X \to Z$ be morphisms in $\mC.$

The antioriented pushout is the $\{0\}^{\triangleleft} $-lax colimit 
of the functor $(\partial\bD^1)^{\triangleleft} \to \mC $
corresponding to $Y \leftarrow X \to Z. $

\item Let $Y \to X \leftarrow Z$ be morphisms in $\mC.$

The antioriented pullback is the $\{0\}^{\triangleleft} $-lax limit 
of the functor $(\partial\bD^1)^{\triangleright} \to \mC $
corresponding to $Y \to X \leftarrow Z$.

\end{enumerate}

\end{corollary}

\begin{proposition}\label{lemkanex}
Let $X$ be an $\infty$-category and $n \geq 0$. Every functor $f: \tau_n(X) \to Y$ is the left Kan extension of the functor $X \to \tau_{n}(X) \xrightarrow{f} Y$ along $X \to \tau_n(X).$
    
\end{proposition}

\begin{proof}

Let $X'$ be the pushout $$ X \coprod_{(\coprod_{m \geq n}\Fun(\bD^m,X) \times \bD^m)} (\coprod_{m \geq n} \Fun(\bD^{m},X) \times \bD^{n}).$$
Then for every $\infty$-category $Z$ the canonical functor $X \to X'$ induces an equivalence 
$$ \Fun(X',Y) \to \Fun(X,Y) \times_{\prod_{m \geq n}\Fun(\Fun(\bD^m,X),\Fun(\bD^m,Y))} \prod_{m \geq n}\Fun(\Fun(\bD^m,X),\Fun(\bD^{n},Y)). $$

By \cite[Corollary 4.1.4.]{oriented} for every $m \geq n$ the induced functor 
$\Fun(\bD^{n},Y) \to \Fun(\bD^{m},Y) $ is fully faithful.

Hence the projection 
$$ \Fun(X',Y) \to \Fun(X,Y) \times_{\prod_{m \geq n}\Fun(\Fun(\bD^m,X),\Fun(\bD^m,Y))} \prod_{m \geq n}\Fun(\Fun(\bD^m,X),\Fun(\bD^{n},Y)) \to \Fun(X,Y)$$ is fully faithful and the essential image is the full subcategory of functors 
$X \to Y$ inverting all morphisms of dimension $m > n$.
In particular, we find that the functor $X \to X'$ is the functor
$\alpha: X \to \tau_n(X).$

Thus the induced functor
$$ \alpha^* : \Fun(\tau_n(X),Y) \to \Fun(X,Y) $$ is fully faithful.
Hence for every functors $f,g: \tau_n(X) \to Y $ the induced functor
$$ \Mor_{\Fun(\tau_n(X),Y)}(f,g) \to \Mor_{\Fun(X,Y)}(f \circ \alpha ,g \circ \alpha) $$
is an equivalence. So $f$ is the left Kan extension of the functor $X \xrightarrow{\alpha} \tau_n(X) \xrightarrow{f} Y$ along $\alpha: X \to \tau_n(X).$
\end{proof}

\begin{definition}Let $n \geq 0$.
A Kan fibration is a cocartesian fibration whose fibers are spaces, i.e. a left fibration, such that the fiber transports are equivalences.
 
\end{definition}

\begin{proposition}
Let $X$ be an $\infty$-category.
A left fibration $Y \to X$ is a Kan fibration if and only if 
the canonical commutative square
$$\begin{xy}
\xymatrix{
Y \ar[d] \ar[r]
& \tau_0 Y \ar[d] 
\\ 
X \ar[r] & \tau_0 X}
\end{xy}$$	
is a pullback square.
    
\end{proposition}

\begin{proof}

If the square of the statement is a pullback square, the functor $Y \to X$ is a Kan fibration since Kan fibrations are stable under base change.

We prove the converse. Let $\phi: Y \to X$ be a Kan fibration.
Then $\phi: Y \to X$ classifies a functor $g: X \to \iota_0(\mS) $ and so extends to a functor $f: \tau_0(X) \to \mS.$
Let $\phi': Y' \to \tau_0(X)$ be the left fibration classifying $f.$
Then there is a pullback square
$$\begin{xy}
\xymatrix{
Y \ar[d]^{\phi} \ar[r]
& Y' \ar[d]^{\phi'}
\\ 
X \ar[r]^\alpha & \tau_0 X}
\end{xy}$$	
We will construct a canonical equivalence $Y' \simeq \tau_0 Y$ over $\tau_0 X.$

By \cref{lemkanex} the functor $f: \tau_0 X  \to \mS $ is the left Kan extension of the functor $g = f \circ \alpha: X \to \mS$ along $\alpha: X \to \tau_0(X).$ In other words, for every functor $h: \tau_0 X \to \mS $ the induced functor
\begin{equation}\label{kolt} \alpha^*: \Mor_{\Fun(\tau_0 X,\mS)}(f, h) \to \Mor_{\Fun(X,\mS)}(g, h \circ \alpha) \end{equation}
is an equivalence. Let $Z \to \tau_0 X $ be the left fibration classifying $h: \tau_0 X \to \mS $. Then $Z$ is a space since every morphism in $Z$ is cocartesian over its image in $\tau_0 X$ and so is an equivalence.
Via the Grothendieck construction the equivalence (\ref{kolt}) identifies with the canonical functor 
$$ \Fun_{\tau_0 X}(Y', Z) \to \Fun_{X}(Y, X \times_{\tau_0 X} Z),$$
which therefore is an equivalence.

The induced functor
$$ \Fun_{\tau_0 X}(\tau_0(Y), Z) \to \Fun_{X}(Y, X \times_{\tau_0 X} Z)$$ identifies with the following canonical equivalence
$$\Fun_{\tau_0 X}(\tau_0(Y), Z) \simeq \{ \tau(\phi)\} \times_{\Fun(\tau_0 Y , \tau_0 X)} \Fun(\tau_0 Y, Z) \simeq
\{ \alpha \circ \phi \} \times_{\Fun(Y, \tau_0 X)} \Fun(Y, Z) $$$$ \simeq \Fun_{\tau_0 X }(Y, Z) \simeq 
\Fun_X(Y, X \times_{\tau_0 X} Z). $$
Hence there is a canonical equivalence $Y' \simeq \tau_0 Y $ over $\tau_0(X) .$
\end{proof}

\begin{corollary}\label{Kanfi}
Let $n\geq 0$ and $Y \to X$ a Kan fibration and $t \in X.$
The fiber sequence $Y_t \to Y \to X$ is preserved by the functor
$\tau_0: \infty\Cat \to \mS.$
    
\end{corollary}

\begin{lemma}\label{clafib} Let $X$ be an $\infty$-category.

\begin{enumerate}[\normalfont(1)]\setlength{\itemsep}{-2pt}
\item The embedding $ \LFib_{X} \subset \infty\scat^{\cocart}_{/X}  $
admits a left adjoint, which we call the classifying left fibration.

\item For every $t \in X$ the functor
$\infty\scat_{/X}^\cocart  \to \infty\scat$ which takes the fiber over $t$,
and the forgetful functor $\infty\scat_{/X}^\cocart  \to \infty\scat$, 
send local equivalences to functors inverted by $\tau_0.$

\end{enumerate}
\end{lemma}

\begin{proof}
(1): The embedding $ \LFib_{X} \subset \infty\scat^{\cocart}_{/X} $ identifies via the Grothendieck construction with the induced embedding
$ \Fun(X,\mS) \subset \Fun(X,\infty\scat),$
which is right adjoint to the functor
$$\Fun(X,\tau_0): \Fun(X,\infty\scat) \to \Fun(X,\mS).$$

(2): The first part of (2) follows from the construction of the left adjoint. We prove the second part of (2), the assertion regarding the forgetful functor.

Let $Y \to X$ be a cocartesian fibration and $Y' \to X$ the classifying left fibration. We prove that the universal map $\theta: Y \to Y' $
of cocartesian fibrations over $X$ is inverted by $\tau_0.$
Let $F: X \to \infty\scat$ be the functor classified by $Y \to X$.
Let $\bj: \mS \subset \infty\scat$ be the canonical embedding.
By definition, $Y' \to X$ is classified by the functor
$ \bj \circ \tau_0 \circ F: X \to \infty\scat. $
By \cref{laxgro} the universal map $\theta: Y \to Y' $
of cocartesian fibrations over $X$ identifies with the canonical functor
$$ \overset{\lax}{\colim}(F) \to \overset{\lax}{\colim}(\bj \circ \tau_0 \circ F).$$

The functor $\tau_0: \infty\scat \to \mS$ preserves lax colimits because it is a left adjoint. Hence the map $$\tau_0(\theta): \tau_0(Y) \to \tau_0(Y') $$ identifies with the canonical equivalence
$$ \tau_0(\overset{\lax}{\colim}(F)) \simeq \overset{\lax}{\colim}(\tau_0 \circ F) \simeq \overset{\lax}{\colim}(\tau_0 \circ \bj \circ \tau_0 \circ F) \simeq \tau_0(\overset{\lax}{\colim}(\bj \circ \tau_0 \circ F)). $$
\end{proof}

\begin{definition}
A functor is a weak equivalence if it is inverted by $\tau_0: \infty\fcat\to \mS.$
    
\end{definition}

\begin{remark}
Let $\phi: Y \to X$ be a cocartesian fibration. By construction of the
classifying left fibration we have the following:
The classifying left fibration of $\phi$ is a Kan fibration if and only if the fiber transports of $\phi$ are weak equivalences.
    
\end{remark}

\begin{proposition}\label{cocafi}

Let $Y \to X$ be a cocartesian fibration whose fiber transports are weak equivalences.
For every $t \in X $ the fiber sequence $Y_t \to Y \to X$ is preserved by the functor
$\tau_0: \infty\fcat \to \mS.$

\end{proposition}

\begin{proof}

Let $\phi: Y \to X$ be a cocartesian fibration. Let $\phi': Y' \to X$ be the classifying left fibration of $\phi$ and $\theta: Y \to Y'$
the universal map of cocartesian fibrations over $X.$
By assumption the functor $\phi'$ is a Kan fibration.

We consider the following commutative diagram of pullback squares:
$$\begin{xy}
\xymatrix{
Y_t \ar[d]^{\theta_t} \ar[r]
& Y \ar[d]^{\theta}
\\ 
Y'_t \ar[r] \ar[d] & Y' \ar[d]^{\phi'}
\\ 
\{ t \} \ar[r] & X}
\end{xy}$$	

We want to see that the outer pullback square is preserved by $\tau_0: \infty\fcat \to \mS.$
By \cref{Kanfi} the bottom pullback square is preserved by $\tau_0.$
So it suffices to see that the top pullback square is preserved by $\tau_0.$
By \cref{clafib} the functor $\tau_0: \infty\fcat \to \mS $
inverts both vertical functors $\theta_t$ and $\theta$ in the top pullback square.
\end{proof}

\begin{proposition}\label{envlemo}Let $Y \to X$ be a functor. There is a canonical equivalence 
$$ \tau_0(Y \,{\vec{\times}}_X X) \simeq \tau_0(Y)$$ over $\tau_0(X).$

\end{proposition}

\begin{proof}Let $\alpha: X \to \tau_0(X)$ be the universal functor.
The adjunction $$\tau_0: \infty\scat_{/X} \rightleftarrows \mS_{/\tau_0(X)} : \alpha^* $$
restricts to an adjunction $$\tau_0: \infty\scat^\cocart_{/X} \rightleftarrows \mS_{/\tau_0(X)} : \alpha^*. $$
Since $\alpha^*: \mS_{/\tau_0(X)} \to \infty\scat_{/X} $ factors as
$\mS_{/\tau_0(X)} \xrightarrow{\alpha^*} \infty\scat^\cocart_{/X}  \subset \infty\scat_{/X} $, uniqueness of left adjoints implies that 
$ \tau_0: \infty\scat_{/X} \to \mS_{/\tau_0(X)}$ factors as
$  \infty\scat_{/X} \xrightarrow{\Env}  \infty\scat^\cocart_{/X} \xrightarrow{\tau_0} \mS_{/\tau_0(X)}. $
Hence there is a canonical equivalence 
$$ \tau_0(Y \,{\vec{\times}}_X X) \simeq \tau_0(Y)$$ over $\tau_0(X).$
\end{proof}

\begin{theorem}\label{QuillenB}
Let $\phi: Y \to X$ be a functor such that the fiber transports of the enveloping cocartesian fibration of $\phi$ are weak equivalences.  
In other words, for every morphism $s \to t $ in $X$
the induced functor $$ Y \times_X X_{//^\oplax s} \to Y \times_X X_{//^\oplax t} $$
is a weak equivalence.

Then, for every $t \in X $, the functor $\tau_0: \infty\fcat \to \mS$
preserves the oriented right fiber of $\phi: Y \to X $ over $t$.

\end{theorem}

\begin{proof}

By \cref{cocafi} applied to the enveloping cocartesian fibration $ Y \,{\vec{\times}}_X X \to X $, the induced sequence
$$\tau_0(Y \,{\vec{\times}}_X \{ t\}) \to \tau_0(Y \,{\vec{\times}}_X X) \to \tau_0(X) $$
is a fiber sequence.
By \cref{envlemo} the second functor in the latter sequence identifies with the functor
$$ \tau_0(Y) \to \tau_0(X).$$
\end{proof}

\subsection{Oriented realization}

\begin{notation}
Let $\Min, \Max \subset \Fun(\bD^1, \Delta)$ be the full subcategories of order preserving maps preserving the minimum or maximum, respectively. 
    
\end{notation}

\begin{definition}

Let $\mC$ be an oriented category and $X $ a simplical object in $\mC$, i.e. a functor $\Delta^\op \to \mC.$
The oriented realization of $X$ is the $\Max$-oriented colimit.
    
\end{definition}

\begin{theorem}\label{realchar} Let $\mC$ be an oriented category that $\Delta$-indexed coends and right tensors with objects of $\Delta$.
Let $X$ be a simplical object in $\mC$.
There is a canonical equivalence
$$ \overset{\to}{| X|} \simeq \int_{[n] \in \Delta} X_n \ot [n]$$

\end{theorem}

\begin{proof}

For every $[n] \in \Delta$ let $ \overline{\Max}$ be the set of morphisms of $\Delta_{/ [n]}$ lying over morphisms of $\Delta$ that belong to $\Max.$
By \cref{weichara} the $\Max$-oriented colimit is the colimit weighted with respect to the weight $$ \Delta \to \infty\Cat, [n] \mapsto \Delta_{//^\oplax [n]}[(\overline{\Max})^{-1}]= \Delta_{/ [n]} [(\overline{\Max})^{-1}].$$

Every morphism of $\Delta$ uniquely factors as a map in $\Delta$ preserving the maximum followed by an inert map in $\Delta$ preserving the minimum, i.e. a map of the form $[m] \simeq \{0 < ... < m \} \subset [n]$. This gives a factorization system whose left class is the class $\Max$ and whose right class is the class $\Min'$ of inert maps in $\Delta$ preserving the minimum.
This implies by \cite[Theorem 2.2.12.]{gepner2026homotopy} that the embedding $\Min' \subset \Fun(\bD^1,\Delta)$
admits a left adjoint that sends a map $f: [m] \to [n]$ to the map
$h: [m'] \to [n]$ appearing in the unique factorization $ [m] \xrightarrow{g} [m'] \xrightarrow{h} [n]$ of $f$ into a map $ [m] \to [m']$ in $\Max$ and a map $[m'] \to [n]$ in $\Min'.$ The unit of this localization is the commutative square given by $(g, \id_{[n]}).$
In particular, the local equivalences of this localization are inverted by evaluation at the target $ \Fun(\bD^1,\Delta) \to \Delta.$
So the localization $\Fun(\bD^1,\Delta) \rightleftarrows \Min'$ is relative to $\Delta$ via evaluation at the target. Hence it induces a localization $L: \Delta_{/ [n]} \rightleftarrows \Min'_{[n]} $ on the fiber over $ [n]$ whose local equivalences are maps in $\overline{\Max}.$
The localization functor sends $\overline{\Max} $ to $\overline{\Max}$
since for every maps $f: [n] \to [m],g: [m] \to [k]$ in $\Delta$ the morphism $g$ preserves the maximum if $f $ and $g f$ preserve the maximum.
Hence the latter localization induces a localization
$$ L: \Delta_{/ [n]}[(\overline{\Max})^{-1}] \rightleftarrows \Min'_{[n]}[(\overline{\Max})^{-1}]$$ whose unit is an equivalence since it belongs to
$\overline{\Max} $. So the latter localization is an equivalence.

Since the localization $\Fun(\bD^1,\Delta) \rightleftarrows \Min'$ is relative to $\Delta$ via evaluation at the target, \cite[Lemma 2.2.4.11., Remark 2.2.4.12.]{lurie.higheralgebra} implies that evaluation at the target $\Min' \to \Delta$ is a cocartesian fibration and the localization functor $ \Fun(\bD^1,\Delta) \to \Min'$ is a map of cocartesian fibrations
over $\Delta.$
Therefore also the fiber transports of evaluation at the target $\Min' \to \Delta$ send $\overline{\Max} $ to $\overline{\Max}$.
Therefore the equivalence $$ L: \Delta_{/ [n]}[(\overline{\Max})^{-1}] \simeq \Min'_{[n]}[(\overline{\Max})^{-1}]$$
is natural in $[n] \in \Delta.$

A map in $\Min'_{[n]}$ never preserves the maximum unless it is an identity because objects in $\Min'_{[n]}$ are of the form 
$[m] \simeq \{0 < ... < m \} \subset [n]$. 
Thus the canonical functor $\Min'_{[n]} \to \Min'_{[n]}[(\overline{\Max})^{-1}]$, which is natural in $[n] \in \Delta$, 
is an equivalence.

The Yoneda embedding $$[n] \to \infty\Cat_{/ [n]}, i \mapsto [n]_{/i} \simeq \{0< ... < i \} \subset [n]$$ induces an equivalence
$[n] \simeq \Min'_{[n]}.$
This equivalence is natural in $[n] \in \Delta$ because
for every map $\phi: [n] \to [m]$ in $\Delta$
the restriction $\{0< ... < i \} \subset [n] \xrightarrow{\phi} [m]$
uniquely factors as $$\{0< ... < i \} \to \{ 0 < ... < \phi(i)\} \subset [m],$$ where the first map preserves the maximum and the second map is intert and preserves the minimum. Thus the induced functor
$\Min'_{[n]} \to \Min'_{[m]}$ sends $\{0< ... < i \} \subset [n]$
to $\{ 0 < ... < \phi(i)\} \subset [m].$

We obtain a canonical equivalence
$$ [n] \simeq \Min'_{[n]} \to \Min'_{[n]}[(\overline{\Max})^{-1}] \simeq \Delta_{/ [n]}[(\overline{\Max})^{-1}],$$ natural in $[n] \in \Delta.$

Hence the $\Max$-oriented colimit is the colimit weighted with respect to the embedding $ \bj: \Delta \subset \infty\Cat.$

For every $Y \in \mC$ by the end formula for mapping spaces in ordinary functor categories \cite[Proposition 5.1.]{articles} there is a canonical equivalence
$$ \Map_{\Fun(\Delta, \infty\Cat)}(\bj, \Mor_{\mC}(-,Y) \circ X) \simeq$$$$ \int^{[n]\in \Delta} \Map_{\infty\Cat}([n], \Mor_{\mC}(X_n,Y)) $$$$ \simeq \int^{[n]\in \Delta} \Map_{\mC}(X_n \ot [n],Y) \simeq $$$$\Map_{\mC}(\int_{[n]\in \Delta} X_n \ot [n],Y). $$
\end{proof}

\begin{corollary}\label{realisorire} Let $\mC$ be an oriented category that admits $\Delta$-indexed coends and right tensors with objects of $\Delta$ and right cotensors with objects of $\Delta^\op$ and finite limits.

There is an adjunction 
\begin{equation}\label{adjori}
L: \Fun(\Delta^\op, \mC) \rightleftarrows \Fun(\bD^1, \mC): R, \end{equation}
where the left adjoint sends a simplicial object 
$X \in \Fun(\Delta^\op, \mC)$ to $ X_0 \to \overset{\to}{| X|} $
and the right sends a morphism $A \to B$ in $\mC$ to the simplicial object
$$ [n] \mapsto A^{\{0,...,n \}} \times_{B^{\{0,...,n \}}} B^{[n]}.$$
    
\end{corollary}

\begin{proof}

By the end formula for mapping spaces in functor categories, there is an adjunction
$$ \Fun(\Delta^\op, \mC) \rightleftarrows \mC, $$
where the left adjoint sends a simplicial object 
$X \in \Fun(\Delta^\op, \mC)$ to $ \int_{[n] \in \Delta} X_n \ot [n] $
and the right sends an object $ B \in \mC $ to the simplicial object
$$B^{[\bullet]}: [n] \mapsto B^{[n]}.$$

By \cref{realchar} the left adjoint assigns the oriented realization.

The functor $\Fun(\Delta^\op, \mC)  \to \mC$ evaluating at $[0]$
admits a right adjoint sending $Z $ to the simplicial object $$[n] \mapsto Z^{\{0,...,n \}} \simeq Z^{\times n+1}.$$

Let $X \in \Fun(\Delta^\op, \mC)$ and $f: A \to B$ a functor.
Hence there is a canonical equivalence
$$ \RMor_{\Fun(\bD^1, \mC)}(L(X), f) \simeq $$$$ \RMor_{\mC}(\overset{\to}{| X|}, B) \times_{\RMor_{\mC}(X_0, B)} \RMor_{\mC}(X_0, A) \simeq $$
$$\RMor_{\Fun(\Delta^\op,\mC)}(X, B^{[\bullet]}) \times_{\RMor_{\mC}(X_0, B)} \RMor_{\mC}(X_0, A) \simeq $$$$ \RMor_{\Fun(\Delta^\op,\mC)}(X, R(f)). $$
\end{proof}

\begin{remark}\label{equivori}

Let $$ \Fun^\surj(\bD^1,  \infty\fcat)\subset \Fun(\bD^1,  \infty\fcat)$$
be the full subcategory of essentially surjective functors.

For $\mC= \infty\fcat$ the adjunction 
$$
L: \Fun(\Delta^\op,  \infty\fcat) \rightleftarrows \Fun(\bD^1,  \infty\fcat): R $$ of
(\ref{adjori})
restricts to an adjunction 
\begin{equation}\label{adjori2}
\Fun(\Delta^\op,  \infty\fcat) \rightleftarrows \Fun^\surj(\bD^1,  \infty\fcat) \end{equation}
because for every $X \in \Fun(\Delta^\op,  \infty\fcat)$ the canonical functor $X_0 \to \int_{[n] \in \Delta} X_n \ot [n]$
is essentially surjective.

For $\mC= \infty\fcat$, Loubaton \cite{loubaton2025effectivity} proves that the right adjoint in the adjunction (\ref{adjori2}) is fully faithful.
    
\end{remark}

\begin{notation}

Let $\mV$ be a monoidal category, $A $ an associative algebra in $\mV$ and $M$ a right $A$-module in $\mV$ and $N$ a left $A$-module in $\mV.$ 
We write $\B_\bullet(M,A,N)$ for the Bar-construction of $M, N,$
which is a simplicial object in $\mV$ sending
$[n]$ to $M \ot A^{\ot n} \ot N.$

\end{notation}

For the following corollary $\mV$ is the cartesian monoidal structure on
$\infty\Cat.$

\begin{corollary}

Let $\mA$ be a monoidal $\infty$-category.
There is a canonical equivalence of $\infty$-categories:
$$|\overset{\to}{\B_\bullet(*,\mA,*)}| \simeq \B\mA.$$

\end{corollary}

\begin{proof}

The right adjoint of adjunction (\ref{adjori}) sends
the essentially surjective functor $* \to \B\mA$ to the Bar-construction 
$\B_\bullet(*,\mA,*): [n] \mapsto A^{\times n}$ for the monoid $\mA$ in the cartesian monoidal structure on $\infty\Cat.$

Hence by \cref{realisorire} and \cref{equivori} there is a canonical equivalence in $\Fun(\bD^1, \infty\fcat): $
$$ (* \to |\overset{\to}{\B_\bullet(*,\mA,*)}|) \simeq L(R(* \to \B\mA)) \simeq (* \to \B\mA).$$
\end{proof}

\section{\mbox{Classification of fibrations}}

\subsection{The Grothendieck construction for enriched correspondences}

\begin{definition}
Let $\mC$ be a category that admits an initial object.    
The initial object of $\mC$ is empty if an object $X$ of $\mC$ is initial if and only if there is a morphism $X \to \emptyset$ in $\mC.$

\end{definition}

\begin{remark}
The initial object of $\mC$ is empty if for every non-initial object $X$ of $\mC$ the functor
$\Map_\mC(X,-): \mC \to \mS$ corepresented by $X$ preserves the initial object, or if the presheaf $ \Map_\mC(-,\emptyset): \mC^\op \setminus \{\emptyset\} \to \mS$ is initial.
    
\end{remark}

\begin{definition}
A monoidal category is semicartesian if the tensor unit is a final object.
    
\end{definition}

\begin{definition}
Let $\mV$ be a semicartesian presentably monoidal category such that the initial object is empty.
A $\mV$-enriched correspondence is a $\mV$-enriched
functor $\phi: \mM \to \bD^1.$
    
\end{definition}

\begin{remark}
Let $\mV$ be a semicartesian presentably monoidal category such that the initial object is empty.
A $\mV$-enriched correspondence $\phi: \mM \to \bD^1$ is a $\mV$-enriched
cocartesian fibration if for every $X \in \mM_0 $ there is an object $X' \in \mM_1 $
and a morphism $X \to X'$ in $\mM$ that induces for every $Y \in \mM_1$ an equivalence
$ \Mor_\mM(X',Y) \to \Mor_\mM(X,Y).$
Therefore a $\mV$-enriched correspondence $\mM \to \bD^1$ 
is a $\mV$-enriched cocartesian fibration
if and only if the embedding $\mM_1 \subset \mM$ admits a $\mV$-enriched left adjoint.
    
\end{remark}

\begin{theorem}\label{Groprof}
	
Let $\mV$ be a semicartesian presentably monoidal category such that the initial object is empty.
The functor $ {\mV\mathrm{-}\Cat}_{/\bD^1} \to {\mV\mathrm{-}\Cat} \times {\mV\mathrm{-}\Cat}$
is a cartesian fibration and classifies the functor
$${\mV\mathrm{-}\Cat}^\op \times {\mV\mathrm{-}\Cat}^\op \to \Cat, (\mC,\mD) \mapsto {\mV\mathrm{-}\Fun}(\mD, \mP_\mV(\mC)).$$
Moreover the canonical equivalence
$$ \{\mC,\mD\} \times_{{\mV\mathrm{-}\Cat} \times  {\mV\mathrm{-}\Cat}} {\mV\mathrm{-}\Cat}_{/\bD^1} \simeq {\mV\mathrm{-}\Fun}(\mD, \mP_\mV(\mC))$$
sends a $\mV$-enriched correspondence $\mM\to \bD^1$
to the $\mV$-enriched functor $$\mD \subset \mM \subset \mP_\mV(\mM) \to \mP_\mV(\mC).$$
The latter equivalence restricts to an equivalence
$$ \{\mC,\mD\} \times_{{\mV\mathrm{-}\Cat} \times  {\mV\mathrm{-}\Cat}} {\mV\mathrm{-}\Cat}^\cart_{/\bD^1} \simeq {\mV\mathrm{-}\Fun}(\mD, \mC), $$
where the left hand side is the full subcategory of cartesian fibrations.
\end{theorem}

\begin{proof}

The canonical commutative square
$$\begin{xy}
\xymatrix{
{\mV\mathrm{-}\Cat}_{/\bD^1} \ar[d]^{} \ar[r]
& {\mV\mathrm{-}\Cat} \ar[d]
\\ 
\mS_{/\{0,1\}} \ar[r]& \mS,}
\end{xy}$$
where the horizontal functors are the forgetful functors and the verical functors assign the underlying space,
induces an embedding to the pullback. The essential image
precisely consists of the $\mV$-enriched categories $\mC$ with space of objects $X$ equipped with a map $X \to \{0,1\}$
such that the morphism object from an object lying over 1 to an object lying over 0 is the initial object.
This holds since the map to the pullback induces on the fiber over any $\mC \in {\mV\mathrm{-}\Cat}$ the induced map
$$\Map_{{\mV\mathrm{-}\Cat}}(\mC, \bD^1) \to \Map_\mS(X,\{0,1\})$$ whose fiber over any $\varphi\in  \Map_\mS(X,\{0,1\})$ is the space
$\Map_{{\mV\mathrm{-}\Cat}_X}(\mC, \varphi^*(\bD^1)). $
The latter space is contractible if $\mC$ is as above.

Let 
$$ {\mV\mathrm{-}\Cat}'_{X \coprod Y} \subset {\mV\mathrm{-}\Cat}_{X \coprod Y}$$
be the full subcategory of $\mV$-enriched categories 
$\mC$ with space of objects $X \coprod Y$ such that
the morphism objects from an object in $Y$ to an object in $X$ is the initial object.
Then for every small spaces $X,Y$ there is a diagram of pullback squares
$$\begin{xy}
\xymatrix{
{\mV\mathrm{-}\Cat}'_{X \coprod Y} \ar[d]^{} \ar[r]
& {\mV\mathrm{-}\Cat}_{/\bD^1} \ar[d]
\\ 
{\mV\mathrm{-}\Cat}_{X} \times {\mV\mathrm{-}\Cat}_{Y} \ar[r] \ar[d] &  {\mV\mathrm{-}\Cat} \times {\mV\mathrm{-}\Cat}\ar[d]
\\ \{(X,Y)\} \ar[r]& \mS \times \mS \simeq \mS_{/\{0,1\}}
,}
\end{xy}$$
where the upper right vertical functor takes the fibers over 0,1, the left upper vertical functor restricts along the embeddings $X \subset X \coprod Y, Y \subset X \coprod Y$.
Hence the fiber of the functor $ {\mV\mathrm{-}\Cat}_{/\bD^1} \to {\mV\mathrm{-}\Cat} \times {\mV\mathrm{-}\Cat}$ over any $(\mC,\mD) \in {\mV\mathrm{-}\Cat}_X \times {\mV\mathrm{-}\Cat}_Y$
is the fiber over $(\mC,\mD)$ of the functor $$ {\mV\mathrm{-}\Cat}'_{X \coprod Y} \to {\mV\mathrm{-}\Cat}_{X} \times {\mV\mathrm{-}\Cat}_{Y}. $$

Let $\BM:= (\Delta_{/[1]})^\op$.
There are two functors $\BM \to \Delta^\op$ taking the fiber over 0,1, respectively, and two embeddings $\Delta^\op \to \BM$ sending $[n]$ to the constant functor $[n] \to [1]$ with value 0,1, respectively, which we call left and right embedding.
Let $$\theta : \BM_{X,Y} \to \BM \times_{\Delta^\op \times \Delta^\op} \Delta^\op_{X \coprod Y}$$ be the map of left fibrations over $\BM$ 
classifying the map of presheaves on $\Delta_{/[1]}$
from the presheaf $$\varphi: [n] \to [1] \mapsto X^{\varphi^{-1}(\{0\})} \times Y^{\varphi^{-1}(\{1\})}$$ to the presheaf 
$\varphi \mapsto (X\coprod Y)^n$, whose component at any
$\varphi: [n] \to [1]$ is the canonical map
$$ X^{\varphi^{-1}(\{0\})} \times Y^{\varphi^{-1}(\{1\})} \to (X \coprod Y)^{\varphi^{-1}(\{0\}) \coprod \varphi^{-1}(\{1\})}. $$

The left and right embeddings $\Delta^\op \subset \BM$
induce equivalences $$ \Delta^\op_X \simeq \Delta^\op \times_\BM \BM_{X,Y}, \Delta^\op_Y \simeq \Delta^\op \times_\BM \BM_{X,Y}$$ over $\Delta^\op.$
The restriction $$ \Delta^\op_X \subset \BM_{X,Y} \to \Delta^\op_{X \coprod Y} $$
is the embedding induced by the embedding $X \subset X \coprod Y,$
and similar for $Y.$
Let $\theta': \BM_{X,Y} \to \Delta^\op_{X \coprod Y}$ be $\theta$ followed by the projection, which is a map of left fibrations over $\Delta^\op$
since $\BM \to \Delta^\op$ is a left fibration.

By \cite[Proposition 10.20.]{HEINE2023108941new} for every double category $\mO \to \Delta^\op$ 
the category $\Fun(\mO_{[1]}, \mV)$ carries a canonical presentably monoidal structure, denoted by $\mV^\mO$,
such that
$$\Alg_{\Delta^\op}(\mV^\mO) \simeq \Alg_{\mO/\Delta^\op}(\mV).$$
Moreover every map of double categories $\mO \to \mO'$ 
the induced functor
$\Fun(\mO'_{[1]}, \mV) \to \Fun(\mO_{[1]}, \mV)$
refines to a lax monoidal functor $\mV^{\mO'} \to \mV^\mO,$
which admits a monoidal left adjoint.
Hence using the map $$\theta': \BM_{X,Y} \to \Delta^\op_{X \coprod Y}$$ of double categories the induced functor
$$\theta'^*: \Fun((X \times X) \coprod (X \times Y) \coprod (Y \times X) \coprod (Y \times Y),\mV) \to \Fun((X \times X) \coprod (X \times Y) \coprod (Y \times Y),\mV)$$
is lax monoidal and admits a monoidal left adjoint $\theta'_!.$
Since $\theta'^* $ factors as 
$$ \Fun(X \times X,\mV) \times \Fun(X \times Y,\mV) \times \Fun(Y \times X,\mV) \times \Fun(Y \times Y,\mV) \to \Fun(X \times X,\mV) \times \Fun(X \times Y,\mV) \times
\Fun(Y \times Y,\mV), $$
the left adjoint $\theta'_!$ is induced by the embedding 
$\{\emptyset\}\subset \Fun(Y \times X, \mV)$ and so is a monoidal embedding.
The monoidal adjunction
$$ \theta'_! : \Fun((X \times X) \coprod (X \times Y) \coprod (Y \times Y),\mV) \rightleftarrows \Fun((X \times X) \coprod (X \times Y) \coprod (Y \times X) \coprod (Y \times Y),\mV): \theta'^*$$
over 
$\Fun(X \times X,\mV) \times \Fun(Y \times Y,\mV)$ induces on associative algebras an adjunction
$$ \Alg_{\BM_{X,Y}/\Delta^\op}(\mV) \rightleftarrows \Alg_{\Delta^\op_{X \coprod Y}/\Delta^\op}(\mV) \simeq {\mV\mathrm{-}\Cat}_{X \coprod Y}$$ 
over $\Alg_{\Delta^\op_{X}/\Delta^\op}(\mV) \times \Alg_{\Delta^\op_{Y}/\Delta^\op}(\mV)$ whose left adjoint is fully faithful.
The essential image of the left adjoint precisely consists of the $\Delta^\op_{X \coprod Y}$-algebras in $\mV$ whose underlying functor 
$$(X \times X) \coprod (X \times Y) \coprod (Y \times X) \coprod (Y \times Y) \to \mV$$ lands in the essential image of $\kappa,$
i.e. restricts on $Y \times X$ to the constant functor with value $\emptyset.$
Hence the left adjoint induces an equivalence
$$ \Alg_{\BM_{X,Y}/\Delta^\op}(\mV) \to {\mV\mathrm{-}\Cat}'_{X \coprod Y} $$
over $\Alg_{\Delta^\op_{X}/\Delta^\op}(\mV) \times \Alg_{\Delta^\op_{Y}/\Delta^\op}(\mV) \simeq {\mV\mathrm{-}\Cat}_X \times {\mV\mathrm{-}\Cat}_Y $.
Consequently, the fiber of the functor $$ {\mV\mathrm{-}\Cat}'_{X \coprod Y} \to {\mV\mathrm{-}\Cat}_{X} \times {\mV\mathrm{-}\Cat}_{Y} $$ over $(\mC,\mD)$ is the fiber of the functor
$$ \Alg_{\BM_{X,Y}/\Delta^\op}(\mV) \to {\mV\mathrm{-}\Cat}_{X} \times {\mV\mathrm{-}\Cat}_{Y} $$ over $(\mC,\mD).$
By \cite[Proposition 5.11.]{HEINE2023108941new} there is a canonical natural equivalence
$$ \{(\mC,\mD)\} \times_{{\mV\mathrm{-}\Cat}_{X} \times {\mV\mathrm{-}\Cat}_{Y}} \Alg_{\BM_{X,Y}/\Delta^\op}(\mV) \simeq {\mV\mathrm{-}\Fun}(\mD, \mP_\mV(\mC)).$$

We prove next that the functor 
$$ {\mV\mathrm{-}\Cat}_{/\bD^1} \to {\mV\mathrm{-}\Cat} \times {\mV\mathrm{-}\Cat}$$
is a cartesian fibration.
The latter fits into a commutative triangle 
$$\begin{xy}\label{pulbast}
\xymatrix{
{\mV\mathrm{-}\Cat}_{/\bD^1} \ar[d]^{} \ar[r]
& {\mV\mathrm{-}\Cat} \times {\mV\mathrm{-}\Cat} \ar[d]
\\ 
\mS_{/\{0,1\}} \ar[r]^\simeq &  \mS \times \mS,}
\end{xy}$$
where both vertical functors are cartesian fibrations induced by the cartesian fibration ${\mV\mathrm{-}\Cat} \to \mS$. Moreover the top horizontal functor preserves cartesian morphisms.
Hence the top horizontal functor is a cartesian fibration if it induces fiberwise cartesian fibrations whose cartesian morphisms are preserved by the fiber transports.

By the existence of the pullback square (\ref{pulbast}), the top horizontal functor induces on the fiber over $(X,Y) \in \mS \times \mS$
the functor 
$$ {\mV\mathrm{-}\Cat}'_{X \coprod Y} \to {\mV\mathrm{-}\Cat}_{X} \times {\mV\mathrm{-}\Cat}_{Y}.$$
As we have shown, the latter functor is equivalent to 
$ \gamma: \Alg_{\BM_{X,Y}/\Delta^\op}(\mV) \to \Alg_{\Delta^\op_{X}/\Delta^\op}(\mV) \times \Alg_{\Delta^\op_{Y}/\Delta^\op}(\mV).$ The equivalences $$ \Delta^\op_X \simeq \Delta^\op \times_\BM \BM_{X,Y}, \Delta^\op_Y \simeq \Delta^\op \times_\BM \BM_{X,Y}$$ over $\Delta^\op$ give rise to a lax monoidal functor
$$ \mV^{\BM_{X,Y}} \to \mV^{\Delta^\op_X} \times \mV^{\Delta^\op_Y}.$$
The latter induces a functor 
$$ \Alg_\BM(\mV^{\BM_{X,Y}}) \to \Alg_{\Delta^\op}(\mV^{\Delta^\op_X}) \times \Alg_{\Delta^\op}(\mV^{\Delta^\op_Y}), $$
which is equivalent to $\gamma$ and which is a cartesian fibration. Moreover the $\gamma$-cartesian morphisms are detected by the forgetful functor
$$  \Alg_{\BM_{X,Y}/\Delta^\op}(\mV) \to \Fun(X \times Y,\mV).$$
This characterization implies that the $\gamma$-cartesian morphisms are preserved by the fiber transports. 
\end{proof}

\begin{corollary} Let $\mV$ be a semicartesian presentably monoidal category such that the initial object is empty.
A $\mV$-enriched correspondence $\mM \to \bD^1$ is a $\mV$-enriched cocartesian and cartesian fibration if and only if it classifies a $\mV$-enriched adjunction $\mM_0 \rightleftarrows \mM_1.$
    
\end{corollary}

\begin{proof}

By \cref{Groprof} a $\mV$-enriched correspondence $\mM \to \bD^1$ classifies a $\mV$-enriched functor $$\mM_1 \to {\Fun_\mV}(\mM^\circ_0,\mV)$$ corresponding to a $\mV^\rev$-enriched functor $\mM_0^\circ \to {\mV\mathrm{-}\Fun}(\mM_1,\mV)$. 
The correspondence classifies a $\mV$-enriched cocartesian fibration
$\mM \to \bD^1$ if and only if the $\mV$-enriched functor $\mM_1 \to {\Fun_\mV}(\mM^\circ_0,\mV)$ factors as $G: \mM_1 \to \mM_0$.
The correspondence classifies a $\mV$-enriched cartesian fibration
$\mM \to \bD^1$ if and only if the $\mV$-enriched functor $\mM_0^\circ \to {\mV\mathrm{-}\Fun}(\mM_1,\mV)$ factors as $F^\circ: \mM_0^\circ \to \mM_1^\circ $.

By \cite[Remark 2.72.]{heine2024bienriched} a $\mV$-enriched functor $G: \mM_1 \to \mM_0$ admits a $\mV$-enriched left adjoint if and only if the $\mV$-enriched functor $\mM_0^\circ \to {\mV\mathrm{-}\Fun}(\mM_1,\mV)$ factors as $F^\circ: \mM_0^\circ \to \mM_1^\circ $, and dually
a $\mV$-enriched functor $F: \mM_0 \to \mM_1$ admits a $\mV$-enriched right adjoint if and only if the $\mV$-enriched functor $\mM_1 \to {\Fun_\mV}(\mM^\circ_0,\mV)$ factors as $G: \mM_1 \to \mM_0$.
This proves the result.
\end{proof}

\begin{corollary}\label{Grsuspo} Let $\mV$ be a semicartesian presentably monoidal category such that the initial object is empty
and $A \in \mV.$
The functor $ {\mV\mathrm{-}\Cat}_{/S(A)} \to {\mV\mathrm{-}\Cat} \times {\mV\mathrm{-}\Cat}$
taking the fibers over 0,1 is a cartesian fibration and classifies the functor
$${\mV\mathrm{-}\Cat}^\op \times {\mV\mathrm{-}\Cat}^\op \to \Cat, (\mC,\mD) \mapsto {\mV\mathrm{-}\Fun}(\mD, {\Fun_\mV}(\mC^\circ,\mV_{/A})).$$

\end{corollary}

\begin{proof}

By \cref{Groprof} the functor $ {\mV\mathrm{-}\Cat}_{/\bD^1} \to {\mV\mathrm{-}\Cat} \times {\mV\mathrm{-}\Cat}$ is a cartesian fibration.
Hence passing to slice categories the functor
$$ {\mV\mathrm{-}\Cat}_{/S(A)} \simeq ({\mV\mathrm{-}\Cat}_{/\bD^1})_{/S(A)} \to {\mV\mathrm{-}\Cat_{/*}} \times {\mV\mathrm{-}\Cat}_{/*} \simeq {\mV\mathrm{-}\Cat} \times {\mV\mathrm{-}\Cat}$$ is a cartesian fibration.
The latter is the functor $$ {\mV\mathrm{-}\Cat}_{/S(A)} \to {\mV\mathrm{-}\Cat}_{/\bD^1} \to {\mV\mathrm{-}\Cat} \times {\mV\mathrm{-}\Cat},$$ which identifies with the functor taking the fibers over 0,1.

By \cref{Groprof} there is a canonical equivalence
$$ \{\mC,\mD\} \times_{{\mV\mathrm{-}\Cat} \times  {\mV\mathrm{-}\Cat}} {\mV\mathrm{-}\Cat}_{/\bD^1} \simeq {\mV\mathrm{-}\Fun}(\mD, \mP_\mV(\mC)).$$
Let $S(A)' \to S(\mA)$ be a cartesian lift lying over the $\mV$-enriched functors
$\mC \to *, \mD \to *.$
The latter induces an equivalence
$$ {\mV\mathrm{-}\Fun}(\mD, {\Fun_\mV}(\mC^\circ,\mV_{/A})) \simeq $$$$ {\mV\mathrm{-}\Fun}(\mD, \mP_\mV(\mC))_{/\underline{A}} \simeq $$
$$ (\{(\mC,\mD)\} \times_{{\mV\mathrm{-}\Cat} \times {\mV\mathrm{-}\Cat}} {\mV\mathrm{-}\Cat}_{/\bD^1})_{/S(A)'} \simeq $$$$ \{(\mC,\mD)\} \times_{{\mV\mathrm{-}\Cat}_{/\mC} \times {\mV\mathrm{-}\Cat}_{/\mD} } ({\mV\mathrm{-}\Cat}_{/\bD^1})_{/S(A)'} \to $$
$$ \{(\mC,\mD)\} \times_{{\mV\mathrm{-}\Cat} \times {\mV\mathrm{-}\Cat}} ({\mV\mathrm{-}\Cat}_{/\bD^1})_{/S(A)} \simeq $$$$\{(\mC,\mD)\} \times_{{\mV\mathrm{-}\Cat} \times {\mV\mathrm{-}\Cat}} {\mV\mathrm{-}\Cat}_{/S(A)}.$$
Here $\underline{A} \in {\mV\mathrm{-}\Fun}(\mD, {\Cat_\mV}(\mC^\circ,\mV)) $ is the image of $A$ under the diagonal functor
$ \mV \to {\mV\mathrm{-}\Fun}(\mD, {\Cat_\mV}(\mC^\circ,\mV)). $
\end{proof}

\begin{corollary}\label{corrbifib}
Let $\mV$ be a semicartesian presentably monoidal category such that the initial object is empty
and $A \in \mV.$
The functor ${\mV\mathrm{-}\Cat_{/S(A)}} \to {\mV\mathrm{-}\Cat} \times {\mV\mathrm{-}\Cat} $ taking the fibers over 0,1 is a bifibration.
    
\end{corollary}

\begin{proof}

By \cref{Grsuspo} the functor ${\mV\mathrm{-}\Cat_{/S(A)}} \to {\mV\mathrm{-}\Cat} \times {\mV\mathrm{-}\Cat} $ taking the fibers over 0,1 is a cartesian fibration.
Thus by \cite[Corollary 5.2.2.4.]{lurie.HTT} and \cite[Lemma 2.42.]{heine2023monadicity} it is a bifibration if and only if for every $\mV$-enriched category $\mD$ and $\mV$-enriched functor $\phi: \mB \to \mC$ corresponding to a functor $\theta: \bD^1 \to {\mV\mathrm{-}\Cat}$
the functor $\{\mD\} \times_{{\mV\mathrm{-}\Cat}} {\mV\mathrm{-}\Cat_{/S(A)}} \times_{{\mV\mathrm{-}\Cat}} \bD^1 \to \bD^1 $ is a cocartesian fibration. 
By \cref{Grsuspo} the functor $\theta$ is a cartesian fibration that classifies the functor
$$\phi^* : {\mV\mathrm{-}\Fun}(\mD, {\Fun_\mV}(\mC^\circ,\mV_{/A})) \to {\mV\mathrm{-}\Fun}(\mD, {\Fun_\mV}(\mB^\circ,\mV_{/A})). $$
Hence the functor $\theta$ is also a cocartesian fibration if and only 
if the functor $\phi^*$ admits a left adjoint.
By \cite[Proposition 5.20.]{heine2024bienriched} the $\mV$-enriched functor 
$ {\Fun_\mV}(\mC^\circ,\mV_{/A}) \to {\Fun_\mV}(\mB^\circ,\mV_{/A})$
admits a $\mV$-enriched left adjoint. This implies that the functor
$\phi^*$ admits a left adjoint.
\end{proof}

\subsection{Straightening over a suspension}

\begin{proposition}\label{glue}
Let $1 \leq \m \leq \infty.$
We assume that the $m-1$-categorical Grothendieck construction is an equivalence.
Let $\n \geq 2$ and $\mA_1, ..., \mA_n$ be $m-1$-categories and
$\mM_1 \to S(\mA_1), ..., \mM_n \to S(\mA_n)$ cocartesian fibrations
between $m$-categories equipped with equivalences
$$\mM_\bi \cap \mM_{\bi+1}:= \{\bi\}\times_{S(\mA_\bi)} \mM_\bi \simeq \{\bi\}\times_{S(\mA_{\bi+1})} \mM_{\bi+1}, $$
where we denote the set of objects of $S(\mA_1) \vee \ldots \vee S(\mA_n)$
by $\{0,\ldots,n\}.$
Let $\mN$ be a $m$-category.

\begin{enumerate}[\normalfont(1)]\setlength{\itemsep}{-2pt}
\item The induced functor
\[
\mM:= \mM_1 \coprod_{\mM_1 \cap \mM_2}\cdots\coprod_{\mM_{n-1} \cap \mM_{n}} \mM_n \to S(\mA_1) \vee \ldots \vee S(\mA_n)
\]
is a cocartesian fibration.

\vspace{1mm}
\item A functor $\mM \to \mN$ over $ S(\mA_1) \vee \ldots \vee S(\mA_n) $
is a map of cocartesian fibrations over $ S(\mA_1) \vee \ldots \vee S(\mA_n) $ if and only if all restrictions to $\mM_i$ are maps of cocartesian fibrations over $ S(\mA_i) $ for every $i=1,\ldots, n$.

\vspace{1mm}
\item Let $\mM \to S(\mA_1\ldots\mA_n) $ be a cocartesian fibration.
The induced functor
\[
(S(\mA_1) \times_{S(\mA_1) \vee \ldots \vee S(\mA_n)} \mM) \coprod_{\{1\} \times_{S(\mA_1) \vee \ldots \vee S(\mA_n)} \mM} \cdots \coprod_{\{n-1\} \times_{S(\mA_1) \vee \ldots \vee S(\mA_n)} \mM} (S(\mA_n) \times_{S(\mA_1) \vee \ldots \vee S(\mA_n)} \mM) \to \mM
\]
is an equivalence.
\end{enumerate}
\end{proposition}

\begin{proof}
By \cref{cocartcolim} for every $m-1$-category $\mC$ the inclusion
$(m-1) \Cat^\cocart_{/\mC}\subset (m-1) \Cat_{/\mC}$
preserves small colimits under the assumption that the Grothendieck construction for $m-1$-categories holds.
Since the functors $\mM_i \to S(\mA_i)$ are cocartesian fibrations
for $i=1,...,n$ and cocartesian fibrations are closed under products,
this and \cref{homsus} imply that $\mM \to S(\mA_1) \vee \ldots \vee S(\mA_n)$ is enriched in the category of cartesian fibrations.

So to prove (1) and (2) it remains to see that for every $\bi=1,...,n$
the embedding $ \mM_i \subset \mM$ preserves cocartesian morphisms
since cocartesian morphisms are closed under composition.
By \cref{fullyfaith} for every $ \bi < \bj$ the canonical functor
$$\mM_\bi \coprod_{\mM_\bi \cap \mM_{\bi+1}} ... \coprod_{\mM_{\bj-1} \cap \mM_{\bj}} \mM_\bj \to \mM$$ is fully faithful.
Consequently, it suffices to see that 
for every $f \in \mA_{1}$ and $X \to f_!(X) $ a cocartesian lift of $f$ in $\mM_1$ and any $Z \in \mM_n$ the induced commutative square
$$\begin{xy}
\xymatrix{
\Mor_{\mM}(f_!(X),Z) \ar[d]^{} \ar[r]
& \Mor_{\mM}(X,Z) \ar[d] 
\\ 
\{f\} \times \mA_{2} \times ... \times \mA_n \ar[r] & \mA_{1} \times \mA_{2} \times ... \times \mA_n}
\end{xy}$$	
is a pullback square.
By the pasting law this is equivalent to say that the 
induced commutative square
$$\begin{xy}
\xymatrix{
\Mor_{\mM}(f_!(X),Z) \ar[d]^{} \ar[r]
& \Mor_{\mM}(X,Z) \ar[d] 
\\ 
\{f\} \ar[r] & \mA_{1}}
\end{xy}$$	
is a pullback square.
By \cref{homsus} the latter identifies with the following commutative square:
\begin{equation}\label{sqqo}
\begin{xy}
\xymatrix{
\Mor_{\mM}(f_!(X),Z) \ar[d]^{} \ar[r]
& \int^{Y \in \{1\}\times_{S(\mA_1) \vee \ldots \vee S(\mA_n)} \mM} \Mor_{\mM_1}(X,Y) \times \Mor_{\mM}(Y,Z) \ar[d] 
\\ 
\{f\} \ar[r] & \mA_{1}}
\end{xy}\end{equation}	
For every $m-1$-category $\mC$ and $X \in \mC$ the functor
$\Fun(\mC,(m-1)\Cat) \to (m-1)\Cat $
evaluating at $X$ factors as $$\Fun(\mC,(m-1)\Cat) \to (m-1)\Cat^\cocart_{/\mC} \to (m-1)\Cat,$$
where the last functor takes the fiber over $X$.
By assumption the Grothendieck construction for $m-1$-categories holds.
Since small colimits in $\Fun(\mC,(m-1)\Cat)$ are formed object-wise by \cref{colimenrfun}, the functor
$$(m-1)\Cat^\cocart_{/\mC} \to (m-1)\Cat$$ taking the fiber over $X$
preserves small colimits.

Therefore the induced functor to the fiber in the commutative square 
\ref{sqqo} identifies with the following functor, which is equivalent to the identity:
$$ \Mor_{\mM}(f_!(X),Z) \to \int^{Y \in \{1\}\times_{S(\mA_1) \vee \ldots \vee S(\mA_n)} \mM} \{f\} \times_{\mA_1}\Mor_{\mM_1}(X,Y) \times \Mor_{\mM}(Y,Z) $$
$$ \simeq \int^{Y \in \{1\}\times_{S(\mA_1) \vee \ldots \vee S(\mA_n)} \mM} \Mor_{\{1\}\times_{S(\mA_1) \vee \ldots \vee S(\mA_n)} \mM}(f_!(X),Y) \times \Mor_{\mM}(Y,Z) \simeq \Mor_{\mM}(f_!(X),Z).$$
The first equivalence holds since the morphism
$X \to f_!(X)$ is cocartesian over $f$ for the functor
$\mM_1 \to S(\mA_1).$
The second equivalence holds since the functor $$\Mor_{\{1\}\times_{S(\mA_1) \vee \ldots \vee S(\mA_n)} \mM}(f_!(X),-): \{1\}\times_{S(\mA_1) \vee \ldots \vee S(\mA_n)} \mM \to (m-1)\Cat $$ is corepresentable and so is sent by the $(m-1)\Cat$-enriched Yoneda extension
$$ \Fun(\{1\}\times_{S(\mA_1) \vee \ldots \vee S(\mA_n)} \mM, (m-1)\Cat) \to (m-1)\Cat , F \mapsto \int^{Y \in \{1\}\times_{S(\mA_1) \vee \ldots \vee S(\mA_n)} \mM} F(Y) \times \Mor_{\mM}(Y,Z) $$ of the 
functor $ \Mor_{\mM}(-,Z): (\{1\}\times_{S(\mA_1) \vee \ldots \vee S(\mA_n)} \mM)^\op \to (m-1)\Cat$ to $\Mor_{\mM}(f_!(X),Z).$ 

\vspace{1mm}
(3): By (2) the functor of (3) is a map of cocartesian fibrations over $S(\mA_1) \vee \ldots \vee S(\mA_n)$ since the functor $\mM \to S(\mA_1) \vee \ldots \vee S(\mA_n)$ is a cocartesian fibration.
The pullback of the functor of (3)
along the canonical embedding $S(\mA_\bi) \to S(\mA_1) \vee \ldots \vee S(\mA_n)$ is the identity for every $\bi=1 ,..., n.$
Consequently, the functor of (3) induces on the fiber over every object of $S(\mA_1) \vee \ldots \vee S(\mA_n)$ an equivalence.
It is therefore an equivalence by \cref{fiberwiseeq}.
\end{proof}

\begin{corollary}\label{Segal}

Let $1 \leq \m \leq \infty.$
We assume that the $m-1$-categorical Grothendieck construction is an equivalence.
Let $\n \geq 2$ and $\mA_1, ..., \mA_n$ be $m-1$-categories.
The following canonical functor is an equivalence:
$$ {m \Cat^\cocart_{/ S(\mA_1) \vee \ldots \vee S(\mA_n)}} \to {m \Cat^\cocart_{/ S(\mA_1)}} \times_{m\Cat} ... \times_{m\Cat} {m\Cat^\cocart_{/ S(\mA_n)}}.$$

\end{corollary}

\begin{proof}

By \cref{glue} (1) the canonical functor of the statement is essentially surjective.
We prove that it is fully faithful. For every pair of cocartesian fibrations
$\mC \to S(\mA_1) \vee \ldots \vee S(\mA_n)$ and $\mD \to S(\mA_1) \vee \ldots \vee S(\mA_n)$
the functor on morphism $\infty$-categories induced by the functor of the statement identifies with the functor
$$ \alpha: \Fun^\cocart_{S(\mA_1) \vee \ldots \vee S(\mA_n)}(\mC,\mD) \to $$$$ \Fun^\cocart_{S(\mA_1)}(\mS(\mA_1) \times_{S(\mA_1) \vee \ldots \vee S(\mA_n)}\mC,\mS(\mA_1) \times_{S(\mA_1) \vee \ldots \vee S(\mA_n)}\mD) 
\times_{\Fun^\cocart_{\{1\}}(\{1\} \times_{S(\mA_1) \vee \ldots \vee S(\mA_n)}\mC,\{1\} \times_{S(\mA_1) \vee \ldots \vee S(\mA_n)} \mD)} $$$$ \cdots \times_{\Fun^\cocart_{\{n-1\}}(\{n-1\} \times_{S(\mA_1) \vee \ldots \vee S(\mA_n)}\mC,\{n-1\} \times_{S(\mA_1) \vee \ldots \vee S(\mA_n)} \mD)} \Fun^\cocart_{S(\mA_n)}(\mS(\mA_n) \times_{S(\mA_1) \vee \ldots \vee S(\mA_n)}\mC,\mS(\mA_n) \times_{S(\mA_1) \vee \ldots \vee S(\mA_n)}\mD). $$
There is a canonical equivalence
$$ \Fun_{S(\mA_1)}(\mS(\mA_1) \times_{S(\mA_1) \vee \ldots \vee S(\mA_n)}\mC,\mS(\mA_1) \times_{S(\mA_1) \vee \ldots \vee S(\mA_n)}\mD) 
\times_{\Fun_{\{1\}}(\{1\} \times_{S(\mA_1) \vee \ldots \vee S(\mA_n)}\mC,\{1\} \times_{S(\mA_1) \vee \ldots \vee S(\mA_n)} \mD)} $$$$ \cdots \times_{\Fun_{\{n-1\}}(\{n-1\} \times_{S(\mA_1) \vee \ldots \vee S(\mA_n)}\mC,\{n-1\} \times_{S(\mA_1) \vee \ldots \vee S(\mA_n)} \mD)} \Fun_{S(\mA_n)}(\mS(\mA_n) \times_{S(\mA_1) \vee \ldots \vee S(\mA_n)}\mC,\mS(\mA_n) \times_{S(\mA_1) \vee \ldots \vee S(\mA_n)}\mD) $$

$$ \simeq \Fun_{S(\mA_1) \vee \ldots \vee S(\mA_n)}(\mS(\mA_1) \times_{S(\mA_1) \vee \ldots \vee S(\mA_n)}\mC,\mD) 
\times_{\Fun_{S(\mA_1) \vee \ldots \vee S(\mA_n)}(\{1\} \times_{S(\mA_1) \vee \ldots \vee S(\mA_n)}\mC, \mD)} $$$$ ... \times_{\Fun_{S(\mA_1) \vee \ldots \vee S(\mA_n)}(\{n-1\} \times_{S(\mA_1) \vee \ldots \vee S(\mA_n)}\mC, \mD)} \Fun_{S(\mA_1) \vee \ldots \vee S(\mA_n)}(\mS(\mA_n) \times_{S(\mA_1) \vee \ldots \vee S(\mA_n)}\mC,\mD) \simeq $$

$$ \Fun_{S(\mA_1) \vee \ldots \vee S(\mA_n)}(\mS(\mA_1) \times_{S(\mA_1) \vee \ldots \vee S(\mA_n)}\mC \coprod_{\{1\} \times_{S(\mA_1) \vee \ldots \vee S(\mA_n)}\mC} \cdots \coprod_{\{n-1\} \times_{S(\mA_1) \vee \ldots \vee S(\mA_n)}\mC} \mS(\mA_n) \times_{S(\mA_1) \vee \ldots \vee S(\mA_n)} \mC,\mD) $$
$$ \simeq \Fun_{S(\mA_1) \vee \ldots \vee S(\mA_n)}(\mC,\mD).$$
By \cref{glue} (2) the composed latter equivalence restricts to an equivalence
$$\beta: \Fun^\cocart_{S(\mA_1)}(\mS(\mA_1) \times_{S(\mA_1) \vee \ldots \vee S(\mA_n)}\mC,\mS(\mA_1) \times_{S(\mA_1) \vee \ldots \vee S(\mA_n)}\mD) 
\times_{\Fun^\cocart_{\{1\}}(\{1\} \times_{S(\mA_1) \vee \ldots \vee S(\mA_n)}\mC,\{1\} \times_{S(\mA_1) \vee \ldots \vee S(\mA_n)} \mD)} $$$$ \cdots \times_{\Fun^\cocart_{\{n-1\}}(\{n-1\} \times_{S(\mA_1) \vee \ldots \vee S(\mA_n)}\mC,\{n-1\} \times_{S(\mA_1) \vee \ldots \vee S(\mA_n)} \mD)} \Fun^\cocart_{S(\mA_n)}(\mS(\mA_n) \times_{S(\mA_1) \vee \ldots \vee S(\mA_n)}\mC,\mS(\mA_n) \times_{S(\mA_1) \vee \ldots \vee S(\mA_n)}\mD) $$
$$ \simeq \Fun^\cocart_{S(\mA_1) \vee \ldots \vee S(\mA_n)}(\mC,\mD).$$
The composition $\beta \circ \alpha$ is the identity.
\end{proof}

\begin{corollary}
Let $1 \leq \m \leq \infty.$
We assume that the $m-1$-categorical Grothendieck construction is an equivalence.
The presheaf $$ \mC \mapsto \iota_0(m\Cat^{\cocart}_{/\mC})$$
on the 1-category $\Theta((m-1)\Cat)$ satisfies the Segal condition.

\end{corollary}

\begin{lemma}\label{flatness}

Let $1 \leq \m \leq \infty.$
We assume that the $m-1$-categorical Grothendieck construction is an equivalence.

Let $\phi: \mC \to \mD$ be a cocartesian fibration between $m$-categories and $\mB \in m\Cat. $
The induced functor
$\phi^*: m\Cat_{/\mD} \to m\Cat_{/\mC} $ preserves the density colimit $$ \colim_{\mA \in \Theta((m-1)\Cat) \times_{m\Cat} m\Cat_{/\mB}} \mA \simeq \mB.$$

\end{lemma}

\begin{proof}

By \cite[Theorem 2.4.8.]{GepnerHeine2026} there is a localization
$$\mP(\Theta((m-1)\Cat)) \rightleftarrows {{\mP((m-1)\mathrm{-}\Cat)}\mathrm{-}\Cat} : \N$$
whose generating local equivalences are of the form
$$ \N(S(\mA_1)) \vee ... \vee \N(S(\mA_n)) \to N(S(\mA_1) \vee \ldots \vee S(\mA_n))$$ for $\mA_1,..., \mA_n \in (m-1)\Cat $.
In particular, $\Theta((m-1)\Cat)$ is dense in ${{\mP((m-1)\mathrm{-}\Cat)}\mathrm{-}\Cat}$
and so for every $\mB \in {{\mP((m-1)\mathrm{-}\Cat)}\mathrm{-}\Cat}$ the canonical map 
$$ \colim_{\mA \in \Theta((m-1)\Cat) \times_{{{\mP((m-1)\mathrm{-}\Cat)}\mathrm{-}\Cat}} {{\mP((m-1)\mathrm{-}\Cat)}\mathrm{-}\Cat}_{/\mB}} \mA \to \mB$$ is an equivalence.
In particular, for every $\mB \in m\Cat$
the canonical map $$ \colim_{\mA \in \theta((m-1)\Cat) \times_{m\Cat} m\Cat_{/\mB}} \mA \to \mB$$ is an equivalence and the latter 
colimit is preserved by the embedding $$m\Cat= {(m-1)\Cat \mathrm{-}\Cat} \subset {{\mP((m-1)\mathrm{-}\Cat)}\mathrm{-}\Cat}.$$
This implies that for every $\mB \in m\Cat_{/\mD}$
the colimit $ \colim_{\mA \in \Theta((m-1)\Cat) \times_{m\Cat} m\Cat_{/\mB}} \mA \to \mB$ is preserved by the embedding $$ m\Cat_{/\mD} \subset ({{\mP((m-1)\mathrm{-}\Cat)}\mathrm{-}\Cat})_{/\mD}.$$
Moreover there is a further localization
$$ \mP(\Theta((m-1)\Cat)) \rightleftarrows m\Cat= {(m-1)\Cat \mathrm{-}\Cat} $$
whose right adjoint is the restriction of $\N$.
It will suffice to prove that the induced functor
$$({\mP((m-1)\Cat)}\mathrm{-}\Cat)_{/\mD} \xrightarrow{\phi^*}({\mP((m-1)\Cat)}\mathrm{-}\Cat)_{/\mC} \to {m\Cat}_{/\mC} $$ preserves small colimits.
This is because the latter functor restricts on ${m\Cat}_{/\mD} $
to $\phi^* : {m\Cat}_{/\mD} \to {m\Cat}_{/\mC}$ since the embedding
${m\Cat} \subset {{\mP((m-1)\mathrm{-}\Cat)}\mathrm{-}\Cat}$ preserves small limits.

There is an induced localization
$$\mP(\Theta((m-1)\Cat))_{/\N(\mD)} \rightleftarrows ({{\mP((m-1)\mathrm{-}\Cat)}\mathrm{-}\Cat})_{/\mD}$$
whose generating local equivalences are of the form
$$ N(S(\mA_1)) \vee ... \vee N(S(\mA_n)) \to N(S(\mA_1) \vee \ldots \vee S(\mA_n))$$ for $\mA_1,..., \mA_n \in (m-1)\Cat $ and a map
$ N(S(\mA_1) \vee \ldots \vee S(\mA_n)) \to N(\mD).$
Moreover there is an induced localization
$$\mP(\Theta((m-1)\Cat))_{/\N(\mC)} \rightleftarrows m\Cat_{/\mC}$$
whose local equivalences are detected by the forgetful functor
to the localization $$L: \mP(\Theta((m-1)\Cat)) \rightleftarrows m\Cat.$$

The functor $$\N(\phi)^*: \mP(\Theta((m-1)\Cat))_{/\N(\mD)} \to \mP(\Theta((m-1)\Cat))_{/\N(\mC)} $$ preserves small colimits
and so admits a right adjoint.
We will prove that this functor sends generating local equivalences
for the localization $$\mP(\Theta((m-1)\Cat))_{/\N(\mD)} \rightleftarrows ({{\mP((m-1)\mathrm{-}\Cat)}\mathrm{-}\Cat})_{/\mD}$$ to local equivalences for the
localization $$\mP(\Theta((m-1)\Cat))_{/\N(\mC)} \rightleftarrows m\Cat_{/\mC}.$$
In this case $\N(\phi)^*$ descends to a left adjoint functor 
$({\mP((m-1)\Cat)}\mathrm{-}\Cat)_{/\mD} \to {m\Cat}_{/\mC} $, 
which factors as $$({\mP((m-1)\Cat)}\mathrm{-}\Cat)_{/\mD} \xrightarrow{\phi^*}({\mP((m-1)\Cat)}\mathrm{-}\Cat)_{/\mC} \to {m\Cat}_{/\mC} $$
since $N$ preserves limits as a right adjoint.

So we have to see that for any $\mA_1,..., \mA_n \in (m-1)\Cat $ and map
$ \N(S(\mA_1) \vee \ldots \vee S(\mA_n)) \to \N(\mD)$
the induced map
$$ \theta: \N(\phi^*(S(\mA_1))) \coprod_{\N(\phi^*(\{1\} ))} ... \coprod_{\N(\phi^*(\{n-1\}))} \N(\phi^*(S(\mA_n))) \simeq $$ 
$$\N(\phi)^*(\N(S(\mA_1))) \coprod_{\N(\phi)^*(\N(\{1\}))} ... \coprod_{\N(\phi)^*(\N(\{n-1\}))} \N(\phi)^*(\N(S(\mA_n))) \simeq $$$$ \N(\phi)^*(\N(S(\mA_1)) \vee ... \vee \N(S(\mA_n))) \to \N(\phi)^*(N(S(\mA_1) \vee \ldots \vee S(\mA_n)))$$$$ \simeq \N(\phi^*(S(\mA_1) \vee \ldots \vee S(\mA_n))) $$
is a local equivalence for the localization $\L: \mP(\Theta((m-1)\Cat)) \rightleftarrows m\Cat: \N$. We prove that $\theta$ is inverted by $L$.
Since $\N$ is fully faithful, $L$  sends $\theta$ to the following functor
$$ \phi^*(S(\mA_1)) \coprod_{\phi^*(\bD^0)} ... \coprod_{\phi^*(\bD^0)} \phi^*(S(\mA_n)) \to \phi^*(S(\mA_1) \vee \ldots \vee S(\mA_n)).$$
The functor $L(\theta)$ is a functor over $\mC$ but also a functor over $S(\mA_1) \vee \ldots \vee S(\mA_n).$

Let $$ \mM := \phi^*(S(\mA_1) \vee \ldots \vee S(\mA_n)) \to S(\mA_1) \vee \ldots \vee S(\mA_n). $$
Since $\phi: \mC \to \mD$ is a cocartesian fibration,
the pullback $\mM \to S(\mA_1) \vee \ldots \vee S(\mA_n)$
is a cocartesian fibration.
The functor $L(\theta)$ identifies with the canonical functor
$$ \kappa: S(\mA_1) \times_{S(\mA_1) \vee \ldots \vee S(\mA_n)} \mM \coprod_{\{1\} \times_{S(\mA_1) \vee \ldots \vee S(\mA_n)} \mM} ... \coprod_{\{n-1\} \times_{S(\mA_1) \vee \ldots \vee S(\mA_n)} \mM} S(\mA_n) \times_{S(\mA_1) \vee \ldots \vee S(\mA_n)} \mM \to \mM.$$ 
The latter is an equivalence by \cref{glue}.
\end{proof}

\begin{proposition}\label{suspe}
Let $1 \leq \m \leq \infty$ and $\mA$ an $m-1$-category, for which the
Grothendieck construction is an equivalence.
The Grothendieck construction of $ S(\mA) $ is an equivalence.

\end{proposition}

\begin{proof}

By \cref{corrbifib} the functor $m\Cat^\cocart_{/S(\mA)} \to m\Cat \times m\Cat$ taking the fibers over the two objects of $S(\mA)$
induces on underlying 1-categories a bifibration fibered in spaces.
Similarly, by \cref{corrbifib} the conservative functor $\Fun(S(\mA), m\Cat) \to m\Cat \times m\Cat$ evaluating at the two objects of $S(\mA)$
induces on underlying 1-categories a bifibration, which is fibered in spaces by conservativity.
By \cref{Grothendieck-char} it suffices to show that the Grothendieck construction
$$\int_{S(\mA)}: \Fun(S(\mA), \infty\scat) \to \infty\scat^\cocart_{/S(\mA)} $$
induces an equivalence on underlying 1-categories. The Grothendieck construction induces on underlying 1-categories a map of bifibrations 
fibered in spaces over the underlying 1-category of $m\Cat \times m\Cat$.
So it suffices to see that the Grothendieck construction induces an equivalence of spaces $$ \kappa: \{(\mC,\mD)\} \times_{m\Cat \times m\Cat} m\Cat^\cocart_{/S(\mA)} \simeq \Map_{\infty\scat}(\mA,\Fun(\mC,\mD)) $$ on the fibers over any $\mC,\mD \in m\Cat.$

By \cref{Grsuspo} there is a canonical equivalence of spaces
$$ \{(\mC,\mD)\} \times_{m\Cat \times m\Cat} m\Cat_{/S(\mA)} \simeq 
\Map_{\infty\scat}(\mC^\circ \times \mD, m\Cat_{/\mA}).$$

Let $ m\Cat'_{/S(\mA)} \subset m\Cat_{/S(\mA)}$ be the subcategory
of functors $\mB \to S(\mA)$ that induce on morphism $\infty$-categories a cartesian fibration and functors over $S(\mA)$ preserving all cartesian morphims of the cartesian fibrations on morphism $\infty$-categories.
The last equivalence restricts to an equivalence
$$ \{(\mC,\mD)\} \times_{m\Cat \times m\Cat} m\Cat'_{/S(\mA)} \simeq \Fun(\mC^\circ \times \mD, m\Cat^\cart_{/\mA}).$$
By assumption the Grothendieck construction $\Fun(\mA^\circ, m\Cat) \to m\Cat^\cart_{/\mA} $ is an equivalence.
We obtain an equivalence $$ \{(\mC,\mD)\} \times_{m\Cat \times m\Cat} m\Cat'_{/S(\mA)} \simeq $$$$ \Map_{\infty\scat}(\mC^\circ \times \mD, m\Cat^\cart_{/\mA}) \simeq $$$$\Map_{\infty\scat}(\mC^\circ \times \mD,\Fun(\mA^\circ, m\Cat)) \simeq $$$$ \Map_{\infty\scat}(\mA^\circ \times \mC^\circ \times \mD, m\Cat)) \simeq $$$$
\Map_{\infty\scat}(\mA^\circ, \Fun(\mC^\circ, \Fun(\mD, m\Cat))) \simeq $$$$\Map_{\infty\scat}(\mA, \Fun(\mC, \Fun(\mD, m\Cat)^\circ)).$$
This equivalence restricts to an equivalence $$ \{(\mC,\mD)\} \times_{m\Cat \times m\Cat} m\Cat^\cocart_{/S(\mA)} \simeq
\Map_{\infty\scat}(\mA, \Fun(\mC, \mD)),$$ 
which identifies with $\kappa.$
\end{proof}

\begin{corollary}\label{suspe2} Let $1 \leq \m \leq \infty.$
We assume that the $m-1$-categorical Grothendieck construction is an equivalence.
The $m$-categorical Grothendieck construction for any object of 
$\Theta((m-1)\Cat)$ is an equivalence.

\end{corollary}

\begin{proof}
In view of \cref{Grothendieck-char} this follows immediately from \cref{suspe} and \cref{Segal},
where we use that the category of categories having small colimits and functors preserving small colimits admits large limits preserved by the inclusion to the category of large categories.
\end{proof}

\begin{corollary}\label{repro}
Let $1 \leq \m \leq \infty.$
We assume that the $m-1$-categorical Grothendieck construction is an equivalence.
The presheaf $$\beta: \Theta((m-1)\Cat)^\circ \to \mS, \mC \mapsto \iota_0(m\Cat^\cocart_{/\mC})$$ is the $\Theta((m-1)\Cat)$-nerve of $\iota_m(m\Cat).$
  
\end{corollary}

\begin{proof}
By \cref{suspe} for every $X, Y \in m\Cat$ and $\mA \in (m-1)\widehat{\Cat} $ there is a canonical equivalence
$$\rho: \Mor_\beta(X,Y)(S(\mA)) = \{(X,Y)\}\times_{m\widehat{\Cat} \times m\widehat{\Cat}} \Map_{m\widehat{\Cat}}(S(\mA),\beta) \simeq $$$$ \{(X,Y)\}\times_{m\widehat{\Cat} \times m\widehat{\Cat}} m\widehat{\Cat}^\cocart_{/S(\mA)} \simeq $$
$$\{(X,Y)\}\times_{\infty\widehat{\Cat} \times \infty\widehat{\Cat}} \Map_{\infty\widehat{\Cat}}(S(\mA),m\Cat) \simeq $$$$\Map_{\infty\widehat{\Cat}}(\mA,(\Fun(X,Y))) 
\simeq $$$$ \Map_{m\widehat{\Cat}}(\mA,\iota_{m-1}(\Fun(X,Y))).$$
\cref{denseinherited} implies that $\beta$ is the $\Theta((m-1)\Cat)$-nerve of 
a $m$-category $m\Cat'$ whose space of objects is $\iota_0(m\Cat)$
and whose $\infty$-categories of morphisms $X \to Y$ are $\iota_{m-1}(\Fun(X,Y)).$
This also guarantees that $m\Cat'$ is tensored over $m\Cat$.
Since $m\Cat$ is the $m\Cat$-enriched category with tensors generated by the point, there is a unique functor $m\Cat \to m\Cat'$ preserving tensors and inducing the identity on spaces of objects. 
The latter functor induces the inverse of the equivalence 
$\rho$ on morphism $\infty$-categories and so is an equivalence.
\end{proof}

\subsection{The Grothendieck construction is an equivalence}

\begin{proposition}\label{Grothe}

Let $1 \leq \m \leq \infty.$
We assume that the $m-1$-categorical Grothendieck construction is an equivalence.
Let $\mC$ be a $m$-category.
The following canonical functor is an equivalence:
$$\iota_1(m\Cat^\cocart_{/\mC}) \to \lim_{\mB \in \theta((m-1)\Cat) \times_{m\Cat} m\Cat_{/\mC}} \iota_1(m\Cat^\cocart_{/\mB}).$$
    
\end{proposition}

\begin{proof}

We prove that the functor of the statement is essentially surjective and fully faithful.
Let $$\bj:  \theta((m-1)\Cat) \subset \iota_1((m-1)\Cat)$$ be the canonical embedding, and let $\rho: \iota_1(m\Cat)^\op \to \Cat$ be the functor of 1-categories that sends $\mC$ to $\iota_1(m\Cat^\cocart_{/\mC})$.	
The unit natural transformation $\rho \to \bj_*(\bj^*(\rho)) $ to the right Kan extension of the restriction
induces at any $\mC \in m\Cat$ the canonical functor of the statement
$$\kappa: \iota_1(m\Cat^\cocart_{/\mC}) \to \lim_{\mB \in \theta((m-1)\Cat) \times_{m\Cat} m\Cat_{/\mC}}(\iota_1(m\Cat^\cocart_{/\mB})).$$

The induced natural transformation $\iota_0 \circ \rho \to \iota_0 \circ \bj_*(\bj^*(\rho)) \simeq \bj_*(\iota_0 \circ  \rho \circ \bj) $ induces at any $\mC \in m\Cat$ the canonical map 
$$\iota_0(m\Cat^\cocart_{/\mC}) \to \lim_{\mB \in \theta((m-1)\Cat) \times_{m\Cat} m\Cat_{/\mC}}(\iota_0(m\Cat^\cocart_{/\mB})). $$

By \cref{repro} the restriction $\iota_0 \circ \rho \circ \bj: \theta((m-1)\Cat)^\op \to \mS$ is represented by $\iota_m(m\Cat)$.
We obtain a natural transformation $$\iota_0 \circ \rho \to \bj_*(\iota_0 \circ  \rho \circ \bj) \simeq \bj_*(\Map_{\m\widehat{\Cat}}(- , \iota_m(m\Cat))\circ \bj) \simeq  \Map_{\m\widehat{\Cat}}(- , \iota_m(m\Cat)), $$ where the last equivalence is by density of
$\theta((m-1)\Cat)$ in $m\Cat,$
that induces at any $\mC \in m\Cat$ the map
$$\iota_0(m\Cat^\cocart_{/\mC}) \to $$$$ \lim_{\mB \in \theta((m-1)\Cat) \times_{m\Cat} m\Cat_{/\mC}}(\iota_0(m\Cat^\cocart_{/\mB})) \simeq $$$$ \lim_{\mB \in \theta((m-1)\Cat) \times_{m\Cat} m\Cat_{/\mC}}\Map_{\m\widehat{\Cat}}(\mB,  \iota_m(m\Cat))\simeq $$
$$ \Map_{\m\widehat{\Cat}}(\colim_{\mB \in \theta((m-1)\Cat) \times_{m\Cat} m\Cat_{/\mC}}, m\Cat) \simeq $$$$ \Map_{\m\widehat{\Cat}}(\mC,  m\Cat) \simeq \Map_{\m\widehat{\Cat}}(\mC,  \iota_m(m\Cat)). $$ 
The latter map sends the cocartesian fibration 
$  \iota_m(m\Cat_{*//}) \to  \iota_m(m\Cat) $ for $\mC= \iota_m(m\Cat)$ to the identity.

The canonical functor $\Fun(\mC, m\Cat) \to m\Cat^\cocart_{/\mC}$
sending the identity to the universal cocartesian fibration $  \iota_m(m\Cat_{*//}) \to  \iota_m(m\Cat) $
induces on underlying spaces a map
$$\Map_{\m\widehat{\Cat}}(\mC, \iota_m(m\Cat)) \simeq \Map_{\infty\widehat{\Cat}}(\mC, m\Cat) \to \iota_0(m\Cat^\cocart_{/\mC}), $$
which is natural in $\mC \in {m\Cat}$,
and so gives a natural transformation $$\Map_{\m\widehat{\Cat}}(-, \iota_m(m\Cat)) \to \iota_0 \circ \rho .$$
By the Yoneda lemma the composition 
$$\Map_{\m\widehat{\Cat}}(-, \iota_m(m\Cat)) \to \iota_0 \circ \rho \to \Map_{\m\widehat{\Cat}}(-, \iota_m(m\Cat)) $$
is the identity since it sends the identity of $ \iota_m(m\Cat)$ to the identity.
Hence the functor 
$$\kappa: \iota_1(m\Cat^\cocart_{/\mC}) \to \lim_{\mB \in \theta((m-1)\Cat) \times_{m\Cat} m\Cat_{/\mC}}\iota_1(m\Cat^\cocart_{/\mB})$$ is essentially surjective.

We prove next that $\kappa$ is fully faithful.
The functor $\kappa$ induces on morphism categories between $\mD, \mE \in m\Cat^\cocart_{/\mC}$
the canonical functor:
$$ \alpha: \Fun^\cocart_\mC(\mD,\mE) \to \lim_{\mB \in \theta((m-1)\Cat) \times_{m\Cat} m\Cat_{/\mC}} \Fun^\cocart_\mB(\mB \times_\mC \mD, \mB \times_\mC \mE).$$
There is a canonical equivalence
$$\sigma: \lim_{\mB \in \theta((m-1)\Cat) \times_{m\Cat} m\Cat_{/\mC}} \Fun_\mB(\mB \times_\mC \mD, \mB \times_\mC \mE) \simeq $$
$$  \lim_{\mB \in \theta((m-1)\Cat) \times_{m\Cat} m\Cat_{/\mC}} \Fun_\mC(\mB \times_\mC \mD, \mE) \simeq \Fun_\mC(\colim_{\mB \in \theta((m-1)\Cat) \times_{m\Cat} m\Cat_{/\mC}} (\mB \times_\mC \mD), \mE)$$
$$ \simeq \Fun_\mC((\colim_{\mB \in \theta((m-1)\Cat) \times_{m\Cat} m\Cat_{/\mC}}\mB) \times_\mC \mD, \mE) \simeq \Fun_\mC(\mD, \mE)$$
where the last equivalence is by \cref{flatness}.
The composition $\sigma \circ \alpha $ is the canonical embedding
$$\Fun^\cocart_\mC(\mD,\mE) \subset \Fun_\mC(\mD,\mE).$$

Hence $\alpha$ is fully faithful. But $\alpha$ is also essentially surjective since any functor $\mD \to \mE$ over $\mB$
between cocartesian fibrations is a map of cocartesian fibrations if
all pullbacks to any suspension of an $m-1$-category $\mA$ are maps of
cocartesian fibrations over $S(\mA).$
In fact this is already satisfied if all pullbacks to any 
disk of dimension smaller $m+1$ are maps of cocartesian fibrations over that disk.
\end{proof}

\begin{theorem}\label{Groeq}
The $\infty$-categorical Grothendieck construction is an equivalence.

\end{theorem}

\begin{proof}
By \cref{Groind} it suffices to see that for every $1 \leq \m \leq \infty$ the $m$-categorical Grothendieck construction is an equivalence.
We prove the latter by induction on $1 \leq \m \leq \infty$.
The case $m=0$ is \cref{Gro0}. Let $1 \leq \m \leq \infty.$
We assume that the $m-1$-categorical Grothendieck construction is an equivalence. 
In view of \cref{Grothendieck-char} we have to see that for every $m$-category $\mC$ the category $\iota_1(\infty\scat^\cocart_{/\mC})$ admits small colimits and for every $X \in \mC$ the functor
$\iota_1(\infty\scat^\cocart_{/\mC}) \to \iota_1(\mC)$
taking the fiber over $X$ preserves small colimits.
This holds by \cref{suspe2} and \cref{Grothe}.
\end{proof}

In the following we deduce important corollaries of \cref{Groeq}.

\begin{corollary}\label{orientpres0}
Let $S$ be an $\infty$-category.
The $\infty$-category $\infty\scat^\cocart_{/S}$
is presentable.

\end{corollary}

\begin{proof}

We use \cref{Groeq} and that the $\infty$-category $\Fun(S,\infty\scat)$
is presentable by \cref{psinho} (3).
\end{proof}

\begin{corollary}\label{orientpres}
Let $S$ be an $\infty$-category.
The antioriented category $\infty\fcat^\cocart_{/S}$
is presentable.

\end{corollary}

\begin{proof}

The underlying category of the antioriented category $\infty\mathfrak{Cat}^\cocart_{/S}$
is $\infty\Cat^\cocart_{/S}$, which is presentable by \cref{orientpres0}.
By \cref{cotensol} the antioriented category $\infty\mathfrak{Cat}^\cocart_{/S}$ is cotensored.
For every $\infty$-category $K$ and cocartesian fibration
$X \to S$ the cotensor is $X_\lax^K \to S.$
By \cite[Lemma 3.69.]{heine2024higher} it suffices to prove that 
the functor $(-)^{K}_\lax: \infty\Cat_{/S}^\cocart \to  \infty\Cat_{/S}^\cocart $ preserves small limits and is accessible.

The functor $(-)^{K}_\lax: \infty\Cat_{/S} \to \infty\Cat_{/S} $
is right adjoint to the functor $K \boxtimes (-): \infty\Cat_{/S} \to \infty\Cat_{/S}$
and so preserves small limits and is accessible.
Thus the functor $(-)^{K}_\lax: \infty\Cat_{/S}^\cocart \to  \infty\Cat_{/S}^\cocart $ also preserves small limits and is accessible because the inclusion $\infty\Cat_{/S}^\cocart \subset \infty\Cat_{/S} $ preserves small limits and small filtered colimits.
\end{proof}

We fix the following notation:

\begin{notation}
Let $\mD$ be an $\infty$-category, $\mE \subset \mD$ a subcategory and $\phi: \mC \to \mD$ a cartesian fibration. 
Let $$ \Fun^{\mE}_\mD(\mD,\mC) \subset \Fun_\mD(\mD,\mC) $$
denote the full subcategory of sections of $\phi$ sending cells of $\mE$ to
$\phi$-(co)cartesian cells, and define
$$ \Fun^{\oplax, \mE}_\mD(\mD,\mC):= \Fun^{\oplax, \cocart}_\mE(\mE,\mE \times_\mD \mC) \times_{\Fun^\oplax_\mE(\mE,\mE \times_\mD \mC)} \Fun^\oplax_\mD(\mD,\mC) \subset \Fun^\oplax_\mD(\mD,\mC). $$
    
\end{notation}

\begin{theorem}\label{limsect} Let $(\mD,\mE)$ be a marked $\infty$-category and $\mC \to \mD$ a cartesian fibration classifying a functor $F: \mD^\circ \to \infty\scat$.

\begin{enumerate}[\normalfont(1)]\setlength{\itemsep}{-2pt}
\item There is a canonical equivalence
$$ \mE-\overset{\lax}{\lim}(F) \simeq \Fun^{\mE}_\mD(\mD,\mC).$$ 

\item There is a canonical equivalence
$$ \mE-\overset{\bar{\to}}{\lim}(F) \simeq \Fun^{\oplax,\mE}_\mD(\mD,\mC).$$ 

\end{enumerate}

\end{theorem}

\begin{proof}
By \cref{weigrot} the cartesian fibration 
$\Fun^\oplax(\bD^1,\mD)\to \mD$ evaluating at the source
classifies the functor $ \W^{\mD}_\lax: \mD^\circ \to \infty\scat$.

(1): For every $\infty$-category $\mA$ there is a canonical equivalence $$\Mor_{\Fun(\mD^\circ,\infty\scat)}(\W^{\mD}_\lax, \Fun(\mA,-) \circ F) \simeq$$$$ \Fun^\cart_{\Fun^\oplax(\{0\},\mD)}(\Fun^\oplax(\bD^1,\mD), \mC^\mA) \simeq $$$$ \Fun_\mD(\mD,\mC^\mA) \simeq $$$$\Fun_\mD(\mA \times \mD,\mC) \simeq $$$$\Fun(\mA, \Fun_\mD(\mD,\mC)).$$
The first equivalence is by \cref{Groeq},
the second equivalence is by \cref{envelo2}.
The last two equivalences are the canonical ones.
The resulting equivalence $$\Mor_{\Fun(\mD^\circ,\infty\Cat)}(\W^{\mD}_\lax, \Fun(\mA,-) \circ F) \simeq \Fun(\mA, \Fun_\mD(\mD,\mC)) $$
restricts to an equivalence $$\Fun(\mA, \mE-\overset{\lax}{\lim}(F)) \simeq \Mor_{\Fun(\mD^\circ,\infty\scat)}(\W^{\mD,\mE}_\lax, \Fun(\mA,-) \circ F) \simeq \Fun(\mA, \Fun^\mE_\mD(\mD,\mC)).$$

(2): For every $\infty$-category $\mA$ there is a canonical equivalence $$ \Mor_{{\Fun\boxtimes}(\mD^\circ,\infty\mathfrak{Cat})}(\W^{\mD}_\lax, \Fun^\lax(\mA,-) \circ F) \simeq $$$$ \Mor_{{\Fun}(\mD^\circ,\infty\scat)}(\W^{\mD}_\lax, \Fun^\lax(\mA,-) \circ F)\simeq $$
$$ \Fun^\cart_{\Fun^\oplax(\{0\},\mD)}(\Fun^\oplax(\bD^1,\mD), \mC_\lax^\mA) \simeq $$$$ \Fun_\mD(\mD,\mC^\mA_\lax).$$
The first equivalence is by \cref{Groeq} and the third equivalence is by \cref{envelo2}.
The resulting equivalence
$$ \Map_{{\Fun\boxtimes}(\mD^\circ,\infty\mathfrak{Cat})}(\W^{\mD}_\lax, \Fun^\lax(\mA,-) \circ F) \simeq $$$$ \Map_{\infty\scat_{/\mD}}(\mD,\mC^\mA_\lax) \simeq $$$$ \Map_{\infty\scat}(\mA, \Fun^\oplax_\mD(\mD,\mC)), $$
where the last equivalence is by \cref{univeqr}, 
restricts to an equivalence $$\Map_{\infty\Cat}(\mA, \mE-\overset{\bar{\to}}{\lim}(F)) \simeq $$$$ \Map_{{\Fun\boxtimes}(\mD^\circ,\infty\mathfrak{Cat})}(\W^{\mD,\mE}_\lax, \Fun^\lax(\mA,-) \circ F) \simeq $$$$ \Map_{\infty\Cat}(\mA, \Fun^{\oplax,\mE}_\mD(\mD,\mC)).$$
\end{proof}

\begin{theorem}\label{Graytenso}

Let $\mC$ be an $\infty$-category, $\mD$ an antioriented category  and $X \in \mD. $
The antioriented colimit of the constant functor $\mC \to \mD$
with value $X$ if it exists, is $\mC \ot X $.
    
\end{theorem}

\begin{proof}
For every $Z \in \mC$ there is a canonical equivalence
$$ \map_\mC(\overset{\bar{\to}}{\colim}(\underline{X}),Z) \simeq $$$$ \map_{\Fun\boxtimes(\mC^\op, \infty\fcat)}(W^{\mC}_\lax,\underline{\L\Mor_\mC(X,Z)}) \simeq $$
$$ \map_{{{\Fun}}(\mC^\op, \infty\scat)}(W^{\mC}_\lax,\underline{\L\Mor_\mC(X,Z)}) \simeq  $$$$ \map_{\infty\scat^\cart_{/\mC}}(\Fun^\oplax(\bD^1,\mC),\L\Mor_\mC(X,Z)\times \mC) \simeq $$$$ \map_{\infty\scat_{/ \mC}}(\mC,\L\Mor_\mC(X,Z) \times \mC) \simeq $$$$ \map_{\infty\scat}(\mC,\L\Mor_\mC(X,Z)).$$
The first equivalence characterizes the antioriented colimit,
the second equivalence is by \cref{presheaves},
the third equivalence is by \cref{Groeq} and the fourth equivalence is \cref{envelo2}.
\end{proof}

We also obtain the following result concerning limits:

\begin{proposition}

Let $X$ be an $\infty$-category and $\alpha: F \to G$ a natural transformation of functors $X \to \infty\scat,$ which is objectwise an inclusion.
The canonical commutative square of $\infty$-categories
$$ \begin{tikzcd}
\lim F \ar{r}{} \ar{d} & \lim G \ar{d}{} \\
\prod_{Z \in X} F(Z)  \ar{r}[swap]{} & \prod_{Z \in X} G(Z)
\end{tikzcd}$$
is a pullback square.

\end{proposition}

\begin{proof}

For every $\infty$-category $K$ applying the functor
$\Fun(K,-): \infty\scat \to \infty\scat$ to the commutative square of the statement gives the same commutative square, where
$\alpha: F \to G$ is replaced by the natural transformation $\Fun(K,-) \circ \alpha: \Fun(K,-) \circ F \to \Fun(K,-) \circ G.$
The latter is also objectwise an inclusion
since monomorphisms are preserved by any right adjoint functor. 
Hence replacing $\alpha$ by $\Fun(K,-) \circ \alpha$, it suffices to prove that the commutative square of the statement 
induces an pullback square on maximal subspaces. 

Since $\alpha: F \to G$ is objectwise an inclusion, both vertical functors in the commutative square are inclusions.
Hence the functor from $\lim F$ to the pullback is an inclusion, and so induces an embedding on maximal subspaces. 
Therefore it remains to see that every object in the pullback of the commutative square lies in $\lim F$.
Let $H \in \lim G$ such that for every $Z \in X $ the image in $G(Z)$ lies in $F(Z)$. We want to see that $H \in \lim F$.
Since $\alpha: F \to G$ is objectwise an inclusion, the induced map $\int F \to \int G$ of cocartesian fibrations over $S$ is an inclusion.
The induced functor $\lim F \to \lim G$ identifies with the functor
$$ \Fun^\cocart_X(X, \int F) \to  \Fun^\cocart_X(X, \int G).$$

So we have to see that every cocartesian section of $\int G \to S$
lands in $\int F \subset \int G$ if it sends objects of $X$ to objects of $\int F. $
If this is shown, the resulting section of $\int F \to S$ is automatically a cocartesian section since the inclusion
$\int F \subset \int G$ is a map of cocartesian fibrations over $X$ and cocartesian lifts are unique. 

We prove by induction on $n \geq 0$ that every cocartesian section of $\int G \to S$ sends $n$-morphisms of $X$ to $n$-morphisms of $\int F$ if it sends every object of $X$ to objects of $\int F.$
The induction start $n=0$ is tautological. 
Let $n \geq 1$. We assume the statement for $n-1.$
By induction hypothesis, the image $\sigma$ of an $n$-morphism of $X$ in $\int G $, which is a cocartesian $n$-morphism, gives rise via the left inclusion $\bD^{n-1} \subset \bD^n$ to an $n-1$-morphism of $\int F.$
The latter extends along the left inclusion $\bD^{n-1} \subset \bD^n$ to a cocartesian $n$-morphism $\tau$ of $\int F$
that lies over the image of $\sigma$ in $X$.
The image of $\tau$ in $\int G$ is a cocartesian $n$-morphism whose restriction along the left inclusion $\bD^{n-1} \subset \bD^n$ agrees with the restriction of $\sigma$ and which lies in $X$ over the same $n$-morphism as $\sigma.$
So by uniqueness of cocartesian lifts, we find that $\tau \simeq \sigma.$ Hence $\sigma$ is an $n$-morphism of $\int F.$
\end{proof}

\cref{Groeq} also implies the following:

\begin{corollary}Lax colimits in $\infty\scat$ commute with weakly contractible limits.
Precisely, for every small $\infty$-category $S$ the functor
$\overset{\lax}{\colim}: \Fun(S,\infty\scat) \to \infty\scat $
preserves small weakly contractible limits.
    
\end{corollary}

\begin{proof}

By \cref{laxgro} the functor $\overset{\lax}{\colim}: \Fun(S,\infty\scat) \to \infty\scat $ factors as $$\Fun(S,\infty\scat) \xrightarrow{\int_S} \infty\scat_{/S} \to \infty\scat.$$ The latter functor is the forgetful functor, which preserves weakly contractible limits.
The Grothendieck construction $\int_S: \Fun(S,\infty\scat) \xrightarrow{} \infty\scat_{/S}$ is an equivalence by \cref{Groeq}. So the result follows.
\end{proof}

\subsection{The Grothendieck construction is an oriented equivalence}

In the following we enhance the Grothendieck construction to an antioriented equivalence.

\begin{construction}

Let $X$ be an $\infty$-category.
By \cref{orientpres} the antioriented category $ \infty\mathfrak{Cat}_{/X}^\cocart $
is a presentable antioriented category.
Thus the antioriented functor
$$ X^\circ \xrightarrow{} \infty\scat_{/X}^\cocart  \subset \infty\mathfrak{Cat}_{/X}^\cocart, Z \mapsto X_{Z//^\oplax} $$
uniquely extends to an antioriented adjunction
$${\Fun\boxtimes}(X,\infty\mathfrak{Cat}) \rightleftarrows \infty\mathfrak{Cat}_{/X}^\cocart: \gamma. $$

\end{construction}

\begin{theorem}\label{Groori}

The antioriented adjunction
$${\Fun\boxtimes}(X,\infty\mathfrak{Cat}) \rightleftarrows \infty\mathfrak{Cat}_{/X}^\cocart: \gamma $$
is an antioriented localization.
The local objects are precisely the oriented functors
$X \to \infty \fcat$ that land in $\infty\scat.$
    
\end{theorem}

\begin{proof}
The antioriented right adjoint $\gamma$ factors as
antioriented functors
$$ \infty\mathfrak{Cat}_{/X}^\cocart \to {\Fun\boxtimes}((\infty\mathfrak{Cat}_{/X}^\cocart)^\circ,\infty\mathfrak{Cat}) 
\to {\Fun\boxtimes}((\infty\scat_{/X}^\cocart)^\circ,\infty\mathfrak{Cat}) \to {\Fun\boxtimes}(X,\infty\mathfrak{Cat}). $$

The canonical antioriented functor
$$ \infty\scat_{/X}^\cocart \to \infty\mathfrak{Cat}_{/X}^\cocart $$
gives rise to a commutative square of antioriented categories
\begin{equation}\label{sqqo}
\begin{xy}
\xymatrix{
\infty\scat_{/X}^\cocart \ar[d]^{} \ar[r]
& {\Fun}((\infty\scat_{/X}^\cocart)^\circ,\infty\scat) \ar[d]^{} \ar[r]
& {\Fun}(X,\infty\scat) \ar[d]
\\ 
\infty\mathfrak{Cat}_{/X}^\cocart \ar[r] & {\Fun\boxtimes}((\infty\scat_{/X}^\cocart)^\circ,\infty\mathfrak{Cat}) \ar[r] & {\Fun\boxtimes}(X,\infty\mathfrak{Cat})}
\end{xy}\end{equation}

The bottom antioriented functor of the outer commutative square is $\gamma.$
The top functor of the outer commutative square is the right adjoint of the Grothendieck construction
$ {\Fun}(X,\infty\scat) \to \infty\scat_{/X}^\cocart$ and so an equivalence by \cref{Groeq}.
The left vertical antioriented functor of the outer commutative square 
induces an equivalence on underlying categories.
The right vertical antioriented functor of the outer commutative square 
induces an embedding on underlying categories by \cref{presheaves}.
Hence the antioriented functor $\gamma$ induces an embedding on underlying categories.
The antioriented functor $\gamma$ is an antioriented right adjoint and so preserves left cotensors. Since source and target of $\gamma$ are antioriented categories that admit left cotensors, the antioriented functor $\gamma$ is an antioriented embedding since it induces an embedding on underlying categories.
The commutative square (\ref{sqqo}) implies the characterization of local objects.
\end{proof}

\begin{corollary}\label{Groorihom}

Let $X$ be an $\infty$-category and $F,G: X \to \infty\fcat$ be oriented functors that land in $\infty\scat.$

There is a canonical equivalence
$$ \L\Mor_{{\Fun\boxtimes}(X,\infty\mathfrak{Cat})}(F,G) \simeq \Fun^{\oplax, \cocart}_X(\int F, \int G)$$
    
\end{corollary}

\subsection{A classification of categorical principal bundles}

\begin{definition}

Let $G$ be a monoidal $\infty$-category and $S$ an $\infty$-category.
A principal $G$-bundle over $S$ is a left $G$-module $X \to S$ in $\infty\scat^\cocart_{/S}$ such that for every $s \in \rS$
the induced left $G$-module structure on the fiber $X_s$ is free on one generator. 

\end{definition}

\begin{notation}Let $G$ be a monoidal $\infty$-category.
Let $BG$ be the unique $\infty$-category whose space of objects is connected and whose morphism $\infty$-category between any object is $\G.$
    
\end{notation}

\begin{notation}

Let $G$ be a monoidal $\infty$-category and $S$ an $\infty$-category.
Let $$\LMod_G(\infty\scat^\cocart_{/S}):= \Fun(BG,\infty\scat^\cocart_{/S})$$ be the  $\infty$-category of left $G$-modules.
    
\end{notation}

\begin{notation}

Let $G$ be a monoidal $\infty$-category and $S$ an $\infty$-category.
Let $$\mathrm{Bun}_G(S) \subset \LMod_G(\infty\scat^\cocart_{/S})$$ be the full subcategory of $G$-principal bundles over $S$.
    
\end{notation}

\cref{Groeq} gives the following:

\begin{corollary}\label{grobun}
Let $G$ be a monoidal $\infty$-category and $S$ an $\infty$-category.
There is a canonical equivalence
$$\mathrm{Bun}_G(S) \simeq \Fun(S, BG). $$
    
\end{corollary}

\begin{proof}

The inverse of the Grothendieck construction is an equivalence
$$ \infty\scat^\cocart_{/S} \simeq \Fun(S, \infty\scat). $$
This equivalence induces an equivalence
$$ \LMod_G(\infty\scat^\cocart_{/S}) \simeq \LMod_G(\Fun(S, \infty\scat)) \simeq \Fun(S, \LMod_G(\infty\scat)). $$
Since $BG$ is the full subcategory of $\LMod_G(\infty\scat)$
spanned by the free left $G$-module on one generator, the latter equivalence restricts to an equivalence
$ \mathrm{Bun}_G(S) \simeq \Fun(S, BG). $
\end{proof}

\begin{remark}
For every space $X$ 
the $n$-th cohomology group of $X$ with coefficients in an abelian group $G$
is given by equivalence classes of maps $X \to B^n(G),$ which admits a geometric model by higher principal bundles.

By \cite{heine2025homology} there is an extension of homology and cohomology to $\infty$-categories.
For every $\infty$-category $X$ 
the $n$-th cohomology symmetric monoidal $\infty$-category of $X$ with coefficients in a symmetric monoidal $\infty$-category $G$
is given by the $\infty$-category of functors and lax transformations $X \to B^n(G).$
The latter admits a geometric model by \cref{grobun}.

\end{remark}

\subsection{Bicartesian fibrations classify higher adjunctions}

\begin{notation}Let $n \geq 1.$

\begin{enumerate}[\normalfont(1)]\setlength{\itemsep}{-2pt}

\item Let $\sigma_{n}: \infty\scat \to \infty\scat$ be the involution that reverses the morphisms of dimension $n.$

\item Let $\sigma_{\geq n}: \infty\scat \to \infty\scat$ be the involution that reverses the morphisms of dimension larger or equal $n.$

\end{enumerate}
    
\end{notation}

\begin{notation}Let $S$ be an $\infty$-category and $n \geq 1.$

Let $$\infty\scat^{\cocart, \bicart\geq n}_{/S} \subset \infty\scat^{\cocart}_{/S} $$
be the subcategory of cocartesian fibrations over $S$ 
that are cartesian fibrations in dimensions larger or equal $n$
and maps of cocartesian fibrations over $S$
that are maps of cartesian fibrations over $S$ in dimensions larger or equal $n$.

\end{notation}

\begin{theorem}\label{Grobica} Let $T$ be an $\infty$-category and $n \geq 1.$
The functor $$ \infty\widehat{\Cat}^\op \to \widehat{\mS}, \qquad S \mapsto \iota_0((\infty\scat^{\cocart, \bicart \geq n}_{/S})_{/S \times T}) $$ is represented by a subcategory $\infty\scat_{/T}^{\mathrm{L},\geq n} \subset \infty\scat_{/T} $.

\end{theorem}

\begin{proof}
By \cref{Groeq} the functor $ \infty\widehat{\Cat}^\op \to \infty\widehat{\Cat}$ given by $$ S \mapsto (\infty\scat^{\cocart}_{/S})_{/ S \times T} \simeq \Fun(S, \infty\scat)_{/T} \simeq \Fun(S, \infty\scat_{/T}) $$ preserves small limits.
It restricts to a functor $$ \infty\widehat{\Cat}^\op \to \infty\widehat{\Cat}, S \mapsto (\infty\scat^{\cocart, \bicart\geq n}_{/S})_{/ S \times T} $$ since cartesian fibrations are stable under base change.

Therefore the restriction $\Theta^\op \to \infty\widehat{\Cat}, \theta \mapsto (\infty\scat^{\cocart}_{/\theta})_{/ \theta \times T} $ preserves the canonical colimit decomposition of every object of $\Theta$ into disks.
In other words, for every $\theta \in \Theta$ and $\{\theta_i\}$ the canonical colimit decomposition of $\theta$ the induced functor
$$ (\infty\scat^{\cocart}_{/\theta})_{/ \theta \times T} \to \lim_i \infty\scat^{\cocart}_{/\theta_i})_{/ \theta_i \times T} $$
is an equivalence. By \cref{inverso} the inverse sends $\{X_i \to \theta_i \times T\}$
to $\colim_i X_i \to \colim_i \theta_i \times T \simeq \theta \times T.$

Hence by \cref{inverso}, the functor $\Theta^\op \to \infty\scat$ given by $S \mapsto (\infty\scat^{\cocart, \bicart\geq n}_{/S})_{/ S \times T} $ also preserves the canonical colimit decomposition of every object of $\Theta$ into disks.
This implies that the presheaf $$ \Theta^\op \to \widehat{\mS},\qquad \theta \mapsto \iota_0((\infty\scat^{\cocart, \bicart\geq n}_{/\theta})_{/ \theta \times T}) $$
is the nerve of an $\infty$-category 
$\infty\scat_{/T}^{\mathrm{L},\geq n}.$

The embedding
$\iota_0((\infty\scat^{\cocart, \bicart,\geq n}_{/\theta})_{/ \theta \times T}) \to \iota_0((\infty\scat^{\cocart}_{/\theta})_{/ \theta \times T}$ natural in $\theta \in \Theta$
represents an inclusion $(\infty\scat^{\mathrm{L},\geq n})_{/ T} \subset \infty\scat_{/ T}$.
The right Kan extension of the presheaf $\Theta^\op \to \widehat{\mS}$ is given by $ \theta \mapsto \iota_0((\infty\scat^{\cocart, \bicart\geq n}_{/\theta})_{/ \theta \times T} $, which is also equivalent to 
$$ \Map_{\infty\widehat{\Cat}}((-)_{\mid \Theta}, \infty\scat^{\mathrm{L},\geq n}_{/ T}): \Theta^\op \to \widehat{\mS}, $$ is the presheaf
$$\Map_{\infty\widehat{\Cat}}(-, \infty\scat^{\mathrm{L},\geq n}_{/ T}): \infty\scat^\op \to \widehat{\mS}.$$ 
This right Kan extension is also the presheaf $$\Theta^\op \to \widehat{\mS},\qquad S \mapsto \lim_{\theta \to S}\iota_0((\infty\scat^{\cocart, \bicart\geq n}_{/\theta})_{/ \theta \times T}).$$

For every $S \in \infty\scat$ the induced map
$\iota_0((\infty\scat^{\cocart, \bicart\geq n}_{/S})_{/ S \times T}) \to \lim_{\theta \to S}\iota_0((\infty\scat^{\cocart, \bicart\geq n}_{/\theta})_{/ \theta \times T}) $
is the restriction of the induced map 
$\iota_0((\infty\scat^{\cocart}_{/S})_{/ S \times T}) \to \lim_{\theta \to S}\iota_0((\infty\scat^{\cocart}_{/\theta})_{/ \theta \times T}),$ which is an equivalence by density of $\Theta \subset \infty\scat$
since the functor $\infty\widehat{\Cat}^\op \to \infty\widehat{\Cat}, S \mapsto (\infty\scat^{\cocart}_{/S})_{/ S \times T} $ preserves small limits.
Thus the map $$\iota_0((\infty\scat^{\cocart, \bicart\geq n}_{/S})_{/ S \times T}) \to \lim_{\theta \to S}\iota_0((\infty\scat^{\cocart, \bicart\geq n}_{/\theta})_{/ \theta \times T}) $$ is an embedding of spaces.
But this map is an equivalence since it is also essentially surjective.
Indeed, by \cref{thetalocal}, a map of cocartesian fibration over
$S$ over $S \times T$ belongs to $(\infty\scat^{\cocart, \bicart,\geq n}_{/S})_{/ S \times T}$
if for every functor $\theta \to S$ the base change to $\theta$
belongs to $(\infty\scat^{\cocart, \bicart,\geq n}_{/\theta})_{/ \theta \times T}.$
\end{proof}

\begin{notation}Let $T$ be an $\infty$-category, $ n \geq 1$ and $\sigma$ an involution of $\infty\scat.$
Let $$\infty\scat_{/T}^{\mathrm{L},\geq n, \sigma} \subset \infty\scat_{/T}$$ be the image of
the subcategory $$\sigma_!(\infty\scat^{\mathrm{L},\geq n}_{/T^\sigma}) \subset \sigma_!(\infty\scat_{/T^\sigma}) $$
under the canonical equivalence
$$ \sigma: \sigma_!(\infty\scat_{/T^\sigma}) \simeq \infty\scat_{/T}. $$
\end{notation}

\begin{corollary}\label{Grobica2}Let $T$ be an $\infty$-category, $ n \geq 1$ and $\sigma$ an involution of $\infty\scat.$
The functor $$ \infty\widehat{\Cat}^\op \to \widehat{\mS},\qquad S \mapsto \iota_0((\infty\scat_{/S}^{\sigma-\cocart, \bicart \geq n})_{/S \times T}) $$ is represented by the $\infty$-category $\sigma_!(\infty\scat^{\mathrm{L},\geq n, \sigma}_{/T})^\sigma $.

\end{corollary}

\begin{proof}Let $S$ be an $\infty$-category.
There is a canonical equivalence
$$ \Map_{\infty\scat}(S,\sigma_!(\infty\scat^{\mathrm{L},\geq n, \sigma}_{/T})^\sigma) \simeq \Map_{\infty\scat}(\sigma_!(S^\sigma),\infty\scat_{/T}^{\mathrm{L},\geq n, \sigma}) \simeq \Map_{\infty\scat}(\sigma_!(S^\sigma), \sigma_!(\infty\scat_{/T^\sigma}^{\mathrm{L},\geq n})) \simeq $$$$ \Map_{\infty\scat}(S^\sigma, \infty\scat_{/T^\sigma}^{\mathrm{L},\geq n}) \simeq \iota_0((\infty\scat_{/S^\sigma}^{\cocart, \bicart\geq n})_{/ S^\sigma \times T^\sigma}) \simeq \iota_0((\infty\scat_{/S}^{\sigma-\cocart, \bicart \geq n})_{/S \times T}),$$
where the last equivalence applies $\sigma.$
\end{proof}

\begin{corollary}\label{Grobica2}Let $T$ be an $\infty$-category, $ n \geq 1$ and $\sigma$ an involution of $\infty\scat$
which fixes the orientation of morphisms of dimension smaller $n.$
There is a canonical equivalence $$\infty\scat^{\mathrm{L},\geq n}_{/T} \simeq \sigma_!(\infty\scat^{\mathrm{L},\geq n, \sigma}_{/T})^\sigma .$$

\end{corollary}

\begin{corollary}\label{Grobica2}Let $T$ be an $\infty$-category and $ n \geq 1$.
There is a canonical equivalence $$\infty\scat^{\mathrm{L},\geq n}_{/T} \simeq (\infty\scat^{\mathrm{L},\geq n, \sigma_{\geq n}}_{/T})^{\sigma_n} .$$

\end{corollary}

\begin{corollary}\label{leftrightadj}
Let $T$ be an $\infty$-category.
There is a canonical equivalence
$$ \infty\scat_{/T}^{\mathrm{L}} \simeq (\infty\scat_{/T}^{\mathrm{R}})^\coop.$$

\end{corollary}

\subsection{Locally cocartesian fibrations and exponentiable fibrations}

The next definition is an $\infty$-categorical generalization of
the notion of exponentiable fibration of \cite[Definition 3.21.]{heine2026local} and \cite[Definition B.3.1.]{lurie.higheralgebra}:

\begin{definition}

A functor $\phi:T\to S$ is an exponentiable fibration if the associated basechange functor $$\phi^*:\infty\scat_{/S}\to\infty\scat_{/T}$$ preserves small colimits.
    
\end{definition}

\begin{remark}

The functor $$\phi^*:\infty\scat_{/S}\to\infty\scat_{/T}$$ preserves tensors and therefore admits a right adjoint if and only if the functor 
$$\phi^*:\infty\Cat_{/S}\to\infty\Cat_{/T}$$ on underlying categories
preserves small colimits.
    
\end{remark}

\begin{theorem}\label{expone}

Every cocartesian fibration is exponentiable.

\end{theorem}

\begin{proof}

Let $f:T\to S$ be a cocartesian fibration. We would like to see that the functor $f^*:\infty\Cat_{/S}\to\infty\Cat_{/T}$ preserves small colimits.
In view of \cref{expbase} we can assume that $f$ is the universal cocartesian fibration.
Let $K$ be a category and $H: K^{\triangleright} \to \infty\widehat{\Cat}_{/ \infty\scat}$ a functor.
Let $ p: X := \colim(H) \to \infty\scat.$
For every $Z \in K$ let $p_Z: H(Z) \to \infty\scat$
be the composition $H(Z) \to X \xrightarrow{p} \infty\scat.$

We check that the induced functor
$$ \colim(H(-)\times_{\infty\scat} \infty\scat_{*//^\oplax}) \to X \times_{\infty\scat} \infty\scat_{*//^\oplax} $$
is an equivalence.
For that we need to see that for every $\infty$-category $\mB$
the induced map
$$\Map_{\infty\widehat{\Cat}}(X \times_{\infty\scat} \infty\scat_{*//^\oplax}, \mB) \to \lim_{Z\in \K}\Map_{\infty\widehat{\Cat}}(H(Z)\times_{\infty\scat} \infty\scat_{*//^\oplax},\mB) $$
is an equivalence.
By \cref{laxgro} the latter map identifies with the canonical map
$$\Map_{\infty\widehat{\Cat}}(\colim^\lax(X), \mB) \to \lim_{Z\in K} \Map_{\infty\widehat{\Cat}}(\colim^\lax(H(Z)),\mB),$$
which by the universal property of the lax limit identifies with the canonical map
$$\Map_{\Fun(X^\op, \infty\widehat{\scat})}(\W_{X}^\lax, p^*(\mB)) \to \lim_{Z\in K} \Map_{\Fun(H(Z)^\op, \infty\widehat{\scat})}(\W_{H(Z)}^\lax, p_Z^*(\mB)).$$

By the $\infty$-categorical Grothendieck construction the latter map identifies with the map
$$\Map_{\infty\widehat{\Cat}^\cart_{/X}}(\Fun^\oplax(\bD^1,X), X\times_{\infty\scat}\infty\scat_{//^\oplax\mB}) \to $$$$ \lim_{Z\in K} \Map_{\infty\widehat{\Cat}^\cart_{/H(Z)}}(\Fun^\oplax(\bD^1,H(Z)), H(Z)\times_{\infty\scat}\infty\scat_{//^\oplax\mB}).$$
By the universal property of the free cartesian fibration (\cref{envelo2}) the latter map identifies with the map
$$\Map_{\infty\widehat{\Cat}_{/X}}(X, X\times_{\infty\scat}\infty\scat_{//^\oplax\mB}) \to \lim_{Z\in K} \Map_{\infty\widehat{\Cat}_{/H(Z)}}(H(Z), H(Z)\times_{\infty\scat}\infty\scat_{//^\oplax\mB}),$$
which identifies with the canonical equivalence
$$\Map_{\infty\widehat{\Cat}_{/\infty\scat}}(X, \infty\scat_{//^\oplax\mB}) \to \lim_{Z\in K} \Map_{\infty\widehat{\Cat}_{/\infty\scat}}(H(Z), \infty\scat_{//^\oplax\mB}). $$
\end{proof}

\begin{corollary}\label{inverso} Let $F: K \to \infty\scat$ be a functor and $S$ the colimit of $F.$
Let $ X=\{X(k)\}_{k\in K}$ be an object of the limit $\lim_{k \in K} \infty\scat^{\cocart}_{/F(k)} $,
corresponding to a natural transformation $X \to F$ of functors
$K \to \infty\scat$.
The induced functor $$\colim_k X(k) \to \colim_k F(k) \simeq S $$
is a cocartesian fibration whose image under the canonical functor
$ \infty\scat^{\cocart}_{/S} \to \lim_{k \in K}\infty\scat^{\cocart}_{/F(k)} $ is $X.$
\end{corollary}

\begin{proof}
By \cref{Groeq} the induced functor
$ \infty\scat^{\cocart}_{/S} \to \lim_{k \in K}\infty\scat^{\cocart}_{/F(k)} $
identifies with the canonical equivalence
$$\Fun(S,\infty\scat) \to \lim_{k \in K} \Fun(F(k), \infty\scat).$$
Let $X': S \to \infty\scat $ be the image of $X$ under the canonical equivalence
$$ \lim_{k \in K}\infty\scat^{\cocart}_{/F(k)} \simeq \lim_{k \in K} \Fun(F(k), \infty\scat) \simeq \Fun(S,\infty\scat). $$

By naturality of the Grothendieck construction, the Grothendieck construction of $X'$ is sent by the functor
$ \infty\scat^{\cocart}_{/S} \to \lim_{k \in K}\infty\scat^{\cocart}_{/F(k)} $ to $X.$
We prove that the the Grothendieck construction of $X'$ is
$\colim X \to \colim F \simeq S. $

The Grothendieck construction of $X'$ is by definition the pullback
$ S \times_{\infty\scat} \infty\scat_{*//} $ of the universal cocartesian fibration along $X': S \to \infty\scat.$
By \cref{expone} the universal cocartesian fibration is exponentiable. Hence there is a canonical equivalence
$$ \int_S X' = S \times_{\infty\scat} \infty\scat_{*//} = (\colim F) \times_{\infty\scat} \infty\scat_{*//} \simeq $$$$ \colim_{k \in K}(F(k) \times_{\infty\scat} \infty\scat_{*//}) \simeq \colim_{k \in K} X(k) .$$
\end{proof}

\begin{corollary}

Let $\theta = \colim_{i} \bD^{n_i} \in \Theta$ and $X \to \theta$ a locally cocartesian fibration. The functor $$ \colim_{i} (\bD^{n_i} \times_{\theta} X) \to \theta = \colim_{i} \bD^{n_i} $$ is a cocartesian fibration. 

\end{corollary}

\begin{proposition}\label{checkontheta}\label{expbase}
Let $p:X\to S$ be a functor.
The following are equivalent:

\begin{enumerate}[\normalfont(1)]\setlength{\itemsep}{-2pt}
\item The functor $p:X\to S$ is exponentiable.

\item For every functor $S' \to S$ the pullback $p':S' \times_S X\to S' $ of $p$ is exponentiable.

\item For every $\theta \in \Theta$ and functor $\theta\to S$ the pullback $ \theta \times_S X \to \theta$ is exponentiable.

\item For every $\theta \in \Theta$ and functor $\theta\to S$ the standard colimit decomposition $\theta=\colim_i\bD^{n_i}$ is preserved by basechange along $p$. 

\end{enumerate}
\end{proposition}

\begin{proof}

Condition (1) implies (2) since the functor $$\infty\Cat_{/S'}\xrightarrow{{p'}^*} \infty\Cat_{/S' \times_S X} \to \infty\Cat_{/X}$$ factors as
$$\infty\Cat_{/S'}\to \infty\scat_{/S} \xrightarrow{p^*} \infty\Cat_{/X}.$$
Trivially, (2) implies (1).
Since (1) implies (2), we find that (1) implies (3).
Trivially, (3) implies (4).

We prove that (4) implies (1).
Let $\N: \infty\Cat \to \mP(\Theta)$ denote the nerve functor.
Consider the pullback functor $\N(p)^*:\mP(\Theta)_{/\N(S)}\to \mP(\Theta)_{/\N(X)}$, which preserves small colimits.
We must show that $\N(p)^*$ preserves local equivalences. If this is shown, $\N(p)^*$ descends to a colimit preserving functor $$ \N(p)^*:\infty\Cat_{/S}\to\infty\Cat_{/X}.$$
The generating local equivalences in $\mP(\Theta)_{/\N(S)}$ are of the form
\[
\colim_i \N(\bD^{n_i})\to\N(\theta)
\]
where $\theta=\colim_i \bD^{n_i}$ is the standard colimit decomposition of an object $\theta\in\Theta_{/S}$.

Taking the pullback along $\N(p)$, which preserves colimits, we obtain the following map in $\mP(\Theta)_{/\N(X)}$:
\[
\colim_i \N(p^*(\bD^{n_i}))\simeq\colim_i \N(p)^*\N(\bD^{n_i})\to\N(p)^*(\N(\theta)) \simeq \N(p^*(\theta)) .
\]
We want to see that the latter map is a local equivalence, i.e. is inverted by the localization functor. Since the nerve $\N$ is fully faithful, the localization functor sends the latter morphism to the functor 
\[
\colim_i p^*(\bD^{n_i}) \to p^*(\theta)
\]
over $X.$
This functor is an equivalence by assumption.
\end{proof}

\begin{corollary}\label{standardexp}

Let $\theta \in \Theta.$
Every standard disk inclusion $\bD^n \to \theta $ is exponentiable.
    
\end{corollary}

\begin{proof}
We apply \cref{expbase}. We observe that for every $\theta'=\colim_i\bD^{n_i} \in \Theta$ and functor $\theta' \to \theta$ the canonical functor
$$\colim_i(\bD^{n_i} \times_{\theta} \bD^n) \to \theta' \times_{\theta} \bD^n $$
is an equivalence.   
\end{proof}

\begin{remark}
This fails for nonstandard inclusions of objects of $\Theta$.
Indeed, the standard example of a nonexponentiable fibration is the endpoint-preserving inclusion $\Delta^1\to\Delta^2\simeq\Delta^1\vee\Delta^1$.
\end{remark}

The functor $\Fun(\bD^1, \infty\scat) \to \infty\scat$ evaluating at the target is a 1-cartesian fibration and so in particular classifies a functor $\infty\Cat^\op \to \infty\widehat{\Cat}.$

\begin{lemma}\label{expful} Let $K$ be a category and $F: K \to \infty\Cat$ a functor. Let $Y \to \colim(F)$ be 
an exponentiable fibration and $Z \to \colim(F)$ a functor.
The induced functor
$$\nu: \infty\scat_{/\colim(F)} \to \lim_{k \in K} \infty\scat_{/F(k)} $$
induces an equivalence 
\begin{equation}\label{erfo}
\Fun_{\colim(F)}(Y, Z) \to \lim_{k \in K} \Fun_{F(k)}(F(k) \times_{\colim(F)} Y, F(k) \times_{\colim(F)} Z). \end{equation}

\end{lemma}

\begin{proof}

The induced functor (\ref{erfo}) identifies with the canonical equivalence
$$ \Fun_{\colim(F)}(Y, Z) \simeq \Fun_{\colim(F)}(\colim(F) \times_{\colim(F)} Y, Z) \simeq \Fun_{\colim(F)}(\colim_{k \in K}(F(k) \times_{\colim(F)} Y), Z) $$$$ \simeq \lim_{k \in K} \Fun_{\colim(F)}(F(k) \times_{\colim(F)} Y, Z) \simeq \lim_{k \in K} \Fun_{F(k)}(F(k) \times_{\colim(F)} Y, F(k) \times_{\colim(F)} Z).$$
\end{proof}

\begin{definition}

A functor $\phi: X \to Y$ is a locally cocartesian fibration if the pullback of $\phi$ along any functor $\bD^n \to Y$ for $n \geq 0$ is a 
cocartesian fibration.
    
\end{definition}

\begin{definition}Let $\phi: X \to Y, \phi': X' \to Y' $ be locally cocartesian fibrations.
A map of locally cocartesian fibrations $\phi \to \phi'$ is a  commutative square
$$\begin{xy}
\xymatrix{
X \ar[d]^{\phi} \ar[r]
& X' \ar[d]^{\phi'} 
\\ 
Y \ar[r] & Y'}
\end{xy}$$	
whose pullback along any functor $\bD^n \to Y$ for $n \geq 1$, 
is a map 
$$\begin{xy}
\xymatrix{
\bD^n \times_Y X \ar[rd] \ar[rr]
&& \bD^n \times_{Y'} X' \ar[ld]
\\ 
& \bD^n}
\end{xy}$$	
of cocartesian fibrations over $\bD^n$.

\end{definition}

\begin{notation}

Let $ \mathrm{Loc}\mathrm{Cocart} \subset \Fun(\bD^1, \infty\scat) $
be the subcategory of locally cocartesian fibrations and maps of such.

\end{notation}

\begin{notation}Let $S$ be an $\infty$-category.
Let $$\infty\scat^{\mathrm{loc}. \vspace{1mm} \cocart}_{/S} \subset \infty\scat_{/S}$$ be the subcategory of locally cocartesian fibrations over $S$ and maps of such.

\end{notation}

\begin{remark}

Note that $\mathrm{coCart} $ is a full subcategory of $ \mathrm{Loc}\mathrm{Cocart} $ and for every $\infty$-category $S$
that $\infty\scat^{\cocart}_{/S} $ is a full subcategory of $\infty\scat^{\mathrm{loc}. \vspace{1mm} \cocart}_{/S}.$
    
\end{remark}

\begin{notation}\label{reuft} Let $X \to S, Y \to S$ be locally cocartesian fibrations.
Let $$ \Fun^\cocart_S(X,Y) \subset \Fun_S(X,Y)$$ be the full subcategory of maps of locally cocartesian fibrations over $S$.

\end{notation}

\begin{remark}

\cref{reuft} is justified since maps of locally cocartesian fibrations between cocartesian fibrations are precisely maps of cocartesian fibrations.

\end{remark}

\begin{remark}

The functor $\Fun(\bD^1, \infty\scat) \to \infty\scat$ is a 1-cartesian fibration, which restricts to a 1-cartesian fibration
$ \mathrm{Loc}\mathrm{Cocart} \to \infty\scat $ since locally cocartesian fibrations and maps of such are stable under pullback.
The latter classifies in particular a functor $\infty\Cat^\op \to \infty\widehat{\Cat}.$

\end{remark}

\begin{theorem}\label{cocexp}

A functor is a cocartesian fibration if and only if it is  exponentiable and a locally cocartesian fibration.
    
\end{theorem}

\begin{proof}

Every cocartesian fibration is a locally cocartesian fibration
since cocartesian fibration are stable under pullback.
Every cocartesian fibration is exponentiable by \cref{expone}.

We prove the converse. Let $\phi: X \to Y$ be a locally cocartesian fibration that is exponentiable.
To see that $\phi$ is a cocartesian fibration, by \cref{thetalocal} it suffices to see that for every $\theta \in \Theta$ and functor $\theta \to Y$
the pullback $\theta \times_{Y} X \to \theta $ is a cocartesian fibration. Since $\phi$ is exponentiable, the functor
$\phi^*: \infty\Cat_{/Y} \to \infty\Cat_{/X}$ preserves small colimits
and so preserves the colimit decomposition $\colim_{i} \bD^i \simeq \theta.$ Hence the canonical functor 
$$ \colim_{i}(\bD^i \times_{Y} X) \to \theta \times_{Y} X $$ 
over $\theta$ is an equivalence.
So it suffices to show that the functor $$ \colim_{i}(\bD^i \times_{Y} X) \to \colim_{i}(\bD^i) = \theta$$ is a cocartesian fibration.

Since $\phi: X \to Y$ is a locally cocartesian fibration, the functor
$\bD^i \times_{Y} X \to \bD^i$ is a cocartesian fibration.
Hence by \cref{inverso} also the induced functor $$ \colim_{i}(\bD^i \times_{Y} X) \to \colim_{i}(\bD^i) = \theta$$ is a cocartesian fibration.
\end{proof}

\cref{inverso} implies the following:

\begin{proposition}

Let $\theta = \colim_{i} \bD^{n_i} \in \Theta$ and $X \to \theta$ a locally cocartesian fibration and $Y \to \theta$ an exponentiable fibration. The canonical functor $$ \colim_{i} (\bD^{n_i} \times_{\theta} X) \to X $$
induces an equivalence $$\Fun_{\theta}(Y,\colim_{i} (\bD^{n_i} \times_{\theta} X)) \to \Fun_{\theta}(Y,X),$$
which restricts to an equivalence $$ \Fun^\cocart_{\theta}(Y,\colim_{i} (\bD^{n_i} \times_{\theta} X)) \to \Fun^\cocart_{\theta}(Y,X). $$

\end{proposition}

\begin{proof}
The canonical functor $$ \colim_{i} (\bD^{n_i} \times_{\theta} Y) \to \colim_{i} (\bD^{n_i}) \times_{\theta} Y \simeq Y $$
is an equivalence since $Y$ is exponentiable.
Moreover the canonical functor $$ \bD^{n_i} \times_{\theta} X \to \bD^{n_i} \times_{\theta} \colim_{j} (\bD^{n_j} \times_{\theta} X) $$
is an equivalence since the standard disk inclusion $\bD^{n_i} \to \theta$ is exponentiable by \cref{standardexp}.

The induced functor $\Fun_{\theta}(Y,\colim_{i} (\bD^{n_i} \times_{\theta} X)) \to \Fun_{\theta}(Y,X)$ identifies with the identity
$$\lim_{i} \Fun_{\bD^{n_i}}(\bD^{n_i} \times_{\theta} Y,\bD^{n_i} \times_{\theta} X) \simeq \lim_{i} \Fun_{\bD^{n_i}}(\bD^{n_i} \times_{\theta} Y,\bD^{n_i} \times_{\theta} \colim_{i} (\bD^{n_i} \times_{\theta} X)) \simeq $$$$ \lim_{i} \Fun_{\theta}(\bD^{n_i} \times_{\theta} Y,\colim_{i} (\bD^{n_i} \times_{\theta} X)) \simeq $$$$ \Fun_{\theta}(\colim_{i} (\bD^{n_i} \times_{\theta} Y),\colim_{i} (\bD^{n_i} \times_{\theta} X)) \to \Fun_{\theta}(\colim_{i} (\bD^{n_i} \times_{\theta} Y),X) $$$$\simeq \lim_{i} \Fun_{\theta}(\bD^{n_i} \times_{\theta} Y,X) \simeq \lim_{i} \Fun_{\bD^{n_i}}(\bD^{n_i} \times_{\theta} Y,\bD^{n_i} \times_{\theta} X). $$
The induced functor $\Fun_{\theta}(Y,\colim_{i} (\bD^{n_i} \times_{\theta} X)) \to \Fun_{\theta}(Y,X)$
restricts to an equivalence $$ \Fun^\cocart_{\theta}(Y,\colim_{i} (\bD^{n_i} \times_{\theta} X)) \to \Fun^\cocart_{\theta}(Y,X)$$ since for any cocartesian fibrations $A \to \theta, B \to \theta,$ a functor $A \to B$ over $\theta$ is a map of cocartesian fibrations over $\theta$ if and only if its pullback along any standard disk inclusion
$\bD^n \to \theta$ is a map of cocartesian fibrations over $\bD^n.$
\end{proof}

\begin{corollary}\label{thetaadj} Let $\theta \in \Theta$.

\begin{enumerate}[\normalfont(1)]\setlength{\itemsep}{-2pt}
\item The embedding of the full subcategory of $\infty\scat_{/\theta}$ of cocartesian fibrations over $\theta$ into the full subcategory of $\infty\scat_{/\theta}$ of locally cocartesian fibrations over $\theta$
admits a right adjoint.

\item The embedding of the subcategory of $\infty\scat_{/\theta}$ of cocartesian fibrations over $\theta$ into the subcategory of $\infty\scat_{/\theta}$ of locally cocartesian fibrations over $\theta$
admits a right adjoint.

\end{enumerate}
    
\end{corollary}

The 1-cartesian fibration $\Fun(\bD^1, \infty\scat) \to \infty\scat$ evaluating at the target restricts to a 1-cartesian fibration
$ \mathrm{COCART} \to \infty\scat $ since cocartesian fibrations are stable under pullback, which in particular classifies a functor $\infty\Cat^\op \to \infty\widehat{\Cat}.$

\begin{corollary}

The presheaf $\infty\Cat^\op \to \infty\widehat{\Cat}$ classified by the 1-cartesian fibration $\mathrm{COCART} \to \infty\scat$ preserves small limits.
    
\end{corollary}

\begin{proof}

Let $K$ be a category and $F: K \to \infty\Cat $ a functor.
The induced functor
$$ \mathrm{coCart}_{\colim(F)} \to \lim_{k \in K} \mathrm{coCart}_{F(k)} $$ induces on maximal subspaces the
canonical map
$$ \iota_0(\infty\scat^\cocart_{/\colim(F)}) \to \lim_{k \in K}  \iota_0(\infty\scat^\cocart_{/F(k)}) $$
which is an equivalence since the Grothendieck construction is an equivalence by \cref{Groeq}.

Moreover the induced functor
$$ \mathrm{coCart}_{\colim(F)} \to \lim_{k \in K}  \mathrm{coCart}_{F(k)} $$ induces on morphism $\infty$-categories between
cocartesian fibrations $X \to \colim(F), Y \to \colim(F) $
the canonical functor
$$ \Fun_{\colim(F)}(X,Y) \to \lim_{k \in K} \Fun_{F(k)}(F(k) \times_{\colim(F)} X, F(k) \times_{\colim(F)} Y), $$
which is an equivalence by \cref{expful} since every cocartesian fibration is exponentiable by \cref{expone}.
\end{proof}

\begin{definition}
A double $\infty$-category is a Segal $\Theta$-object in $\infty\Cat,$ 
a functor $\Theta^\op \to \infty\Cat$ that satisfies the Segal condition,
i.e. which is local in the $(\infty\Cat, \times)$-enriched sense with 
respect to the map 
$$ \N(\bD^{i_0}) \coprod_{\N(\bD^{j_1})} \N(\bD^{i_1})\coprod_{\N(\bD^{j_2})} \cdots \N(\coprod_{\bD^{j_n}}) \N(\bD^{i_n}) \to \N(\bD^{i_0} \coprod_{\bD^{j_1}} \bD^{i_1} \coprod_{\bD^{j_2}} \cdots \coprod_{\bD^{j_n}} \bD^{i_n})$$
for any sequence of natural numbers $n, i_0,\ldots, i_n, j_1,\ldots, j_n$ and monomorphisms 
$ \bD^{j_\ell} \rightarrowtail \bD^{i_\ell},\bD^{j_\ell} \rightarrowtail \bD^{i_{\ell-1}}$, $ 1 \leq \ell \leq n$.

\end{definition}

\begin{corollary}
The presheaf $\Theta^\op \to \infty\widehat{\Cat}$ classified by the 1-cartesian fibration $\Theta \times_{\infty\scat} \mathrm{COCART} \to \Theta$ is a double $\infty$-category.

\end{corollary}

\begin{theorem}\label{expright} Let $S$ be an $\infty$-category.
The embedding $$\infty\scat^\cocart_{/S} \subset \infty\scat^{\mathrm{loc}. \vspace{1mm} \cocart}_{/S}$$
admits a right adjoint.

\end{theorem}

\begin{proof}

For every cartesian fibration of 1-categories $A \to B$ let
$A^\vee \to B^\circ$ be the cocartesian fibration of 1-categories which classifies the same functor.

By \cref{thetaadj} the embedding $$\Theta \times_{\infty\scat}  \mathrm{coCart} \subset \Theta \times_{\infty\scat}  \mathrm{Loc}\mathrm{Cocart} $$ of cartesian fibrations over $\Theta $
admits fiberwise a right adjoint.
Hence the induced embedding $$(\Theta \times_{\infty\scat}  \mathrm{coCart})^\vee \subset (\Theta \times_{\infty\scat}  \mathrm{Loc}\mathrm{Cocart})^\vee $$ of cocartesian fibrations over $\Theta^\circ $
admits fiberwise a right adjoint and so admits a right adjoint relative to $\Theta. $

The adjunction $$ (\Theta \times_{\infty\scat} \mathrm{coCart})^\vee \rightleftarrows (\Theta \times_{\infty\scat} \mathrm{Loc}\mathrm{Cocart})^\vee $$
relative to $\Theta^\circ$ induces for every cocartesian fibration $T \to \Theta^\circ $ an adjunction
$$ \Fun_{\Theta^\circ}(T, (\Theta \times_{\infty\scat} \mathrm{coCart})^\vee) \rightleftarrows \Fun_{\Theta^\circ}(T, (\Theta \times_{\infty\scat} \mathrm{Loc}\mathrm{Cocart})^\vee).$$

The cartesian fibration $\Theta \times_{\infty\scat} \mathrm{coCart} \to \Theta$ has presentable fibers and left adjoint fiber transports. This implies by \cite[Theorem 10.3.]{articles} that the embedding \begin{equation}\label{colozt}
\Fun^\cocart_{\Theta^\circ}(T, (\Theta \times_{\infty\scat} \mathrm{coCart})^\vee) \subset \Fun_{\Theta^\circ}(T, (\Theta \times_{\infty\scat} \mathrm{coCart})^\vee) \end{equation} 
preserves small colimits and starts at a presentable category and so 
admits a right adjoint.

Composing with the colocalization (\ref{colozt})
we obtain an adunction 
$$ \Fun^\cocart_{\Theta^\circ}(T, (\Theta \times_{\infty\scat} \mathrm{coCart})^\vee) \rightleftarrows \Fun_{\Theta^\circ}(T, (\Theta \times_{\infty\scat} \mathrm{coCart})^\vee) \rightleftarrows \Fun_{\Theta^\circ}(T, (\Theta \times_{\infty\scat} \mathrm{Loc}\mathrm{Cocart})^\vee),$$
which restricts to an adjunction
$$ \Fun^\cocart_{\Theta^\circ}(T, (\Theta \times_{\infty\scat} \mathrm{coCart})^\vee) \rightleftarrows \Fun^\cocart_{\Theta^\circ}(T, (\Theta \times_{\infty\scat} \mathrm{Loc}\mathrm{Cocart})^\vee).$$

Let $$ \nu: \mathrm{Loc} \mathrm{Cocart}_S \to \lim_{\theta \in \Theta_{/S}} \mathrm{Loc} \mathrm{Cocart}_\theta $$
be the canonical functor and $\nu'$ the composition $$ \mathrm{coCart}_S \subset \mathrm{Loc} \mathrm{Cocart}_S \xrightarrow{\nu} \lim_{\theta \in \Theta_{/S}} \mathrm{Loc} \mathrm{Cocart}_\theta.$$

Using \cref{envelo}, we obtain an adjunction
$$ \nu': \mathrm{coCart}_S \simeq \lim_{\theta \in \Theta_{/S}} \mathrm{coCart}_{\theta} \simeq $$$$ \lim_{\theta \in \Theta_{/S}} \Fun^\cocart_{\Theta^\circ}((\Theta^\circ)_{\theta/},(\Theta \times_{\infty\scat} \mathrm{coCart})^\vee) \simeq $$
$$ \Fun^\cocart_{\Theta^\circ}(\colim_{\theta \in \Theta_{/S}} (\Theta^\circ)_{\theta/},(\Theta \times_{\infty\scat} \mathrm{coCart})^\vee) \simeq $$
$$ \Fun^\cocart_{\Theta^\circ}(\Theta^\circ \times_{\infty\scat^\circ} (\infty\scat^\circ)_{S/},(\Theta \times_{\infty\scat} \mathrm{coCart})^\vee) \rightleftarrows $$$$ \Fun^\cocart_{\Theta^\circ}(\Theta^\circ \times_{\infty\scat^\circ} (\infty\scat^\circ)_{S/},(\Theta \times_{\infty\scat} \mathrm{Loc} \mathrm{Cocart})^\vee) $$
$$ \simeq \Fun^\cocart_{\Theta^\circ}((\colim_{\theta \in \Theta_{/S}} (\Theta^\circ)_{\theta/},(\Theta \times_{\infty\scat}  \mathrm{Loc} \mathrm{Cocart})^\vee) $$$$ \simeq \lim_{\theta \in \Theta_{/S}} \Fun^\cocart_{\Theta^\circ}((\Theta^\circ)_{\theta/},(\Theta \times_{\infty\scat} \mathrm{Loc} \mathrm{Cocart})^\vee) \simeq $$$$ \lim_{\theta \in \Theta_{/S}} \mathrm{Loc}\mathrm{Cocart}_\theta: R.$$

So by \cref{expful} for every cocartesian fibration $ Z \to S$ and locally cocartesian fibration $ Y \to S$ we obtain a natural equivalence
$$ \Map_{\mathrm{coCart}_S}(Z,R(\nu(Y))) \simeq \Map_{\lim_{\theta \in \Theta_{/S}} \mathrm{Loc} \mathrm{Cocart}_\theta}(\nu(Z),\nu(Y)) \simeq \Map_{\mathrm{Loc} \mathrm{Cocart}_S}(Z,Y).$$
\end{proof}

\begin{remark}

By \cref{expright} the embedding $$\infty\scat^\cocart_{/S} \subset \infty\scat^{\mathrm{loc}. \vspace{1mm} \cocart}_{/S}$$
admits a right adjoint $\R.$
Hence by \cref{envelo} for every locally cocartesian fibration $Y \to S$ and $s \in S$
there is a canonical equivalence
$$ R(Y)_s \simeq \Fun^\cocart_S(S_{s //^\oplax}, R(Y)) \simeq \Fun^\cocart_S(S_{s //^\oplax}, Y),$$
describing the fibers of $R(Y).$

\end{remark}

\bibliographystyle{plain}
\bibliography{mainbib}

\end{document}